\documentclass[12t,reqno,twoside]{amsart}
\usepackage{enumerate}
\usepackage{tikz}
\usetikzlibrary{arrows}
\usepackage{capt-of}
\usepackage{pgfplots}
\usepackage{tikz-3dplot}
\usepackage{caption}
\usetikzlibrary{calc}
\usepackage{multicol}
\usetikzlibrary{calc}
\tdplotsetmaincoords{60}{115}
\pgfplotsset{compat=newest}
\usepackage[english]{babel}
\usetikzlibrary{hobby}
\usepackage{hyperref}
\usepackage{amssymb,amsthm,amsmath,eucal,mathrsfs}
\usepackage{bm}
\usepackage{pgfplots}
\usepgfplotslibrary{fillbetween}
\pgfplotsset{compat=1.11}

\allowdisplaybreaks

\usepackage{amsfonts}
\usepackage{graphicx}
\usepackage[utf8]{inputenc}
\usepackage{lscape}
\usepackage{amstext}
\usepackage{xcolor}
\definecolor{mintbg}{rgb}{.63,.79,.95}
\usepackage{float}
\usepackage{cancel}
\usepackage{mathrsfs}
\usepackage{epsfig}
\usepackage{url}
\usepackage{pspicture}
\usepackage{wrapfig}
\usepackage{graphicx}
\usepackage{multirow}
\usepackage{lmodern}
\usepackage[T1]{fontenc}

\usepackage{cite}

\usepackage{dsfont}

\usepackage{hyperref}

\usepackage{cleveref}

\newtheorem{theorem}{Theorem}[section]

\newtheorem{lemma}[theorem]{Lemma}

\newtheorem{proposition}[theorem]{Proposition}

\newtheorem{corollary}[theorem]{Corollary}

\theoremstyle{definition}

\newtheorem{definition}[theorem]{Definition}

\newtheorem{remark}[theorem]{Remark}

\numberwithin{equation}{section}

\providecommand{\norm}[1]{\lVert#1\rVert} %Norm

\newcommand\restr[2]{{% we make the whole thing an ordinary symbol
  \left.\kern-\nulldelimiterspace % automatically resize the bar with \right
  #1 % the function
  \vphantom{\big|} % pretend it's a little taller at normal size
  \right|_{#2} % this is the delimiter
  }}

\usepackage{eucal}

\usepackage{enumitem}
\usepackage[font=it]{caption}
\usepackage{caption}
\usepackage{subcaption}
\setenumerate{leftmargin=*} 
\usepackage{algorithmic}
 \usepackage[foot]{amsaddr}
\usepackage{mathtools}
\newcounter{mysubequations}

\renewcommand{\themysubequations}{(\arabic{mysubequations})}

\newcommand{\mysubnumber}{\refstepcounter{mysubequations}\themysubequations}

\title[Foldy-Lax approximation of scattered field by many small inclusions]{elastic fields generated by multiple small inclusions\\ with high mass density at nearly resonant frequencies }

\address{ Corresponding Author: Divya Gangadaraiah}
\author{ Durga Prasad Challa$^{*}$, Divya Gangadaraiah$^{\dag}$,  Mourad Sini$^{\ddag}$ } 
\address{$^{*}$Department of Mathematics and Statistics, IIT Tirupati, India. }
\email{durga.challa@iittp.ac.in}

\address{$^{\dag}$Department of Mathematics and Statistics, IIT Tirupati, India.} 
\email{divya.kavyashree@gmail.com}
\address{$^{\ddag}$RICAM, Austrian Academy of Sciences, Altenberger Strasse 69, A-4040, Linz, Austria.}
\email{mourad.sini@oeaw.ac.at}
\address{The Author Mourad Sini is Partially supported by the Austrian Science Fund (FWF): P 30756-NBL and
P 32660.}
\begin{document}
\maketitle

\begin{abstract}
We derive the elastic field generated by multiple small-scaled inclusions distributed in a bounded set of $\mathbb{R}^3$. These inclusions are modeled with moderate values of the Lam\'e coefficients while they have a large relative mass density. These properties allow them to enjoy a sequences of resonant frequencies that can be computed via the eigenvalues of the volume integral operator having the Navier fundamental matrix as a kernel, i.e., the Navier volume operator. The dominant field, i.e., the Foldy-Lax field, models the multiple interactions between the inclusions with scattering coefficients that are inversely proportional to the difference between the used incident frequency and the already mentioned resonances. We show, in particular, that to reconstruct remotely the scattered field generated after $N$ interactions between the inclusions, one needs to use an incident frequency appropriately close to the proper resonance of the inclusions.  We provide an explicit link between the order $N$ of interactions and the distance from the incident frequency to the resonance. Finally, if the cluster of the inclusions is densely distributed in a given bounded domain, then the expression of the induced dominant field suggests that the equivalent homogenized mass density can change sign depending if the used incident frequencies is smaller or larger than a certain threshold (which is explicitly given in terms of the resonant frequencies of the inclusions).  

\bigskip 

 

\end{abstract}

\section{Introduction and statement of the results}
\subsection{Introduction}
\hfill
\par

Most of the  {materials} that we find in nature have moderate (not too high and not too small), positive and linear constitutive properties, see for instance \cite{Serdyukov-et-al, Maugin, Acoustics}. In the recent years, we witness a high and increasing interest in studying waves propagating in the presence of small-scaled heterogeneities enjoying very high contrasts, very small values or with evenly signed, of their constitutive properties. The motivations come primarily from the recent advance in micro, nano and quantum technologies, see \cite{Caloz}. The goal is to design new materials that can enjoy needed properties as enhancing and reorienting (or redirecting) incident fields. It is demonstrated that such properties have tremendous consequences in many fields {such as} material and imaging sciences with applications in many medical and industrial sectors, see for instance \cite{Caloz, Osipov-Tretyakov, Ntziachristos, Heating,MR4573428}.
\bigskip

At the mathematical level, such studies amount to understand how such extreme heterogeneities affect the waves propagating in their presence. Therefore, one can tackle these questions in the framework of perturbation theory applied to waves propagation. Such perturbations are very well understood in the case of moderate and positive signed  heterogeneities, see \cite{MRAMAG-book1, Ammari, Mazya-book, Jikov-book, LOW-FREQUENCY, C-S-2014, B-S-2021} for instance. The situation is less clear concerning extreme materials. Nevertheless, in the recent years, there have been several works devoted to such studies, see \cite{Ammari-Challa-Choudhury-Sini-1, Ammari-Challa-Choudhury-Sini-2, Ammari-F-G-L-Z, MR4510105, Cao-Ghandriche-Sini-2023,MR4362220,MR4003541,MR4126865,MR4083345}, for instance. 
\bigskip

To give a foretaste on the difference between wave propagation in the presence of moderate versus  {extremely} small heterogeneities, we restrict ourselves to the harmonic regime which is the regime of interest in this current work. Let us also keep in mind the three types of waves, namely acoustic, electromagnetic and elastic waves. For such models, the wave (vector) functions are solutions of an equivalent integral system of Lippmann-Schwinger type.
Such integral systems are linear combinations of two main operators to which are attached the different contrasts (between the heterogeneities and the background). One operator is modeled via the volume integral operator with  the fundamental solution (or tensor) as a kernel (think of the  {Newtonian operators} for the acoustic or electromagnetism). The second one is a surface type operator having as the normal derivative of the fundamental solution (think of the Neumann-Poincar\'e operator for the acoustic and electromagnetism \footnote{Or equivalently the Magnetization operator.}). The former models the body waves while the latter models the surface waves. 
For moderate materials, the attached contrast coefficients are moderate and have the good signs, therefore one can always invert such an integral system. However, if the heterogeneities are, small scaled but, highly contrasting or have an even sign, then this system of integral equations might have singularities. These are nothing but  consequences of the Fredholm alternative. The (singular) frequencies at which the system of integral equation is singular are completely characterized by the eigenvalues of the already mentioned volume and surface integral operators. Therefore, in the spirit of the Fredholm alternative, if the used incident frequency is close to one of these singular frequencies, then the induced total field is enhanced with a factor proportional to the distance between the used incident frequency and the indicated singular value. In addition, the dominant part of this total field is given as the incident field re-oriented through the related eigenfunctions  {of} the singular values. In the case of acoustics and electromagntics, such properties are analyzed and confirmed, see \cite{Ammari-Challa-Choudhury-Sini-1,Ammari-Challa-Choudhury-Sini-2,Ammari-F-G-L-Z,MR4510105,Cao-Ghandriche-Sini-2023,Ghandriche-Sini}. 
\bigskip

The object of the current work is to extend such ideas to the elastic case, namely the Lam\'e model where the inhomogeneities are modelled with small-scaled inclusions that enjoy relative high mass densities while keeping moderate Lam\'e parameters. To be precise, let $D_j\subset \mathbb{R}^3$, $j=1,2,\cdots, M$ be small scaled particles, of the form $D_j:=z_j+a\, B_j$, where $a>0$ is a characterisation parameter, $B_j$ are open, bounded, simply connected sets with Lipschitz boundary, containing the origin, and $z_j$ specifies the location of the particle $D_j$. The parameter $ a$ emphasises the smallness assumption on the particles. We set $D:=\cup_{j=1}^M {D_j}$. Let $\Omega$ be  {a} bounded and smooth domain in $\mathbb{R}^3$ such that $D\subset\subset \Omega$.

Let us consider the Lam\'e parameters $\lambda$ and $\mu$ as constants in $\mathbb{R}^3$, satisfying the conditions `$\mu>0$ and $3\lambda+2\mu>0$'. Additionally, the mass density denoted by $\rho(x)$ is assumed to have the following form:
\begin{equation}\label{def-rho-sm-case1-multiple-particles} \rho(x) :=\left\{ \begin{array}{c c c c} \rho_j \, \left(:=c_j a^{-2}\right), & \mbox{ in} & D_j,\,j=1,2,\cdots ,M \\ \rho_0, & \mbox{ in} & \mathbb{R}^3 \setminus \cup_{j=1}^M \bar{D_j} \end{array},\right.\end{equation}where, $c_j$ and the background mass density $\rho_0$ are positive constants, independent of $a$. Hence, $\rho_j$ denotes the mass density of the particle $D_j$ with the scaling $\rho_j=c_j a^{-2}$.\\

 We are interested in the following mathematical model describing the time-harmonic elastic wave scattering by collection of inclusions $D_j$'s, $j=1,2,\cdots, M$;
 \begin{eqnarray}
 \left\{\begin{array}{c c c}
 \Delta^eU^t+\omega^2\rho \,U^t=0& \hspace{-2.5cm}\mbox{in} \;\;\mathbb{R}^3\\
 U^t|_{-}=U^t|_{+} & \mbox{on} \;\; \partial D_j, \; j=1,\cdots,M\\ T_{\nu_j}U^t|_{+}=T_{\nu_j}U^t|_{-} & \mbox{on} \; \partial D_j, \; j=1,\cdots,M\end{array},\right.\label{general-model-R3-first-mp}\end{eqnarray} 
 where, $\omega >0$ is a given frequency, $T_{\nu_j}$ denotes the conormal derivative defined as  $T_{\nu_j}(\star): = {\lambda}(\nabla\cdot (\star))\nu_j+ {\mu}(\nabla (\star)+(\nabla (\star))^{\top})\nu_j$ with $\nu_j$ denoting the outward unit normal to $\partial D_j$, and  $\Delta^e(*)$ is the Lam\'e operator defined by $\Delta^e(\star):=(\lambda+\mu)\nabla(\nabla\cdot\star)+\mu \Delta(\star)$.

Here $U^t$ is the total field defined by $U^t:=U^i+U^s$ with $U^i$ denoting the incident field, and $U^s$ denoting the scattered elastic field by the $M$ small inclusions $D_j$ due to the incident field $U^i$. In our work, we  {consider} only plane incident waves of the form
$
 U^i({x},\theta):=b_1\theta e^{\mathbf{\mathtt{i}}\kappa_{\mathtt{p}}\theta\cdot x}+b_2\theta^{\perp}e^{\mathbf{\mathtt{i}}\kappa_{\mathtt{s}}\theta\cdot x}$,
where $\theta$ is the incident direction in $\mathbb{S}^2$ and $\theta^{\perp}$ is any direction in $\mathbb{S}^2$ perpendicular to $\theta$, $b_1,\,b_2$ are constants. We denote $\mathtt{p}$ and $\mathtt{s}$ parts of the incident waves  by $U_{\mathtt{p}}^i({x},\theta)$ and $U_{\mathtt{s}}^i({x},\theta)$ respectively, i.e., in our case $U_{\mathtt{p}}^i({x},\theta):=b_1\theta e^{\mathbf{\mathtt{i}}\kappa_{\mathtt{p}}\theta\cdot x}$ and $U_{\mathtt{s}}^i({x},\theta):=b_2\theta^{\perp}e^{\mathbf{\mathtt{i}}\kappa_{\mathtt{s}}\theta\cdot x}$. The incident field $U^i$ satisfies the background Navier equation $\Delta^e U^i+\omega^2\rho_{0} U^i = 0, \; \mbox{in} \; \mathbb{R}^3.$ 
The scattered field $U^s$ can be written using  Helmholtz decomposition\cite{LOW-FREQUENCY} as sum of its respective pressure and shear parts $U_{\mathtt{p}}^s(:=-\frac{\nabla(\nabla\cdot \,U^s)}{\kappa_{\mathtt{p}}^2} )$ and  $U_{\mathtt{s}}^s(:=\frac{\nabla\times(\nabla\times U^s)}{\kappa_{\mathtt{s}}^{2} })$.
 These pressure and shear parts  $U_{\mathtt{p}}^s$ and $U^s_{\mathtt{s}}$ individually satisfy the \textit{Sommerfeld radiation condition (S.R.C.)}, which together are called as \textit{Kupradaze radiation conditions({K.R.C.})} and are given below
\begin{eqnarray}\label{KRC-scattered-U-Multiple-bodies}
\lim_{|x|\rightarrow \infty}|x|\left(\frac{\partial U_{\mathtt{p}}^{s}(x)}{\partial |x|}-\mathbf{\mathtt{i}}\kappa_{\mathtt{p}}U_{\mathtt{p}}^{s}(x)\right)=0 \; \mbox{ and }\;   
\lim_{|x|\rightarrow \infty}|x|\left(\frac{\partial U_{\mathtt{s}}^{s}(x)}{\partial |x|}-\mathbf{\mathtt{i}}\kappa_{\mathtt{s}} U_{\mathtt{s}}^{s}(x)\right)=0, 
\end{eqnarray}
 where the two limits are uniform in all the directions $\hat{x}:=\frac{x}{|x|}$ on the unit sphere $\mathbb{S}^2$. Here $\kappa_{\mathtt{s}}$ and $\kappa_{\mathtt{p}}$ denote the transversal and longitudinal wave numbers respectively, i.e., $\kappa_{\mathtt{s}}:=\frac{\omega}{c_{\mathtt{s}}}$ and $\kappa_{\mathtt{p}}:=\frac{\omega}{c_{\mathtt{p}}}$, where $c_{\mathtt{s}}$ and $c_{\mathtt{p}}$  {denote} the transversal and longitudinal phase velocities respectively and  {are} given by  $c_{\mathtt{s}}^2:=\dfrac{\mu}{\rho_0}$ and $c_{\mathtt{p}}^2:=\dfrac{{\lambda+2\mu}}{\rho_0}$. 

\begin{definition}
 To describe the collection of small particles, we define:
 \begin{enumerate}
 \item   $M$ to denote the total number of inclusions, and it can be expressed in terms of `$ a$' satisfying $M:= M(a):=O(a^{-s})\leq {M_{\max}}\,a^{-s}$, where $s\geq 0$ and $M_{\max}$ is 
 a positive constant independent of '$a$'.
 
 \item $\epsilon$ as the maximum among the diameters, $diam$, of the small bodies $D_m$, i.e.
\begin{equation}\label{def-eps}
\epsilon := \max\limits_{1\leq m\leq M}
diam(D_m)\quad [= a\max\limits_{1\leq m\leq M}
diam(B_m)]
\end{equation}

\item $d$ as the minimum distance between the small bodies $D_1, D_2,\cdots, D_M$, assumed to be positive i.e., 
\begin{equation}\label{def-dis-d}
d:=\min\limits_{ 1\leq i,j\leq M,\, i\neq j } d_{ij} \quad [>0]
\end{equation}
 where $d_{ij}:=dist(D_i,D_j)$. We assume that $d$ can be expressed in terms of $a$ satisfying $d:=d(a)\sim a^t$, i.e., ${d}_{\min}a^t\leq d(a)\leq {d}_{\max}a^t$, where $t>0$, ${d}_{\min}$ and ${d}_{\max}$ are positive constants independent of $a$. 
 
 \item $\omega_{\max}$ as the upper bound of the used frequencies, i.e., $\omega\in[0, \omega_{\max}]$.
 
 \item $\Omega$ to be  {a} bounded and smooth domain in $\mathbb{R}^3$ compactly containing the small inclusions $D_m,\, m = 1, \cdots, M$.
\end{enumerate}
 \end{definition}

The elastic scattering problem \eqref{general-model-R3-first-mp} with \eqref{KRC-scattered-U-Multiple-bodies} is well posed in the  {H\"{o}lder} or Sobolev spaces (see \cite{LOW-FREQUENCY,COLTON-KRESS,kupradze1980three,kupradze1967potential} for instance).

By writing $U_{-}^t:=U^{t}|_{\cup_{j=1}^M D_j}$ and $U^t_{+}:=U^t|_{\mathbb{R}^3\setminus \overline{\cup_{j=1}^M D_j}}$, the mathematical model (\ref{general-model-R3-first-mp}--\ref{KRC-scattered-U-Multiple-bodies}) can be equivalently formulated as below: 
{\begin{eqnarray}
{\mu}\Delta U^t_{+}+({\lambda} + {\mu})\nabla \nabla\cdot U^t_{+}+\omega^2\rho_0 U^t_{+} = 0, \quad \mbox{in} \quad \mathbb{R}^3\setminus \overline{\cup_{j=1}^M D_j},\label{background-equation-Multiple-bodies} 
\\
{\mu}\Delta U^t_{-}+({\lambda} + {\mu})\nabla \nabla\cdot U^t_{-}+\omega^2\mathcal{\tilde{\rho}}\, U^t_{-}= 0, \quad \mbox{in} \quad \cup_{j=1}^M D_j\label{internal-equation-Multiple-bodies},
\\
\lim_{|x|\rightarrow \infty}|x|\left(\frac{\partial U_{\mathtt{p}}^{s}(x)}{\partial |x|}-\mathbf{\mathtt{i}}\kappa_{\mathtt{p}}U_{\mathtt{p}}^{s}(x)\right)=0, \; \mbox{and}\;
\lim_{|x|\rightarrow \infty}|x|\left(\frac{\partial U_{\mathtt{s}}^{s}(x)}{\partial |x|}-\mathbf{\mathtt{i}}\kappa_{\mathtt{s}} U_{\mathtt{s}}^{s}(x)\right)=0, \label{KRC-U-s-Multiple-bodies}
\end{eqnarray} 
\begin{eqnarray}\label{Transmission-SM-constant-simplemodel-Multiple-bodies}
U^t_{-}=U^t_{+} \quad \mbox{ and }  \quad& T_{\nu_j}U^t_{+}=T_{\nu_j}U^t_{-} \quad \mbox{on} \quad \partial D_j, \mbox{ for } j=1,\cdots,M,
\end{eqnarray}
where $\tilde{\rho}:=\rho_j$ on $D_j$, $j=1,\cdots,M$.

As the total field satisfies $U^t=U^i+U^s$, and $U^i$ satisfies background Navier-equation ${\mu}\Delta U^i+({\lambda} + {\mu})\nabla\nabla\cdot U^i+\omega^2\rho_{0} U^i=0$ in $\mathbb{R}^3$,  the scattered field $U^s$ satisfies, ${\mu}\Delta U^s+({\lambda} + {\mu})\nabla\nabla\cdot U^s+\omega^2\rho_{0} U^s = 0, \; \mbox{in} \; \mathbb{R}^3\setminus \overline{\cup_{j=1}^MD_j},$
satisfying the radiation conditions  \eqref{KRC-U-s-Multiple-bodies}, thus $U^s$ is  a radiating solution and has an asymptotic behaviour of the form \cite{alves2002far}
\begin{eqnarray}
U^s(x,\theta)=\dfrac{e^{\mathbf{\mathtt{i}}\kappa_{\mathtt{p}}|x|}}{|x|}U_{\mathtt{p}}^\infty(\hat{x},\theta)+\dfrac{e^{\mathbf{\mathtt{i}}\kappa_{\mathtt{s}}|x|}}{|x|}U_{\mathtt{s}}^\infty(\hat{x},\theta)+O\bigg(\dfrac{1}{|x|^2}\bigg),\; |x|\to \infty. \nonumber
\end{eqnarray}

 We denote the farfield corresponding to the scattered field $U^s(x,\theta)$ by $U^\infty(\hat{x},\theta)$ and is defined as the sum of $U_{\mathtt{p}}^\infty(\hat{x},\theta)$ and  $U^\infty_{\mathtt{s}}(\hat{x},\theta)$, P and S part of scattered farfield respectively corresponding to incident and propagation unit directions $\theta$ and $\hat{x}$. Also, we use $U_\mathtt{q}^{\infty,\mathtt{j}}(\hat{x},\theta)$, for $\mathtt{j},\mathtt{q}=\mathtt{p},\mathtt{s} $ to denote the $\mathtt{q}$ part of the farfield associated to $\mathtt{j}$ part of the incident wave. Similarly, we use $U_\mathtt{q}^{s,\mathtt{j}}(\hat{x},\theta)$, for $ \mathtt{j},\mathtt{q}=\mathtt{p},\mathtt{s} $ to denote the $\mathtt{q}$ part of the scattered field associated to $\mathtt{j}$ part of the incident wave.

\bigskip

 Let the fundamental matrix (also called as Kupradze matrix) of the Navier-equation associated to the frequency $\omega$ and mass density $\rho_0$, is denoted by $\Gamma^\omega(x,y)$, i.e., $\Gamma^\omega(x,y)$ satisfies the following problem 
\begin{equation}
 {\mu}\Delta \Gamma^\omega(x,y)+( {\lambda} +  {\mu})\nabla\nabla\cdot \Gamma^\omega(x,y)+\omega^2\rho_{0} \Gamma^\omega(x,y) = -\delta(x-y){I}, \quad \mbox{in} \quad \mathbb{R}^3.\nonumber
\end{equation}

The explicit form of the fundamental matrix $\Gamma^{\omega}(x,y)$ is given by  \cite{AMMARI-BIOMEDICAL-IMAGING}
\begin{eqnarray}
\Gamma^{\omega}(x,y)=\frac{1}{\mu}\phi_{\kappa_{\mathtt{s}}}(x,y)I+\frac{1}{\mu\kappa_{\mathtt{s}}^2}\nabla_x\nabla_x^\top\left[\phi_{\kappa_{\mathtt{s}}}(x,y)-\phi_{\kappa_{\mathtt{p}}}(x,y)\right],\label{fundamental-background-SM-multiple-particle}
\end{eqnarray}
where $\phi_{\kappa}(x,y):=\frac{e^{\mathbf{\mathtt{i}}\kappa|x-y|}}{4\pi|x-y|}$ denotes the fundamental solution of the Helmholtz equation in three dimensions associated to fixed wave number $\kappa$.

Thus the $kl^{th}$, $k,l=1,2,3$, entry of $\Gamma^{\omega}(x,y)$, denoted by $\Gamma^{\omega}_{kl}(x,y)$, is given by \cite{AMMARI-BIOMEDICAL-IMAGING},
\begin{equation}\label{entrywise-FM-mp-Appendix-use}
\Gamma^{\omega}_{kl}(x,y)=\frac{1}{4\pi \mu}\frac{e^{\mathbf{\mathtt{i}}\kappa_{\mathtt{s}}|x-y|}}{|x-y|}\delta_{kl}+\frac{1}{4\pi\omega^2\rho_0}\partial_k\partial_l \frac{e^{\mathbf{\mathtt{i}}\kappa_{\mathtt{s}}|x-y|}-e^{\mathbf{\mathtt{i}}\kappa_{\mathtt{p}}|x-y|}}{|x-y|},
\end{equation}
which further can be expressed in  series representation  as follows, see  \cite{LAYER-POTENTIAL-TECHNIQUES, AMMARI-MATHEMATICAL-METHODS-IN-ELASTIC-IMAGING};
\begin{eqnarray}\label{entrywise-FM-mp}
\Gamma^{\omega}_{kl}(x,y)&=&\frac{1}{4\pi\rho_0}\sum_{n=0}^{\infty}\frac{\mathbf{\mathtt{i}}^n}{(n+2)n!}\left( \frac{n+1}{c_{\mathtt{s}}^{n+2}}+\frac{1}{c_{\mathtt{p}}^{n+2}}\right)\omega^n\delta_{kl}|x-y|^{n-1}\nonumber\\
&&- \frac{1}{4\pi\rho_0}\sum_{n=0}^{\infty}\frac{\mathbf{\mathtt{i}}^n(n-1)}{(n+2)n!}\left(\frac{1}{c_{\mathtt{s}}^{n+2}}-\frac{1}{c_{\mathtt{p}}^{n+2}}\right)\omega^n|x-y|^{n-3}(x-y)_k(x-y)_l,
\end{eqnarray}
where $\delta_{kl}$ denotes the Kronecker symbol.
In addition, observe that the fundamental matrix $\Gamma^0$ associated to zero frequency, also known as Kelvin matrix, is symmetric, i.e., $\Gamma^0(x,y)={\Gamma^0(x,y)}^{\top}$ and $\Gamma^0(x,y)=\Gamma^0(y,x)$, and its entries have following simplified form;
\begin{equation}\label{entrywise-FM-zerof-mp}
\left\{\begin{array}{c c c}
\Gamma^0_{kl}(x,y)=\dfrac{ {\gamma_1}}{4\pi}\dfrac{\delta_{kl}}{|x-y|}+\dfrac{ {\gamma_2}}{4\pi}\dfrac{(x-y)_k(x-y)_l}{ |x-y|^3},
\\
\\
\mbox{where, } {\gamma_1}=\frac{1}{2}\left(\frac{1}{\mu}+\frac{1}{2\mu+\lambda}\right) \quad \mbox{and}\quad   {\gamma_2}=\frac{1}{2}\left(\frac{1}{\mu}-\frac{1}{2\mu+\lambda}\right).
\end{array}\right.
\end{equation}
The asymptotic expansion of $\Gamma^\omega$ is given by  
\begin{eqnarray}
\Gamma^\omega({x},y)&=& {\dfrac{1}{4\pi(\lambda+2\mu)}\hat{x}\,\hat{x}^{\top}\dfrac{e^{\mathbf{\mathtt{i}}\kappa_{\mathtt{p}}|x|}}{|x|} e^{-\mathbf{\mathtt{i}}\kappa_{\mathtt{p}}\hat{x}\cdot y}+\dfrac{1}{4\pi\mu}(I-\hat{x}\,\hat{x}^{\top})\dfrac{e^{\mathbf{\mathtt{i}}\kappa_{\mathtt{s}}|x|}}{|x|}e^{-\mathbf{\mathtt{i}}\kappa_{\mathtt{s}}\hat{x}\cdot y}}+O(|x|^{-2}),\;\; |x|\to \infty.\quad\label{assymptotic-expansion-Gamma-omega}
\end{eqnarray}
The farfield associated to $\Gamma^\omega$ is denoted by $\Gamma^\infty$ and is defined as sum of $\Gamma_\mathtt{p}^\infty$ and $\Gamma_\mathtt{s}^\infty$, which are defined below;
\begin{eqnarray}\label{p,s-parts-of farfield-gamma}
{\Gamma^\infty_{\mathtt{p}}(\hat{x},y)\,:=\, \dfrac{1}{4\pi(\lambda+2\mu)}\hat{x}\,\hat{x}^{\top}e^{-\mathbf{\mathtt{i}}\kappa_{\mathtt{p}}\hat{x}\cdot y}}\mbox{ and } {\Gamma^\infty_{\mathtt{s}}(\hat{x},y)\,:=\,\dfrac{1}{4\pi\mu}(I-\hat{x}\,\hat{x}^{\top})e^{-\mathbf{\mathtt{i}}\kappa_{\mathtt{s}}\hat{x}\cdot y}}.
\end{eqnarray}

\subsection{Statement of the results}\label{section-main-results}
%\hfill\\
\subsubsection{The Foldy-Lax approximation}
\hfill\\
The first main result of this paper is the following theorem, which provides approximate estimations of the scattered and farfields in terms of the parameter $a$ in the regimes described above.

\begin{theorem}\label{theorem-mp}
Let the Lam\'e parameters $\lambda$ and $\mu$ are constants everywhere and assume that the mass density $\rho$ satisfies \eqref{def-rho-sm-case1-multiple-particles}, i.e., `it is a positive constant $\rho_0$ outside the small inclusions $D_j$ and is $\rho_j:=c_ja^{-2}$ in each $D_j$, where $c_j>0$ are constants independent of $a$ for $j=1,\cdots,M$'. 
Then the elastic scattering problem (\ref{background-equation-Multiple-bodies}--\ref{Transmission-SM-constant-simplemodel-Multiple-bodies}) has a unique solution, and the corresponding scattered and farfields respectively satisfies the following asymptotic expansions, uniformly for $x$ in a bounded domain in the exterior of $\Omega$ and  for all the directions  $\theta$ and $\hat{x}$ in $\mathbb{S}^2$;
\begin{align}
\setcounter{mysubequations}{0}\mysubnumber \hspace{0.25cm}U^s(x,\theta)=&\sum_{j=1}^M \alpha_j\omega^2\,\Gamma^{\omega}(x,z_j)\cdot {Q}_j+O\bigg(a^{{1-s+\min\{1-h,\min\{0,1-2h-\frac{s}{2}\}\}}}\bigg),\label{sct-field-Main-Thm}
\\
\mysubnumber \hspace{0.1cm}U^\infty(\hat{x},\theta)\cdot(&(\lambda+2\mu)\beta_1\hat{x}+\mu\beta_2\hat{x}^\perp)=\sum_{j=1}^M \frac{\omega^2}{4\pi}\alpha_j\,U^i(z_j, -\hat{x})\cdot{Q}_j+O\bigg(a^{{1-s+\min\{1-h,\min\{0,1-2h-\frac{s}{2}\}\}}}\bigg).\; \label{farfield-Main-Thm}
\end{align}
In addition, the $p$ and $s$ parts of the farfields satisfy the following asymptotic expansions
{
\begin{eqnarray}
 U^\infty_{\mathtt{p}}(\hat{x},\theta) &=&\frac{1}{4\pi(\lambda+2\mu)}\hat{x}\hat{x}^\top\sum_{j=1}^M\alpha_j \omega^2 
e^{-\mathbf{\mathtt{i}}\kappa_{\mathtt{p}}\hat{x}\cdot z_j} Q_{j}+O\left( a^{1-s+\min\{0,1-2h-\frac{s}{2}\}}\right)\label{p-sct-field-Main-Thm-up-infty}
\\ 
U^\infty_{\mathtt{s}}(\hat{x},\theta)&=&
\frac{1}{4\pi\mu}(I-\hat{x}\hat{x}^\top)\sum_{j=1}^M\alpha_j\omega^2 e^{-\mathbf{\mathtt{i}}\kappa_\mathtt{s}\hat{x}\cdot z_j} Q_{j}+O\left(   a^{1-s+\min\{0,1-2h-\frac{s}{2}\}}\right) {,}\label{s-sct-field-Main-Thm-us-infty}
\end{eqnarray}
}

for all the frequencies $\omega$ satisfying
\begin{eqnarray}\label{def-omega-choosen-mp-1} 
\vert\omega^2-\omega_{n_{0_{(j)}}}^2\vert\simeq a^{h_j}, \mbox{ with } 0<h_j<1 \mbox{ and } \omega_{n_{0_{(j)}}}^2:=\dfrac{1}{\rho_j\lambda_{n_0}^{j}} \mbox{ for } j=1,\cdots,M,
\end{eqnarray}
$\frac{1}{2}\max\{\kappa_{\mathtt{s}},\kappa_{\mathtt{p}}\}\,diam(\cup_{m=1}^M D_m)<1$ and for $h:=\max\limits_{1\leq j\leq M}{h_j}$ satisfying {$\frac{s}{2} \leq h\leq 1-s$}, $t\leq \frac{1}{2}$, where $s=3t$. Here, 
\begin{equation}\label{def-alpham}
\alpha_j:=\rho_j-\rho_0,\;\; j=1,2,\cdots,M \mbox{ which behaves as } a^{-2} \mbox{ due to }\eqref{def-rho-sm-case1-multiple-particles},
\end{equation}
and $\lambda_n^j$ denotes the eigen values associated to  complete orthonormal eigen system $(\lambda_n^j,e_n^j)_{n\in\mathbb{N}}$ of  the  {Navier} operator $N_j^0:(L^2(D_j))^3\to (L^2(D_j))^3$ defined by $N_j^0(U)(x):=\int_{ {D_j}} \Gamma^0(x,y)\cdot U(y)dy$, which is compact and self-adjoint.
The vector coefficients $Q_{j}$, $j=1,\cdots,M$ are the solutions of following algebraic system:
\begin{eqnarray}
 {Q}_{j}=C^{(j)}\,\cdot U^i(z_j,\theta)
+C^{(j)}\cdot\sum_{\substack{m=1 \\ m\neq j}}^M\,\alpha_m\omega^2\Gamma^\omega(z_j,z_m)\cdot  Q_{m}
\label{Algebraic-system-thrm}
\end{eqnarray}
for $j=1,\cdots,M$, with 
\begin{align}\label{Polarization-tensor}
C^{(j)}:= {\dfrac{1}{(1-\alpha_j\,\omega^2 \lambda_{n_0}^j)}\sum_{l=1}^{l_{\lambda^j_{n_0}}}\overline{\int_{D_j}\hspace{-0.2cm}{e_{{n_0},l}^j(x)dx}}\otimes \int_{D_j}\hspace{-0.2cm}{e_{{n_0},l}^j(x)dx}}
\end{align}
{with $e_{{n_0}_l}^j$, for $j$ and $n_0$  {fixed}, are the eventual $l_{\lambda^j_{n_0}}$-multiple linearly independent eigenfunctions corresponding to the eigenvalue $\lambda_{n_0}^j$.}
\end{theorem}

Observe that $\omega_{n_{0_{(j)}}}\sim 1$ as $a \ll 1$, because $\rho_j = c_j a^{-2}$ and $\lambda_{n_0}^{j}=a^2 \tilde{{\lambda}}_{n_0}^j$,  {see \cite{challa2023extraction}}, where $(\tilde{\lambda}_n^j,\tilde{e}_n^j)_{n\in\mathbb{N}}$  denotes the complete orthonormal eigen system  of  the Navier operator $\tilde{N}_j^0:(L^2(B_j))^3\to (L^2(B_j))^3$ defined by $\tilde{N}_j^0(U)(\zeta):=\int_{B_j} \Gamma^0(\zeta,\eta) \cdot U(\eta)d\eta$ . 
We call the frequencies $\omega_{n_{0_j}}, n \in \mathbb{N},$ the Navier resonances of $B_j$.

\subsubsection{Accuracy of the Foldy-Lax approximation and the resonant character of the inclusions}
\hfill\\
In Theorem \ref{theorem-mp}, we have the estimation of the dominant elastic fields { {via \eqref{sct-field-Main-Thm} and \eqref{Algebraic-system-thrm}}}. { {Observe that \eqref{Algebraic-system-thrm} can be inverted using the Born series where the first term is given by $C^{(j)}\,\cdot U^i(z_j,\theta), j:=1, ..., M$}}. { {In the Born approximation, be definition, the coefficient $Q_j$ in the right hand side of \eqref{sct-field-Main-Thm} is replaced by $C^{(j)}\,\cdot U^i(z_j,\theta)$}}.  This implies that the
Born approximation entirely neglects the effect of multiple scattering and deals with weak scattering.
On the other hand, in Foldy's model, { {where the whole Born series expansions is taken into account}}, the full multiple interactions between the inclusions are considered.  Equation \eqref{sct-field-Main-Thm} describes the Foldy approximation, with the $Q_{j}$ vector terms being calculated from the Foldy algebraic system \eqref{Algebraic-system-thrm}. 

As a consequence of this theorem, we state Corollary \ref{corollary-N-interactions}, which describes the intermediate levels of scattering between the Born and Foldy models.

%%%%%%%%%%%%%%%%%%%%%%%%%%%%%%%%%%%%%%%%%%%%%%%%%%%%%%%%%%%%%%%%%%%%%%
\begin{corollary}\label{corollary-N-interactions}
For any non-negative integer $N$, the scattered field after $N-$ interactions between the inclusions is,
\begin{align}
U^{s,N}(x):=\sum_{j=1}^M \alpha_j\omega^2\Gamma^\omega(x,z_j)\cdot \bar{Q}^N_{j}, \label{scattered-field-for-truncated-solution*}
\end{align}
where, $\bar{Q}^N_{j}$ are vectors of size $3\times 1$ defined as 
$\bar{Q}^N_{j}:=\left[ {\bar{Q}^N({3j-2}),\bar{Q}^N({3j-1}),\bar{Q}^N({3j})}\right]^\top,  j=1,\cdots, M,$  with  $\bar{Q}^N$ being  {a vector of size $3M\times 1$} given by $\bar{Q}^N:=\sum_{n=0}^N B^n\cdot U^I$.  {Here, $U^I$ is a vector of size $3M\times 1$, and $B$ is a matrix of size $3M\times 3M$ with entries defined as $U^I\left(3(j-1)+l\right):=\left(C^{(j)}\cdot U^i(z_j)\right) (l)$, and  $B\left(3(j-1)+l,3(k-1)+m \right):=\left\{\begin{matrix}0,& j=k\\ \left(\omega^2\alpha_k C^{(j)}\, \Gamma^\omega(z_j,z_k)\right) (l,m), &j\neq k\end{matrix}\right.;$ for $j,k=1,2,\cdots,M$ and $l,m=1,2,3$. }
 This scattered field, $U^{s,N}(x)$, satisfies the following  {equations}:
\begin{align}
&U^s(x)-U^{s,N}(x)=O\bigg(a^{{1-h-s+(N+1)(1-h-s)}}\bigg)
\mbox{
 and  }
 \nonumber\\
&U^{s,N}(x)-U^{s,N-1}(x)\sim a^{1-h-s+N\,(1-h-s)}\gg a^{1-h-s+(N+1)(1-h-s)},
\end{align}
 under the assumption $0< 1-h-s\leq \min\left\{\frac{{h}}{N+1},\frac{\frac{s}{2}}{N}\right\}$,
which is uniform in terms of $N<\infty$ and for $x$ in a bounded domain away from collection of centres $z_j$ of inclusion $D_j$, $j=1,2,\cdots,M$.\end{corollary}
%%%%%%%%%%%%%%%%%%%%%%%%%%%%%%%%%%%%%%%%%%%%%%%%%%%%%%%%%%%%%%%%%%%%%%%%%%%%%%%%%%%%%%%%%%%%%%%%%%%%%%%%%%%%%%%%%%%%%%%%%%%%%%%%%%%%%%%%%%%%%%%%%%%%%%%%%%%%%%%%%%%%%%%%%%%%%%%%%%
As mentioned above, the intermediate level of scattering between the Born and the Foldy models describe the finite level of interaction between the scatterers. Specifically, for any non-negative integer $N$, the $N^{\mbox{th}}$-level scattering entails considering $N$ interactions between the scatterers. Therefore, one can observe that $N=0$-level deals with the Born approximation and $N=\infty$- deals with the Foldy-model. Let $U^{s,N}$ represent the scattered field corresponding to the $N^{\mbox{th}}$-level scattering, then following the above corollary, we see that as $N$ tends to infinity, the scattered field $U^{s,N}$ converge to the scattered field associated to the Foldy model $U^{s}$.

%%%%%%%%%%%%%%%%%%%%%%%%%%%%%%%%%%%%%%%%%%%%%%%%%%%%%%%%%%%%%%%%%%%%%
\subsubsection{Characterization of the effective medium}
\hfill\\
The approximations derived in \eqref{farfield-Main-Thm} are useful in solving inverse problems and determining the effective medium. The latter possibility is discussed below. The former possibility, regarding applications in inverse problems, we refer to the previous works \cite{C-C-S-18, A-A-C-K-S-2016, Ghandriche-Sini, G-S-2023, D-G-S-2021, G-S-2022}  {where} such ideas are well developed in different contexts. 
 In this work we show that, for estimating the farfields, the computational workload required to solve integral equations is significantly reduced,  as it transforms into solving the algebraic system presented in \eqref{Algebraic-system-thrm}, which we call as the Foldy-Lax system. If the number of inclusions is relatively large, these asymptotic expansions offer insights into the type of effective medium that can produce the same farfields. The subsequent theorem \eqref{theorem-eff-med}, provides the error rate between the fields generated by the inclusions and those generated by the effective medium.

\begin{theorem}\label{theorem-eff-med}
	
 Let $\Omega$ be a bounded domain in $\mathbb{R}^3$ with unit volume, which is subdivided into $\Omega_m$ cubes, $m=1,2,\cdots, [a^{-s}]$, as described in section \eqref{distribution of inclusions}. Consider small inclusions $D_m$  {being} periodically distributed in each $\Omega_m$.  Assume that the number of inclusions $M \sim a^{-s}=a^{h-1} $, and their minimun distance $d \sim a^t=a^{\frac{s}{3}}$, and $t\leq \frac{1}{2}$, $2h\geq s$. 
Suppose that these inclusions are characterized by the same elastic scattering matrix $\tilde{C}$ and mass density i.e., $\rho_j=\rho_m$, $j\neq m$. In addition, we assume that the scattering matrix $\tilde{C}$ is invertible.In this context, we have the following asymptotic expansions for the difference of scattered fields. These expansions hold uniformly in terms of $\theta$ and $\hat{x}$ in $\mathbb{S}^2$:  
\begin{align}\label{effective-model}
(U_{\mathtt{p}}^{\infty}(\hat{x},\theta)-Y_{\mathtt{p}}^\infty(\hat{x},\theta))\cdot 4\pi(\lambda+2\mu)\hat{x} ={O(a^{\min\{h,\frac{s}{3}\}})}, \quad \mbox{and} \quad (U_{\mathtt{s}}^{\infty}(\hat{x},\theta)-Y_{\mathtt{s}}^\infty(\hat{x},\theta))\cdot 4\pi \mu\hat{x}^\top ={O(a^{\min\{h,\frac{s}{3}\}})} 
\end{align}
 where $Y^\infty$ represents the farfield corresponding to the scattering problem:
\begin{eqnarray}
(\Delta^e+\omega^2\rho_0)Y+ \omega^2\,C\, \chi(\Omega)\,Y=0, \; \mbox{in } \; \mathbb{R}^3,
\end{eqnarray}
where $Y$ is sum of the incident field $Y^i$ and the scattered field $Y^s$.  $Y^i$ satisfies $(\Delta^e+\omega^2\rho_0)Y^i=0,$ in $\mathbb{R}^3$, while $Y^s$ satisfies the \textit{K.R.C}. Additionally, $C$ is a $3\times 3$ matrix defined as $\tilde{\alpha}\,\tilde{C}$, where $\tilde{C}$ is elastic scattering matrix and $\tilde{\alpha}$ is constant, independent of $a$.
	 \end{theorem}

The form of the effective model described in (\ref{effective-model}) can have several important applications. Here we describe few of them.
\begin{enumerate}
\item Observe that the scattering tensor can have positive or negative values depending on the sign $(1-\alpha_1\,\omega^2 \lambda_{n_0}^1)$, see  {\eqref{Polarization-tensor}}, or on how we choose the incident frequency $\omega$ close to the resonant frequency $\sqrt{\frac{1}{\alpha_1\lambda_{n_0}^1}}$. Therefore, choosing a positive (respectively negative) sign, we can enlarge (respectively reduce) the original mass density $\rho_0$ inside $\Omega$. Therefore an application of this result in elastic material design is obvious (including design of elastic cloaking devices). Observe that the matrix tensor can be chosen diagonal if the used inclusions are spherically shaped.
\medskip
\item A second possible application is related to the elastic inverse problem of reconstructing a mass density (and eventually the shear and pressure velocities) in a given domain $\Omega$ using boundary measurements or eventually farfield measurements. It is known that this problem suffers from two shortcomings, namely (1)-Nonlinearity and (2)-Instability. Dealing with the acoustic waves model, we proposed in \cite{ghandriche2023calderon} an approach how one can remove the two mentioned shortcomings. Instead of using the traditional boundary (or farfield) measurements, the idea is first to inject resonant bubbles/droplets into the domain of interest and then perform the usual experiments to collect the boundary/farfield measurement. If such experiments are performed then, based on a homogenization procedure, we could linearize the forward acoustic problem (and hence get read of the non-linear issue) and transform the background as highly dense (or bulky) one.  This high dense/bulky back ground allows us to stabilize the boundary/farfield linearize mapping. This last issue is reminiscent to the use of high frequency of incidence. Here, the do not change the incident frequency but rather we change the mass density or the bulk modulus of the back ground. The details on this approach can be found in \cite{ghandriche2023calderon}. Extending such an approach to elasticity is highly plausible and we leave this for a next contribution.   
\end{enumerate}

The rest of the paper is organised as follows. In Section \ref{section-main-results}, we state and discuss Theorem \eqref{theorem-mp} which include the solution to the problem (\ref{background-equation-Multiple-bodies}--\ref{Transmission-SM-constant-simplemodel-Multiple-bodies}) as well as the approximation of the dominant elastic scattered field and farfield.
The proofs of these results can be found in Section \ref{section-proofs}. In Section \ref{effective-medium}, we deal with the equivalent medium and the proof of Theorem \eqref{theorem-eff-med}.
Finally, in Section \ref{section-appendix} (provided as an Appendix), we derive several technical properties that were useful in proving our results.
\section{Proof of Theorem \ref{theorem-mp}}\label{section-proofs}

The dominant elastic scattered and farfields in Theorem \ref{theorem-mp} are determined by employing the Lippmann-Schwinger equation in conjunction with apriori estimates associated with the total field of the problem (\ref{background-equation-Multiple-bodies}--\ref{Transmission-SM-constant-simplemodel-Multiple-bodies}).

The Lippmann-Schwinger equation corresponding to the problem (\ref{background-equation-Multiple-bodies}--\ref{Transmission-SM-constant-simplemodel-Multiple-bodies}) is:
{
\begin{eqnarray}\label{lippmann-schwinger-equation-multiple-particle}
U^t(x)=U^i(x)+\sum_{j=1}^M\int_{D_j}{\omega^2({\rho_j}-\rho_{0})\Gamma^{\omega}(x,y)\cdot U^t(y)\, dy}, \; x\in\mathbb{R}^3.
\end{eqnarray}
}

We prove Theorem \ref{theorem-mp} in three steps. In the first level, we approximate the scattered field in terms of the total fields with the help of the Lippmann Schwinger equation, and then in the second level we derive the apriori estimates satisfied by the total field. Finally, we complete the proof of deriving the scattered field and farfield approximations using the Lippmann Schwinger equation and derived apriori estimates. All these three levels are explained in the subsequent subsections Section \ref{section-scatttered-field-estimation}, Section \ref{section-apriori-estimates}  and Section \ref{section-ending-proof}, respectively.
%%%%%%%%%%%%%%%%%%%%%%%%%%%%%%%%%%%%%%%%%%%%%%%%%%%%%%%%%%%%%%%%%%%%%%%%%%%%%%%%%%%%%%%%%%%%%
\subsection{Scattered field approximation}\label{section-scatttered-field-estimation}\hfill

From the Lippmann Schwinger equation \eqref{lippmann-schwinger-equation-multiple-particle} of the problem (\ref{background-equation-Multiple-bodies}--\ref{Transmission-SM-constant-simplemodel-Multiple-bodies}) and
by making use of Taylor series expansion of $\Gamma^\omega(x,y)$ about $y$ and near centres $z_j$, scattered field $U^s$ as for $x$ away from $D:=\cup_{j=1}^MD_j$ can be expressed as,
\begin{align}
U^s(x)&=\sum_{j=1}^M\int_{D_j}{\omega^2({\rho_j}-\rho_{0})\Gamma^{\omega}(x,y)\cdot U^t(y) dy}
\nonumber\\
&=\sum_{j=1}^M(\rho_j-\rho_0)\omega^2\left[\Gamma^{\omega}(x,z_j)\cdot \int_{D_j}U^t(y)dy+A_j(x)\right], \quad x\in \mathbb{R}^3\setminus\bar{D}\label{Us-approximation-with-Bj-multiple-particle}
\end{align}
where $A_j(\cdot)$ is a vector function of size $3\times 1$ and it is defined as
\begin{align}\label{def-Aj}
A_j(x)[l]&:=  \int_{D_j}  \int_{0}^1  \nabla_y\Gamma^{\omega}_l(x,z_j+t(y-z_j))\cdot  (y-z_j)dt\cdot U^t(y)\,dy.
\end{align}
Knowing that $|y-z_j|\leq a\,diam(B_j)$ for $y\in D_j$,  $j=1,\cdots, M$, and by applying the Cauchy Schwartz inequality (CSI), we can observe the behaviour of $A_j$ as follows;
\begin{align}
|A_j(x)|^2=\;&\sum_{k=1}^3 |\int_{D_j}  \int_{0}^1  \nabla_y\Gamma^{\omega}_k(x,z_j+t(y-z_j))\cdot (y-z_j) dt \cdot U^t(y)\,dy|^2
\nonumber\\
\underset{CSI}{\leq}&  \norm{U^t}_{(L^2(D_j))^3}^2\sum_{k=1}^3\norm{\int_{0}^{1} \nabla\Gamma^{\omega}_k(x,z_j+t(\cdot-z_j)) \cdot (\cdot-z_j)dt}_{(L^2(D_j))^3}^2 \qquad \nonumber\\
\leq\;& \norm{U^t}_{(L^2(D_j))^3}^2\sum_{k=1}^3\int_{D_j}\sum_{l=1}^3|\int_{0}^{1}\nabla_y\Gamma^{\omega}_{lk}(x,z_j+t(y-z_j))\cdot (y-z_j)dt\,|^2dy
\nonumber
\\
\underset{\eqref{grad-Gamma-bounded-x-away-D-y-D-mp}}{\leq}& \norm{U^t}_{(L^2(D_j))^3}^2\sum_{k=1}^3\int_{D_j}\sum_{l=1}^3 a^2\, diam(B_j)^2\,{
H_2^2}\,dy
\nonumber\\
\implies |A_j(x)|=\;& O( a^{\frac{5}{2} }\norm{U^t}_{(L^2(D_j))^3)}),\label{B_j-app-multiple-particle}
\end{align}
where $H_2$ is a constant satisfying the estimate $\, |\nabla_y\Gamma^{\omega}_{lk}(x,z_j+t(y-z_j))|\leq\, H_2$ for $y\in D_j$ and $x$ away from $D_j$, see {Lemma \ref{lemma-Gamma-propertie-Dl-mp}} in the appendix for details.

Now, by making use of \eqref{def-rho-sm-case1-multiple-particles} and \eqref{B_j-app-multiple-particle} in \eqref{Us-approximation-with-Bj-multiple-particle}, we get
\begin{eqnarray}
U^s(x)&\underset{\eqref{def-rho-sm-case1-multiple-particles},\eqref{B_j-app-multiple-particle}}{=}&
\sum_{j=1}^M (\rho_j-\rho_0)\omega^2\Gamma^{\omega}(x,z_j)\cdot\int_{D_j}U^t(y)dy+O\bigg( a^{\frac{1}{2} }\sum_{j=1}^M\norm{U^t}_{(L^2(D_j))^3}\bigg)
\nonumber\\
&=& \sum_{j=1}^M (\rho_j-\rho_0)\omega^2\,\Gamma^{\omega}(x,z_j)\cdot \int_{D_j}U^t(y)dy +O\left( a^{\frac{1}{2} }M^{\frac{1}{2}}\left(\sum_{j=1}^M\norm{U^t}_{(L^2(D_j))^3}^2\right)^{\frac{1}{2}}\right).
\label{us-app-before-substituting-norm-u-integral-u-multiple-particle}
\end{eqnarray}

\subsection{Apriori Estimates}\label{section-apriori-estimates}
\hfill

In \eqref{us-app-before-substituting-norm-u-integral-u-multiple-particle} of the previous section, we approximated the scattered field in terms of the total field for $x$ away from all the inclusions, which is not sufficient enough. Hence, to improvise further, the estimates of the total field and its integrated field are required.
In this section, we state and prove apriori estimates satisfied by the total field $U^t$. 

To achieve this, first we introduce the volume integral operators  $N_j^\omega:(L^2(D_j))^3\to (L^2(D_j))^3$ and $\tilde{N}_j^\omega:(L^2(B_j))^3\to (L^2(B_j))^3$, for $j=1,\cdots,M$  defined by
\begin{eqnarray}\label{def-newtonian}
 N_j^\omega U(x):=\int_{D_j}{\Gamma^\omega(x,y)\cdot U(y)\,dy}, & { {\tilde{N}}}_j^\omega (\zeta):=\int_{B_j}{\Gamma^\omega(\zeta,\eta)\cdot {U}(\eta)\,d\eta};
 \end{eqnarray}
 
 Then, for zero frequency, these operators are nothing but the  {Navier} operators and further we can observe that $N_j^0$ is a compact and self-adjoint on the separable Hilbert Space $(L^2(D_j))^3$ and hence has a complete orthonormal eigen system, denoted by $\{(\lambda_n^j, e_n^j)\}_{n\in\mathbb{N}}$. The similar properties holds for the  {Navier} operator ${N}_j^0$ with its complete orthogonal system is denoted by $\{({\lambda}_n^j, {e}_n^j)\}_{n\in\mathbb{N}}$. These eigen values and eigen functions are related and satisfy the following properties, see \cite[Lemma 2.1]{challa2023extraction} for instance;

\begin{enumerate}
\item The  eigensystem  $\{(\lambda_n^D, e_n^D)\}$ satisfies the following scaling properties;
\begin{align}
 \int_{D_j}{e_n^{j}(x)\,dx}={a}^{\frac{3}{2}}\int_{B_j}{\tilde{e}_n^j(\eta)\,d\eta} \;\mbox{ and }\; \lambda_n^{j}={a}^2\tilde{\lambda}_n^{j}.&&\hspace{\labelwidth}\label{eigen-function-scaling-mp}
\end{align}
\item Fixing a natural number $n_0$ and by choosing the frequency $\omega$  satisfying
 \begin{eqnarray}\label{def-omega-choosen-mp-2*}
\omega^2:=\frac{1 \pm b_j\,a^{h_j}}{\rho_j\lambda_{n_0}^{j}}, \mbox{ with $b_j$ being a positive constant  and $1>h_j>0$  for } j=1,\cdots,M,  \quad  
\\
 \mbox{(or equivalently)}\qquad\qquad\qquad\qquad\qquad\qquad\qquad
 \nonumber\\
\vert\omega^2-\omega_{n_{0_{(j)}}}^2\vert\simeq a^{h_j}, \mbox{ with } 0<h_j<1 \mbox{ and } \omega_{n_{0_{(j)}}}^2:=\dfrac{1}{\rho_j\lambda_{n_0}^{j}} \mbox{ for } j=1,\cdots,M,\qquad\qquad\label{def-omega-choosen-mp}
\end{eqnarray} we have
\begin{align}
\quad& \sigma_j :=\underset{n(\neq n_0)}{\inf} \{|1-\alpha_j\,\omega^2\lambda_n^{j}|^2\}  \quad[>0] \mbox{ exists} &&\hspace{\labelwidth} \label{def-sigma-positive} \\
\quad& (I-\alpha_j\,\omega^2N_j^0):(L^2(D_j))^3\to (L^2(D_j))^3 \mbox{ is invertible}.&&\hspace{\labelwidth}\label{inv-I-minusalphaj}
\end{align}
\end{enumerate}

Also observe that, plugging in the definition of $N_j^\omega$, we can rewrite Lippmann Schwinger equation \eqref{lippmann-schwinger-equation-multiple-particle} as
 \begin{align}
 U^t(x)
=&U^i(x)+\sum_{m=1 }^M \alpha_m \omega^2 N_m^\omega U^t(x).\label{norm-u-first-step-with-newtonian-operator-multiple-particle-0}
\end{align}
\subsubsection{The relative distribution of the small inclusions { {to derive the effective medium}}}\label{distribution of inclusions}
\hfill
\\
	 { As mentioned earlier, let $\Omega$ be a bounded domain of unit volume, containing $D_m$, $m=1,\cdots,M$. As $M=O(a^{-s})$, we shall divide $\Omega$ into $[a^{-s}]$ cubes $\Omega_m$, $m=1,\cdots,[a^{-s}]$, of same volume. Each $\Omega_m$ contains $D_m$, with centre at $z_m$, and may contain some other $D_j$'s, and has sides estimated as  {$(\frac{a}{2}+d^{\alpha})$}, with $0\leq \alpha\leq 1$.
We consider the $\Omega_m$'s to be equal, up to translation, and periodically distributed in $\Omega$. The distribution of inclusions need not be periodic. The case dealt with in the homogenization theory is when the distribution of inclusions in each $\Omega_m$ is periodic and precisely one. In our work, we consider this case and discuss the effective medium by considering such a distribution. In this situation, we have $\alpha=1$ and one can observe that each $\Omega_m$ contains atmost one inclusion $D_m$ and $s=3t$.  The results can be extended to cases where $\Omega_m$ contains a different number of inclusions.}
\bigskip
\\
\begin{minipage}{0.5\textwidth}
%\textbf{Your title}\par\medskip
\centering
\begin{tikzpicture}
\draw[red, very thick, use Hobby shortcut,closed=true]
(-2.65,0) .. (-2,2) .. (0,2.38) .. (1.7,1.9)..(2.1,1.5)..(2.35,0)..(2.3,-0.5)..(1.4,-2.28)..(-0.4,-2.4)..(-2.25,-1.6);

\foreach \x in {-1.56,-1.3,-1.04,-0.78,-0.52,-0.26,0,0.26,0.52,0.78,1.04,1.3, 1.56}
  \foreach \y in {-1.82,-1.56,-1.3,-1.04,-0.78,-0.52,-0.26,0,0.26,0.52,0.78,1.04,1.3, 1.56}
   { \draw[gray!70](\x,\y)+ (-0.13,-0.13) rectangle ++(.13,.13); 
   \shade[ball color=cyan! 50] (\x,\y) circle (0.07cm); 
   } 
\foreach \x in {1.82}
\foreach \y in {-1.3,-1.04,-0.78,-0.52,-0.26,0,0.26,0.52,0.78,1.04,1.3}
   { \draw[gray!70](\x,\y)+ (-0.13,-0.13) rectangle ++(.13,.13); 
   \shade[ball color=cyan! 50] (\x,\y) circle (0.07cm); 
   } 
\foreach \x in {1.82}
 \foreach \y in {-1.82,-1.56,-1.3,-1.04,-0.78,-0.52,-0.26,0,0.26,0.52,0.78,1.04,1.3, 1.56}
  { \draw[gray!70](\x,\y)+ (-0.13,-0.13) rectangle ++(.13,.13); 
   } 

   \foreach \x in {-1.82}
  \foreach \y in {-1.56,-1.3,-1.04,-0.78,-0.52,-0.26,0,0.26,0.52,0.78,1.04,1.3, 1.56}   { \draw[gray!70](\x,\y)+ (-0.13,-0.13) rectangle ++(.13,.13); 
   \shade[ball color=cyan! 50] (\x,\y) circle (0.07cm); 
   } 
   \foreach \x in {-1.82}
  \foreach \y in {-1.82,-1.56,-1.3,-1.04,-0.78,-0.52,-0.26,0,0.26,0.52,0.78,1.04,1.3, 1.56}   { \draw[gray!70](\x,\y)+ (-0.13,-0.13) rectangle ++(.13,.13); 
   } 
   \foreach \x in {-2.08}
  \foreach \y in {-1.3,-1.04,-0.78,-0.52,-0.26,0,0.26,0.52,0.78,1.04,1.3}   { \draw[gray!70](\x,\y)+ (-0.13,-0.13) rectangle ++(.13,.13); 
   \shade[ball color=cyan! 50] (\x,\y) circle (0.07cm); 
   }  
   \foreach \x in {-2.08}
  \foreach \y in {-1.82,-1.56,-1.3,-1.04,-0.78,-0.52,-0.26,0,0.26,0.52,0.78,1.04,1.3, 1.56}   { \draw[gray!70](\x,\y)+ (-0.13,-0.13) rectangle ++(.13,.13); 
   } 
  
  \foreach \x in {-2.34}
  \foreach \y in {-0.78,-0.52,-0.26,0,0.26,0.52,0.78}
   { \draw[gray!70](\x,\y)+ (-0.13,-0.13) rectangle ++(.13,.13); 
   \shade[ball color=cyan! 50] (\x,\y) circle (0.07cm); 
   }  
    \foreach \x in {-2.34,-2.60}
  \foreach \y in {-2.08,-1.82,-1.56,-1.3,-1.04,-0.78,-0.52,-0.26,0,0.26,0.52,0.78,1.04,1.3, 1.56}
   { \draw[gray!70](\x,\y)+ (-0.13,-0.13) rectangle ++(.13,.13); 
   }  
  \foreach \x in {2.08}
  \foreach \y in {-0.52,-0.26,0,0.26,0.52,0.78,1.04}
   { \draw[gray!70](\x,\y)+ (-0.13,-0.13) rectangle ++(.13,.13); 
   \shade[ball color=cyan! 50] (\x,\y) circle (0.07cm); 
   }  
   \foreach \x in {2.08}
  \foreach \y in {-1.82,-1.56,-1.3,-1.04,-0.78,-0.52,-0.26,0,0.26,0.52,0.78,1.04,1.3, 1.56}
   { \draw[gray!70](\x,\y)+ (-0.13,-0.13) rectangle ++(.13,.13); 
     }
 \foreach \x in { 2.34}
  \foreach \y in {-1.04,-0.78,-0.52,-0.26,0,0.26,0.52,0.78,1.04,1.3, 1.56}
   { \draw[gray!70](\x,\y)+ (-0.13,-0.13) rectangle ++(.13,.13); 
     }
    \foreach \y in {-2.08}
  \foreach \x in {-0.78,-0.52,-0.26,0,0.26,0.52,0.78,1.04,1.3}
   { \draw[gray!70](\x,\y)+ (-0.13,-0.13) rectangle ++(.13,.13); 
   \shade[ball color=cyan! 50] (\x,\y) circle (0.07cm); 
 }  
\foreach \y in {-2.08,-2.34}
  \foreach \x in {-2.08,-1.82,-1.56,-1.3,-1.04,-0.78,-0.52,-0.26,0,0.26,0.52,0.78,1.04,1.3,1.56,1.82,    2.08}
   { \draw[gray!70](\x,\y)+ (-0.13,-0.13) rectangle ++(.13,.13);
   }
   \foreach \y in {-2.6}
  \foreach \x in {-0.78,-0.52,-0.26,0,0.26,0.52,0.78}
   { \draw[gray!70](\x,\y)+ (-0.13,-0.13) rectangle ++(.13,.13);
   }
   \foreach \y in {1.82}
  \foreach \x in {-1.82,-1.56,-1.3,-1.04,-0.78,-0.52,-0.26,0,0.26,0.52,0.78,1.04,1.3}
   { \draw[gray!70](\x,\y)+ (-0.13,-0.13) rectangle ++(.13,.13); 
   \shade[ball color=cyan! 50] (\x,\y) circle (0.07cm); 
 }
  \foreach \y in {2.08}
  \foreach \x in {-1.3,-1.04,-0.78,-0.52,-0.26,0,0.26,0.52}
   { \draw[gray!70](\x,\y)+ (-0.13,-0.13) rectangle ++(.13,.13); 
   \shade[ball color=cyan! 50] (\x,\y) circle (0.07cm); 
 }  
\foreach \y in {1.82,2.08}
  \foreach \x in {-2.34,-2.08,-1.82,-1.56,-1.3,-1.04,-0.78,-0.52,-0.26,0,0.26,0.52,0.78,1.04,1.3,1.56,1.82,    2.08}
   { \draw[gray!70](\x,\y)+ (-0.13,-0.13) rectangle ++(.13,.13);
   }
   \foreach \y in {2.34}
  \foreach \x in {-2.08,-1.82,-1.56,-1.3,-1.04,-0.78,-0.52,-0.26,0,0.26,0.52,0.78,1.04,1.3}
   { \draw[gray!70](\x,\y)+ (-0.13,-0.13) rectangle ++(.13,.13);
   }
   \node[black,very thick] at (0.5, -2.3){{$\Omega$}};
   \node[black, very thick] at (-3, 0){$\partial\Omega$};
   \node[red] at (-0.18,0) {\tiny{$\Omega_m$}};
\end{tikzpicture}
\\
\captionof{figure}{}
\end{minipage}
\begin{minipage}{0.4\textwidth}
\centering
\begin{tikzpicture}
		[cube/.style={very thick,black},
			grid/.style={very thin,gray},
			axis/.style={->,blue,thick}]
	\draw[dotted] (1,1,1) ellipse (0.5cm and 0.1cm);
\draw (1,1,1) circle (0.5cm);
\draw[fill] (1,1,1) circle (0.01cm);
\draw[dotted] (3,1,1) ellipse (0.5cm and 0.1cm);

\draw (3,1,1) circle (0.5cm);
\draw[fill] (3,1,1) circle (0.01cm);
\draw[red] (1.5,1,1)--(2.5,1,1);
\node at (2,1,1.5) {$d$}; 

	\draw[dotted] (0,0,0) -- (2,0,0) ;
	\draw[cube] (2,0,0)--(2,2,0) -- (0,2,0) -- (0,0,0);
	\draw[cube] (0,0,2) -- (0,2,2) -- (2,2,2) -- (2,0,2) -- cycle;
	\draw[dotted] (2,0,0)--(4,0,0);
	\draw[cube] (4,0,0)--(4,0,2)--(2,0,2)--(2,0,0);
	\draw[cube] (2,2,2)--(2,2,0)--(4,2,0)--(4,2,2)--cycle;
	\draw[cube] (4,2,2)--(4,0,2);
	\draw[cube] (4,2,0)--(4,0,0);
	%draw the edges of the cube
	\draw[cube] (0,0,0) -- (0,0,2);
	\draw[cube] (0,2,0) -- (0,2,2);
	\draw[cube] (2,0,0) -- (2,0,2);
	\draw[cube] (2,2,0) -- (2,2,2);

	%draw the edges of the cube
	\draw[cube] (0,0,0) -- (0,0,2);
	\draw[cube] (0,2,0) -- (0,2,2);
	\draw[cube] (2,0,0) -- (2,0,2);
	\draw[cube] (2,2,0) -- (2,2,2);
 
\node at (1,1,1)[red]{{{$z_m$}}};
\node at (1,0.3,1)[red]{{{$D_m$}}};
\node at (3,1,1)[red]{{{$z_j$}}};
\node at (3,0.3,1)[red]{{{$D_j$}}};
	\node at (3.1,0,2.6)[red]{{{$\Omega_j$}}};
 \node at (1.3,0,2.6)[red]{{{$\Omega_m$}}};
\end{tikzpicture}
\\
\captionof{figure}{}
\end{minipage}

To obtain apriori estimates, we need to estimate sums of the form $\sum_{m=1,m\neq j}^M |z_m- z_j|^{k}$, where $k$ is positive real number, and $z_m$ is the center of inclusion $D_m$. Building upon references \cite{multiscalechallasini2016} and \cite{Ammari-Challa-Choudhury-Sini-1}, we will provide a brief overview of the approach for estimating these sums using a systematic counting method. We need to count the number of inclusions to estimate these sums. However, we can simplify this problem by counting the $\Omega_m$ cubes, since the number of inclusions within each $\Omega_m$ is uniformly bounded in terms of $m$. We use the periodic structure of the $\Omega_m$ cubes to achieve this. The counting process is executed as follows: For any fixed $m=1,\cdots,M,$ we take the cube $\Omega_m$ as the starting point, containing $D_m$ at its center. To distinguish between the points $z_j$, for $j\neq m$, we attach other cubes to it, forming different layers of cubes based on their distance from $D_m$, resembling a Rubik's cube with $\Omega_m$ at the center. As a result, the total number of cubes up to the $n^{th}$ layer, commencing from $\Omega_m$, comprises $(2n+1)^3$ cubes.  Consequently, the number of cubes located in the $n^{th}$ layer (for $n\neq 0)$ will be atmost $[(2n+1)^3-(2n-1)^3]$, and their
distance from center inclusion $D_m$ is more than $n\,d^\alpha$ or $n\bigg(a^{\frac{s}{3}}-\frac{a}{2}\bigg)$.

With this way of counting, for $j$ fixed and $k>0$, the following formulas are derived in \cite{Ammari-Challa-Choudhury-Sini-1}:
\begin{enumerate}
\item If $k<3$, then
\begin{eqnarray}
\underset{i=1,i\neq j}{\sum^M}|z_i-z_j|^{-k}=O(d^{-k})+O(d^{-3\alpha}).\label{sum-i-neq-j-reciprocal-dij-k<3-mp}
\end{eqnarray}
\item for $k=3,$ then
\begin{eqnarray}
\underset{i=1,i\neq j}{\sum^M}|z_i-z_j|^{-k}=O(d^{-k})+O(d^{-3\alpha}|ln(d)|).\label{sum-i-neq-j-reciprocal-dij-k=3-mp}
\end{eqnarray}
\item If $k>3$, then
\begin{eqnarray}
\underset{i=1,i\neq j}{\sum^M}|z_i-z_j|^{-k}=O(d^{-k})+O(d^{-\alpha k}).\label{sum-i-neq-j-reciprocal-dij-k>3-mp}
\end{eqnarray}
\end{enumerate}

{ In our analysis, estimating the volume of cubes that intersect with $\partial \Omega$ is essential. This set is not empty since $\Omega$ can assume any arbitrary shape (unless $\Omega$ is a cube). Denoting the cubes intersecting with $\partial \Omega$ as $\Omega^{'}_{\mathtt{m}}$,  the volume of  $\Omega \cap \Omega^{'}_{\mathtt{m}}$ depends on the shape of $\Omega$, making an exact estimation of the volume is challenging. However, it is apparent that the volume will be of the order $a^{s}$. 
We refrain from placing inclusions within cubes touching $\partial\Omega$.  Figure 1 illustrates a schematic representation of such an arrangement of inclusions in $\Omega$. The volume of each  $\Omega^{'}_{\mathtt{m}}$ is of order $a^s$, while its maximum radius is of order  $a^{\frac{s}{3}}$.  Hence, the area of its intersecting surface with $\partial\Omega$ is of order $a^{\frac{2}{3}s}$. Considering $\partial\Omega$ has an area of order one, the number of cubes intersecting with it won't exceed the order of $a^{-\frac{2}{3}s}$. Thus, as $a\to 0$, the volume of this set will not surpass the order of  $a^{-\frac{2}{3}s}a^{s}=a^{\frac{s}{3}}$.}

Using the properties mentioned above and the modified Lippmann Schwinger equation \eqref{norm-u-first-step-with-newtonian-operator-multiple-particle-0}, few estimates satisfied by the total field are derived and are presented in the following propositions. Do observe that, in our work we are dealing with the situation where $\alpha=1$.
\begin{proposition}\label{proposition-norm-u-apriori-estimate-mp} The total field $U^t$ satisfies the estimate:
\begin{equation}\label{norm-u-app-piecewise-density-multiple-particle-lemma}
\sum_{j=1}^M\norm{U^t}_{(L^2(D_j))^3}^2\lesssim a^{-2h}\sum_{j=1}^M\norm{U^i}_{(L^2(D_j)^3}^2,
\end{equation}
whenever $s+h\leq 1$, where $h:=\underset{j=1,\cdots,M}{\max}\{h_j\}$.
\end{proposition}
 \begin{proof}
For $x\in D_j$,  introducing the   {Navier} operator the Lippmann Schwinger equation \eqref{norm-u-first-step-with-newtonian-operator-multiple-particle-0} can be rewritten as
\begin{align}
 U^t(x)-\alpha_j\,\omega^2 N^0_j(U^t)(x)
&=U^i(x)+\alpha_j\,\omega^2 (N^\omega_j-N^0_j)(U^t)(x)+\hspace{-0.1cm}\sum_{\substack{m=1 \\ m\neq j}}^M\int_{D_m}\hspace{-0.2cm}\alpha_m\,\omega^2\Gamma^{\omega}(x,y)\cdot U^t(y)dy.\label{norm-u-first-step-with-newtonian-operator-multiple-particle}
\end{align}
 Since, $(\lambda_n^j, {e^j_{n,l})_{n\in\mathbb{N},\; l=1,\cdots l_{\lambda_{n_0}^j}}}$ is a complete orthonormal eigensystem of the compact, self adjoint operator $N_j^0$, by Parseval's identity, we have \footnote{Here $\langle\,; \rangle_j$ represents the $(L^2(D_j))^3$ inner product}
\begin{align}
\norm{U^t-\alpha_j\,\omega^2N^0_j(U^t)}_{(L^2(D_j))^3}^2&=\sum_{n=1}^\infty|\,\langle\, U^t-\alpha_j\,\omega^2N^0_j(U^t)\,;\,e_n^{j}\,\rangle_j\,|^2
\nonumber\\
&=|(1-\alpha_j\,\omega^2\lambda_{n_0}^{j})|^2\,{\sum_{l=1}^{l_{\lambda_{n_0}^j}}|\,\langle\,U^t\,;\,e_{{n_0},l}^{j}\,\rangle_j\,|^2\hspace{-0.05cm}+\hspace{-0.2cm}\sum_{n\neq {n_0}}{\hspace{-0.1cm}|(1-\alpha_j\,\omega^2\lambda_n^{j})|^2\,\sum_{l=1}^{l_{\lambda_n^j}}|\langle\,U^t\,;\,e^{j}_{n,l}\,\rangle_j\,|^2}}.\label{without_sigma-multiple-particle}
\end{align}
 Making use of the fact that $\sigma_j :=\underset{n(\neq n_0)}{\inf} \{|1-\alpha_j\,\omega^2\lambda_n^{j}|^2\}$ is positive from \eqref{def-sigma-positive} and Parseval's identity we obtain
\begin{align}
\norm{U^t}^2_{(L^2(D_j))^3}
\underset{\eqref{without_sigma-multiple-particle} }{\leq}&  \left[ \dfrac{1}{|1-\alpha_j\,\omega^2\lambda_{n_0}^{j}|^2}+\dfrac{1}{\sigma_j}\right] \norm{U^t-\alpha_j\,\omega^2N^0_{j}(U^t)}_{(L^2(D_j))^3}^2
 \nonumber\\
\implies \norm{U^t}_{(L^2(D_j))^3}\leq \;\;&\left(1+\frac{|1-\alpha_j\,\omega^2\lambda_{n_0}^{j}|^2}{\sigma_j}\right)^{\frac{1}{2} } \dfrac{1}{|1-\alpha_j\,\omega^2\lambda_{n_0}^{j}|} \norm{U^t-\alpha_j\,\omega^2N^0_j(U^t)}_{(L^2(D_j))^3}.\label{norm_u_term-multiple-particle} 
\end{align}
Now to estimate $\norm{U^t-\alpha_j\,\omega^2N^0_j(U^t)}_{(L^2(D_j))^3}$, 
consider  \eqref{norm-u-first-step-with-newtonian-operator-multiple-particle} to get the inequality
\begin{align}
\norm{(I\hspace{-0.1cm}-\hspace{-0.07cm}\alpha_{j}\,\omega^2N^0_j)U^t}_{(L^2(D_j))^3}\hspace{-0.05cm}\leq \hspace{-0.05cm} \norm{U^i}_{(L^2(D_j))^3}\hspace{-0.1cm}+\hspace{-0.1cm}|\alpha_{j}\omega^2|\;\norm{(N^\omega_j-N^0_j)U^t}_{(L^2(D_j))^3}
\hspace{-0.1cm}+\hspace{-0.1cm}\sum_{\substack{m=1 \\ m\neq j}}^M \hspace{-0.05cm}\norm{\hspace{-0.1cm}\int_{D_m}\hspace{-0.48cm}\alpha_m\,\omega^2 \Gamma^{\omega}(\cdot,\hspace{-0.05cm}y)\hspace{-0.05cm}\cdot\hspace{-0.05cm} U^t(y)dy}_{(L^2(D_j))^3}.
\label{term_in_norm_u-multiple-particle}
\end{align}
The last two terms of the above inequality can be estimated as follows;
\begin{align}
\setcounter{mysubequations}{0}
      \hspace{-7.55cm}\mysubnumber\qquad\qquad \norm{(N^\omega_j-N^0_j)U^t}_{(L^2(D_j))^3}^2
=\;&\int_{D_j}{\sum_{i=1}^3\bigg(|\int_{D_j}{(\Gamma^{\omega}-\Gamma^0)_{i}(x,y)\cdot U^t(y)dy}|^2\bigg)}dx
\nonumber
\\
\leq \;& \;\;\norm{U^t}^2_{L^2(D_j)^3}\sum_{i=1}^3\sum_{k=1}^3\int_{D_j}\int_{D_j}|(\Gamma^{\omega}-\Gamma^0)_{ki}(x,y)|^2dydx
\nonumber
\\
 \underset{\eqref{Gamma-Gammm0-bounded-x,y-in-D-H1-value}}{\leq} & \;9 (H_1)^2\norm{U^t}^2_{(L^2(D_j))^3} a^6|B_j|^2
 \nonumber
 \\
\implies\norm{(N^\omega_j-N^{0}_j)U^t}_{(L^2(D_j))^3}\leq \;& {3\,H_1 a^3|B_j|\,\,\norm{U^t}_{(L^2(D_j))^3}}\label{Nw-N0_approximation-multiple-particle}.      
\end{align}
Here $H_1$ is a constant satisfying {the estimate $\, |(\Gamma^{\omega}-\Gamma^0)_{ki}(x,y)|\leq\, H_1$ for $x,\,y\in D_j$, see Lemma \ref{lemma-Gamma-propertie-Dl-mp} in the appendix for details.}
\begin{align}
\quad\mysubnumber\quad \norm{\int_{D_m}\alpha_m\,\omega^2 \Gamma^{\omega}(\cdot,y)\cdot U^t(y)dy}_{(L^2(D_j))^3}^2=\;&\int_{D_j}\sum_{k=1}^3|\int_{D_m}\hspace{-0.3cm}\alpha_m\,\omega^2 \Gamma^{\omega}_k(x,y)\cdot U^t(y)dy|^2dx, \quad m\neq j
\nonumber\\
\underset{\text{CSI}}{\leq} & \;|\alpha_m\,\omega^2|^2\norm{U^t}^2_{(L^2(D_m))^3}\sum_{k=1}^3\sum_{i=1}^3\int_{D_j}\int_{D_m}\hspace{-0.3cm}|\Gamma^{\omega}_{ik}(x,y)|^2dydx
\nonumber\\
\underset{\eqref{Gamma-is-bounded-x-in-Di-y-in-Dj-mp}}{\leq}&
3|\alpha_m\,\omega^2|\left(\frac{H_3}{d_{mj}}+H_4\right) a^3\,|B_j|^{\frac{1}{2}}\,|B_m|^{\frac{1}{2}}\norm{U^t}_{(L^2(D_m))^3} ,
\nonumber
\end{align}
\noindent
and hence
\begin{align}
 \sum_{\substack{m=1 \\ m\neq j}}^M\norm{\hspace{-0.1cm}\int_{D_m}\hspace{-0.45cm}\alpha_m\,\omega^2 \Gamma^{\omega}(\cdot,y)\hspace{-0.05cm}\cdot\hspace{-0.05cm} U^t(y)dy}_{(L^2(D_j))^3}
\leq& 3\;\underset{\substack{m=1\\ m\neq j}}{\max^M}\{|\alpha_m\,\omega^2||B_m|^{\frac{1}{2}} \} |B_j|^{\frac{1}{2}} a^3\hspace{-0.15cm}\left(\hspace{-0.2cm}\left(\hspace{-0.08cm}\sum_{\substack{m=1 \\ \;m\neq j}}^M\hspace{-0.13cm}\frac{H_3^2}{d_{mj}^2}\hspace{-0.1cm}\right)^{\frac{1}{2}}\hspace{-0.3cm}+\hspace{-0.05cm}H_4(M-1)^{\frac{1}{2}} \hspace{-0.1cm}\right)\hspace{-0.2cm}\left(\hspace{-0.1cm}\sum_{\substack{m=1 \\ \;m\neq j}}^M \norm{U^t}_{(L^2(D_m))^3}^2\hspace{-0.14cm}\right)^{\hspace{-0.14cm}\frac{1}{2}}\hspace{-0.2cm}.
\label{sum-term-app-multiple-particle}
\end{align}
Here, $H_3$ and $H_4$ are constants satisfying the estimate $|\Gamma_{ik}^\omega(x,y)|\leq \frac{H_3}{d_{mj}}+H_4$ for $x\in D_m$ and $y\in D_j$, $m\neq j$, $i,k=1,2,3$, see Lemma \ref{lemma-gamma-in-Dp-Dq-p-ne-q} in Appendix.

%%%%%%%%%

Substituting \eqref{Nw-N0_approximation-multiple-particle} and \eqref{sum-term-app-multiple-particle} in \eqref{term_in_norm_u-multiple-particle}, and so in \eqref{norm_u_term-multiple-particle}, we get
\begin{align}
\norm{U^t}_{(L^2(D_j))^3}\leq& \left(1+\frac{|1-\alpha_j\,\omega^2\lambda_{n_0}^{j}|^2}{\sigma_j}\right)^{\frac{1}{2}}\dfrac{1}{|1-\alpha_j\,\omega^2\lambda_{n_0}^{j}|}\bigg[\norm{U^i}_{(L^2(D_j))^3}
+3\, \underset{m=1}{\max^M}\,\{|\alpha_m\,\omega^2| \} |B_j|^{\frac{1}{2}}
\nonumber\\
&\underset{m=1}{\max^M}\,\{|B_m|^{\frac{1}{2}}|\} a^3
\left(\left(\hspace{-0.2cm}\left(\hspace{-0.08cm}\sum_{\substack{m=1 \\ \;m\neq j}}^M\hspace{-0.13cm}\frac{H_3^2}{d_{mj}^2}\hspace{-0.1cm}\right)^{\frac{1}{2}}\hspace{-0.3cm}+\hspace{-0.05cm}H_4(M-1)^{\frac{1}{2}} \hspace{-0.1cm}\right)\hspace{-0.2cm}\left(\hspace{-0.1cm}\sum_{\substack{m=1 \\ \;m\neq j}}^M \norm{U^t}_{(L^2(D_m))^3}^2\hspace{-0.14cm}\right)^{\hspace{-0.14cm}\frac{1}{2}}\hspace{-0.2cm}+
H_1\norm{U^t}_{(L^2(D_j))^3}\right)\Bigg],
\nonumber
\end{align}
which further gives
\begin{align}
\sum_{j=1}^M\norm{U^t}^2_{(L^2(D_j))^3}\hspace{-0.35cm}\underset{\eqref{def-omega-choosen-mp-2*},\eqref{def-alpham}}{\leq}
\quad\;\,& \hspace{-0.4cm}3\sum_{j=1}^M\dfrac{1}{|1-\alpha_j\,\omega^2\lambda_{n_0}^{j}|^2}\norm{U^i}_{(L^2(D_j))^3}^2+27\,K\underset{j=1}{\max^M}\{ a^{-2h_j}\}\,\underset{m=1}{\max^M}\{|\alpha_m\,\omega^2|^2\}
\,\underset{m=1}{\max^M}\{|B_m|\}^2 a^6
 \nonumber\\
&\left[\sum_{j=1}^M\hspace{-0.05cm}\left(\hspace{-0.1cm}H_3\left(\sum_{\substack{m=1 \\ m\neq j}}^M \frac{1}{d_{mj}^2}\right)^{\frac{1}{2}}\hspace{-0.2cm}+H_4(M-1)^{\frac{1}{2}}\right)^2\hspace{-0.2cm}\sum_{\substack{m=1 \\ m\neq j}}^M \norm{U^t}^2_{(L^2(D_m))^3}+\sum_{j=1}^MH_1^2\norm{U^t}_{(L^2(D_j))^3}^2\right]
\nonumber
\\
{\underset{\eqref{sum-i-neq-j-reciprocal-dij-k<3-mp}}{\leq}} \quad&\hspace{-0.3cm}3\,K\sum_{j=1}^M\dfrac{1}{|1-\alpha_j\,\omega^2\lambda_{n_0}^{j}|^2}\norm{U^i}_{(L^2(D_j))^3}^2\hspace{-0.1cm}+\hspace{-0.1cm}27\,K\max_{j}\{ a^{-2h_j}\}\max_{m}\{|\alpha_m\,\omega^2|^2\}\max_m\{|B_m|\}^2 a^6 \nonumber\\
&{[2(M-1)H_3^2(d_1a^{-2t}+d_1^{\circ} a^{-s})+2(M-1)^2H_4^2+H_1^2]\,\sum_{j=1}^M \norm{U^t}^2_{(L^2(D_j))^3}},\quad
\label{norm-ut-before-taking-to lhs-mp}
\end{align}
where, $K:=\underset{j}{\max}\{(1+\frac{|1-\alpha_j\omega^2\lambda_{n_0}^j|^2}{\sigma_j})\}$ and $\underset{j=1,j\neq m}{\sum^M}d_{mj}^{-2}=d_0^\circ d^{-2}+d_1^\circ a^{-s}$.

 \par Considering the behaviour of the coefficients of $\sum_{j=1}^M\norm{U^t}_{(L^2(D_j))^3}^2$ in the right part of \eqref{norm-ut-before-taking-to lhs-mp}, we can observe that for  {$s+h<1$},   $\sum_{j=1}^M\norm{U^t}_{(L^2(D_j))^3}^2$ satisfies the estimate
{
\begin{equation}\label{norm-u-app-piecewise-density-multiple-particle}
\sum_{j=1}^M\norm{U^t}_{L^2(D_j)^3}^2\lesssim \sum_{j=1}^M\dfrac{1}{|1-(\rho_{j}-\rho_{0})\omega^2\lambda_{{n_0}}^{j}|^2}\norm{U^i}_{(L^2(D_j)^3}^2,
\end{equation}}
which is nothing but the result stated due to \eqref{def-omega-choosen-mp-2*}. We can observe that the result is true for $1-h-s=0$ with {$27\,K\max_{m}\{|\tilde{\alpha}_m\,\omega^2|^2\}\max_m\{|B_m|^2\} 2\max\{H_3^2,H_4^2\}(d_1^{\circ}+1) $ is less than $1$} .
\end{proof}

%%%%%%%%%%%%%%%%%%%%%%%%%%%%%%%%%%%%%%%%%%%%%%%%%%%%%%%%%%%%%%%%%%%%%%%%%%%%%%%%%%%%%%%%%%%%%%%%%%%%%%%%

%%%%%%%%%%%%%%%%%%%%%%%%%%%%%%%%%%%%%%%%%%%%%%%%%%%%%%%%%%%%%%%%%%%%%%%%%%%%%%%%%%%%%%%%%%%%%%%%%%%%
\begin{proposition}\label{proposition-int-Ut-apriori-estimate-mp}
The integrated total field $U^t$ over inclusions satisfies the following system of equations,
 for $\frac{s}{2} \leq h \leq 1-s$ and $t \leq \frac{1}{2}$,
 {
 \begin{align}\label{int-ut-app-Dj-multiple-particle-lemma-simple}
\int_{D_j}\hspace{-0.36cm}U^t(y)dy &=\hspace{-0.13cm}\dfrac{1}{(1-\alpha_j\,\omega^2 \lambda_{n_0}^j)}\hspace{-0.1cm}\sum_{l=1}^{l_{\lambda_{n_0}^j}}
\overline{\int_{D_j}\hspace{-0.39cm}{e_{n_0,l}^j(x)dx}}\hspace{-0.06cm}\otimes\hspace{-0.15cm}{\int_{D_j}\hspace{-0.39cm}{e_{n_0,l}^j(x)dx}}\hspace{-0.05cm}\cdot\hspace{-0.16cm}\left[\hspace{-0.06cm} U^i(z_j)\hspace{-0.1cm}+\hspace{-0.15cm}\sum_{\substack{m=1 \\ m\neq j}}^M  \hspace{-0.17cm}\left(\hspace{-0.1cm}\Gamma^\omega(z_j,z_m)\hspace{-0.07cm}\cdot\hspace{-0.2cm}\int_{D_m}\hspace{-0.5cm}\alpha_m\omega^2U^t(y)dy\hspace{-0.1cm}\right)\hspace{-0.1cm}\right]\hspace{-0.18cm}+\hspace{-0.08cm}{O\hspace{-0.08cm}\left(\hspace{-0.06cm}a^{\hspace{-0.06cm}3+\min\{0,1-2h-\frac{s}{2}\}}
\right)}
 \end{align}
 }
\end{proposition}
\begin{proof}
For $x\in D_j$, consider the modified Lippmann Schwinger equation \eqref{norm-u-first-step-with-newtonian-operator-multiple-particle}, and apply the Taylor series expansion of incident field $U^i$ about $z_j$ to obtain
\begin{eqnarray}
(I-\alpha_j\,\omega^2N^0_j)U^t(x)&=&U^i(z_j)+\int_{0}^{1}{\nabla_x U^i(z_j+t(x-z_j))\cdot (x-z_j)dt}+\alpha_j\,\omega^2\int_{D_j}{(\Gamma^{\omega}-\Gamma^0)(x,y)\cdot U^t(y)dy}\nonumber\\
&&+\sum_{\substack{m=1 \\ m\neq j}}^M\int_{D_m}\alpha_m\,\omega^2\Gamma^{\omega}(x,y)\cdot U^t(y)dy.\label{before-dot-with-W-multiple-particle}
\end{eqnarray}

One can observe that $(I-\alpha_j\,\omega^2N^0_j)^{-1}$ exists under the condition \eqref{def-omega-choosen-mp}, thus we set $W^j_k$ as 
$ W_k^j:=(I-\alpha_j\,\omega^2N^0_j)^{-1}\mathtt{e_k}$
{and $W^j:=[W_1^j\; W_2^j\; W_3^j]^\top$ 
 with $\mathtt{e_k}$, $k=1,2,3$ denoting the standard unit vectors in $\mathbb{R}^3$.}

 Taking the dot product with $W^j$ from left in \eqref{before-dot-with-W-multiple-particle} and integrate over $D_j$, and making use the self adjoint property of $I-\alpha_j\,\omega^2N^0_j$, we get

\begin{eqnarray}
\int_{D_j}{U^t(x)dx}&=&\int_{D_j}{ W^j dx}\cdot U^i(z_j)+I_1 + I_2+I_3
\label{int-ut-app-multiple-particle}
\end{eqnarray}
where,
\begin{equation}\label{def-I1to3}
\left.\begin{array}{ccc}
I_1&:=&\int_{D_j}{W^j\cdot \left(\int_{0}^{1}{\nabla_x U^i(z_j+t(x-z_j))\cdot (x-z_j) dt}\right)  dx}
\\
I_2&:=&\alpha_j\,\omega^2\int_{D_j}W^j\cdot \left(\int_{D_j}{(\Gamma^{\omega}-\Gamma^0)(x,y)\cdot U^t(y)dy}\right) dx
\\
I_3&:=&\int_{D_j}W^j\cdot \left(\sum_{\substack{m=1 \\ m\neq j}}^M\int_{D_m}\alpha_m\,\omega^2\Gamma^{\omega}(x,y)\cdot U^t(y)dy\right) dx
\end{array}
\right\}.
\end{equation}

In order to estimate the terms $I_1,I_2$ and $I_3$ in the above, it can be observed that the behaviour of $\int_{D_j}{W_k^j dx}$ and $\norm{W_k^j}_{(L^2(D_j))^3}
$ for $k=1,2,3$, are required.
To achieve these, making use the facts that $N^0_j$ is a self adjoint operator with orthonormal eigen system $(\lambda_n^{j},e_n^{j})$ and the definition of $W_k^j$, rewrite $\int_{D_j}{e_n^{j}(x)dx}$ as follows;
\begin{align}
\int_{D_j}{e_n^{j}(x)dx}
=\,&\int_{D_j}{\bigg( (I-\alpha_j\,\omega^2 N^0_j)W_1^j\cdot e_n^{j}(x),(I-\alpha_j\,\omega^2 N^0_j)W_2^j\cdot e_n^{j}(x),(I-\alpha_j\,\omega^2 N^0_j)W_3^j\cdot e_n^{j}(x)\bigg)^\top dx}
\nonumber\\
=\,& (1-\alpha_j\,\omega^2 \lambda_n^{j})\int_{D_j}{W^j\cdot e_n^{j}(x) dx},\label{int-e-n-Dj-value-multiple-particle}
\end{align}
which further gives
\begin{eqnarray}
\dfrac{1}{(1-\alpha_j\,\omega^2 \lambda_n^{j})}\int_{D_j}{ \mathtt{e_k}\cdot e_n^{j}(x) dx}=\int_{D_j}{W_k^j\cdot e_n^{j}(x) dx}=\langle{W_k^j\,;\,\overline{e_n^{j}}}\rangle_j,\quad k=1,2,3.\label{inner-product-enD-W1-multiple-particle}
\end{eqnarray}
The above equality will lead us to study the behaviour of $\int_{D_j}{W_k^j dx}$ and $\norm{W_k^j}_{(L^2(D_j))^3}
$ for $k=1,2,3$, as follows;
\begin{itemize}
\item  Primarily, observe
\begin{align}\label{int-W1-approximation-multiple-particle-component*}
\left(\int_{D_j}W_k^jdx\right)_i
=\;\;&\int_{D_j}W_k^jdx\cdot \mathtt{e_i}
\,=\,\left[\sum_{n}{\langle W_k^j\,;\,e_n^{j}\rangle_j\int_{D_j}{e_n^{j}(x)dx}}\right]\cdot \mathtt{e_i} \mbox{ for } k,i=1,2,3, 
\nonumber\\
{\underset{\eqref{inner-product-enD-W1-multiple-particle}}{=}}&{\dfrac{1}{(1-\alpha_j\,\omega^2 \lambda_{n_0}^j)}\sum_{l=1}^{l_{\lambda_{n_0}^j}}\langle{\mathtt{e_k}\,;\,e_{n_0,l}^j}\rangle_j\;\langle{e_{n_0,l}^j\,;\,\mathtt{e_i}}\rangle_j}
+\sum_{n\neq {n_0}}\dfrac{1}{(1-\alpha_j\,\omega^2 \lambda_n^{j})}\sum_{l=1}^{l_{\lambda_n^j}}\langle{\mathtt{e_k}\,;\,e_{n,l}^j}\rangle_j\;\langle{\mathtt{e_i}\,;\,e_{n,l}^j}\rangle_j
\nonumber\\
{=}\;\;& \dfrac{1}{(1-\alpha_j\,\omega^2 \lambda_{n_0}^j)}{\sum_{l=1}^{l_{\lambda{n_0}^j}}\langle{\mathtt{e_k}\,;\,e_{n_0,l}^j}\rangle_j\;\langle{e_{n_0,l}^j\,;\,\mathtt{e_i}}\rangle_j}+O(a^3),
\end{align}
and hence we have,
\begin{eqnarray}
{\int_{D_j}{
\hspace{-0.1cm}W_k^j dx}=\dfrac{1}{(1-\alpha_j\,\omega^2 \lambda_{n_0}^j)}{\sum_{l=1}^{l_{\lambda_{n_0}^j}}\langle{\mathtt{e_k}\,;\,e_{n_0,l}^j}\rangle_j\;\int_{D_j}{e_{n_0,l}^j(x)dx}}+O( a^3)\label{int-W1-approximation-multiple-particle}, \quad k=1,2,3,}\label{int-W-k-j-app-1}
\end{eqnarray}
which implies
\begin{eqnarray}
\int_{D_j}{W_k^j dx}{\underset{\eqref{def-omega-choosen-mp-2*}, \eqref{def-alpham}}{=}}O\bigg(a^{3-h}\bigg).\label{int-W1-approximation-multiple-particle*}
\end{eqnarray}
Here, the last equality in \eqref{int-W1-approximation-multiple-particle-component*} is due to the behaviour of $\underset{n\neq {n_0}}{\sum}\dfrac{1}{(1-\alpha_j\,\omega^2 \lambda_n^{j})}\langle{\mathtt{e_k}\,;\,e_{n}^j}\rangle_j\;\langle{\mathtt{e_l}\,;\,e_{n}^j}\rangle_j$
  as $O( a^3)$, which we have from \cite[Lemma 4.1]{challa2023extraction} and $\{e^j_{{n_0,l}}, \,l=1,\cdots, l_{\lambda_{n_0}^j}\}$ are basis of eigenspace corresponding to the eigen value $\lambda_{n_0}^j$ which is denoted by $E_{\lambda_{n_0}^j}$, so $dim(E_{\lambda_{n_0}^j})=l_{\lambda_{n_0}^j}$, a finite number.

\item  Making use of Parseval's identity, we have
\begin{align}
\norm{W_k^j}_{(L^2({D_j}))^3}^2
=\;\;&\sum_{n}{|\,\langle\,W_k^j\,;\,e_{n}^j\,\rangle_j\,|}^2
\nonumber\\
{\underset{\eqref{inner-product-enD-W1-multiple-particle}}{=}}&\dfrac{1}{|1-\alpha_j\,\omega^2 \lambda_{n_0}^j|^2}|\,{\sum_{l=1}^{l_{\lambda_{n_0}^j}}\langle\,\mathtt{e_k}\,;\,e_{n_0,l}^j\,\rangle_j}\,|^2+\sum_{n\neq {n_0}}{\dfrac{1}{|1-\alpha_j\,\omega^2 \lambda_{n}^j|^2}|\sum_{l=1}^{l_{\lambda_n^j}}\langle\,\mathtt{e_k}\,;\,e_{n,l}^{j}\,\rangle_j\,|^2}
\nonumber\\
\underset{\text{CSI}}{\leq}
\;&\dfrac{1}{|1-\alpha_j\,\omega^2 \lambda_{n_0}^j|^2}\norm{\mathtt{e_k}}^2_{(L^2(D_j))^3}{\sum_{l=1}^{l_{\lambda_{n_0}^j}}\norm{e_{n_0,l}^{j}}_{(L^2(D_j))^3}^2}+\sum_{n\neq {n_0}}{\dfrac{1}{|1-\alpha_j\,\omega^2 \lambda_{n}^{j}|^2}|\sum_{l=1}^{l_{\lambda_n^j}}\langle\,\mathtt{e_k}\,;\,e_{n,l}^{j}\,\rangle_j\,|^2}\qquad\quad
\nonumber\\
&\hspace{-1cm}\underset{(\ref{def-omega-choosen-mp-2*},\ref{def-alpham})}{\leq}
 O( a^{3-2h_j})+O( a^3)
\nonumber
\\
\mbox{ i.e.,}  \norm{W_k^j}_{(L^2(D_j))^3}=\;\;&O( a^{\frac{3}{2}-h_j}),\quad k=1,2,3\label{norm-W-multiple-particle}.
\end{align}
\end{itemize}
Using \eqref{inner-product-enD-W1-multiple-particle} and \eqref{norm-W-multiple-particle}, we approximate $I_1$, $I_2$ and $I_3$ in \eqref{int-ut-app-multiple-particle}. 
\begin{enumerate}
\item To approximate $I_1$, consider, 
\begin{align}
|I_1|^2
\underset{{\eqref{def-I1to3},\,{CSI}}}{\leq}& \sum_{k=1}^3 \norm{\int_{0}^{1}{\nabla U^i(z_j+t(\cdot-z_j))\cdot (\cdot-z_j)dt}}_{(L^2(D_j))^3}^2\norm{W_k^j}^2_{(L^2(D_j))^3}
\nonumber\\
\leq \quad\;\;&\hspace{-0.4cm}\int_{D_j}\sum_{l=1}^3|\int_{0}^{1}{\hspace{-0.2cm} \nabla U_l^i(z_j+t(x-z_j))\cdot (x-z_j)|dt}|^2dx\,\sum_{k=1}^3\norm{W_k^j}_{(L^2(D_j))^3}^2
\nonumber\\
\underset{\eqref{norm-W-multiple-particle}}{=}\quad& O(a^{2+3}a^{3-2h_j})\nonumber
\\
\mbox{ i.e.,}\;\,  I_1
=\quad\;\;&O\left( a^{4-h} \right).\label{I_1-app-multiple-particle}   
\end{align}

\item To approximate $I_2$, consider,
\begin{align}
|I_2|^2
\underset{{\eqref{def-I1to3},\,{CSI}}}{\leq}& |\alpha_j\,\omega^2|^2\sum_{k=1}^3\norm{\int_{D_j}{(\Gamma^{\omega}-\Gamma^0)(\cdot,y)\cdot U^t(y)dy}}^2_{(L^2(D_j))^3}\norm{W_k^j}^2_{(L^2(D_j))^3}
\nonumber\\
\underset{CSI}{\leq}\quad\,&|\alpha_j\,\omega^2|^2\norm{U^t}_{(L^2(D_j))^3}^2 \sum_{l=1}^{3}
\sum_{i=1}^3\int_{D_j}\int_{D_j}|(\Gamma^\omega-\Gamma^0)_{il}(x,y)|^2dx\,\sum_{k=1}^3\norm{W_k^j}^2_{(L^2(D_j))^3}
\nonumber\\
\underset{\eqref{Gamma-Gammm0-bounded-x,y-in-D-H1-value}}{\leq}\quad& |\alpha_j\,\omega^2|^2{\norm{U^t}_{(L^2(D_j))^3}^2}9{H_1}^2 a^6|B_j|^2\sum_{k=1}^3\norm{W_k^j}^2_{(L^2(D_j))^3}\underset{\eqref{def-alpham},\eqref{norm-W-multiple-particle},\eqref{norm-u-app-piecewise-density-multiple-particle-lemma}}{=}O(a^{-4+{3-s-2h}+6+3-2h})
\nonumber\\
\mbox{i.e., }\; I_2\;{=}\quad\;& \,O\left(a^{4-2h-{\frac{s}{2}}}\right).
\label{I-2-app-multiple-particle}
\end{align}
Here, $H_1$ is a constant satisfying $|(\Gamma^{\omega}-\Gamma^0)_{il}(x,y)|\leq H_1$ for $x,y\in D_j$, $j=1,\cdots,M$ and $i,l=1,2,3$, see \eqref{Gamma-Gammm0-bounded-x,y-in-D-H1-value} of Lemma \ref{lemma-Gamma-propertie-Dl-mp} in Appendix.
\medskip\\
\item To approximate $I_3$, 
 first apply Taylor series expansion for $\Gamma^\omega(x,y)$ about $x\in D_j$ near $z_j$ and again about $y\in D_m, m\neq j$ near $z_m$, in the definition \eqref{def-I1to3} of $I_3$, to get
\begin{eqnarray}
I_3=\int_{D_j}W^jdx\cdot \sum_{\substack{m=1 \\ m\neq j}}^M\left(\Gamma^\omega(z_j,z_m)\cdot\int_{D_m}\,\alpha_m\omega^2U^t(y)dy\right)+S_1+S_2\label{I3-terms-gamma-zi-zj-multiple-particle}
\end{eqnarray}
where,
\begin{align}\label{def-S1to2}
\left.\begin{array}{ccc}
S_1:=\int_{D_j}W^j\cdot \underset{\substack{m=1 \\ m\neq j}}{\sum^M}\int_{D_m}\,\alpha_m\omega^2\int\limits_{0}^1\nabla_y \Gamma^\omega(z_j,z_m+t(y-z_m))\cdot (y-z_m)dt\cdot U^t(y)dy\, dx
\\ ~ \\
S_2:=\int_{D_j}W^j\cdot\underset{\substack{m=1 \\ m\neq j}}{\sum^M}  \int_{D_m}\,\alpha_m\omega^2\int\limits_{0}^1\nabla_x \Gamma^\omega(z_j+t(x-z_j),\,y)\cdot (x-z_j)dt\cdot U^t(y)dy\, dx\quad
\end{array}
\right\}.
\end{align}
Now to approximate the terms $S_1$ and $S_2$ appearing in \eqref{I3-terms-gamma-zi-zj-multiple-particle}, consider
\begin{align}
 \quad|S_1|^2=\;&\sum_{k=1}^3|\int_{D_j}W_k^j\cdot \left(\sum_{\substack{m=1 \\ m\neq j}}^M\int_{D_m}\hspace{-0.4cm}\,\alpha_m\omega^2\int_{0}^1\hspace{-0.2cm}\nabla_y \Gamma^\omega(z_j,z_m+t(y-z_m))\cdot (y-z_m)dt\cdot U^t(y)dy\right)dx|^2
\nonumber
\\
\underset{\text{CSI}}{\leq}& \sum_{k=1}^3\left(\sum_{\substack{m=1 \\ m\neq j}}^M\norm{\int_{D_m}\hspace{-0.2cm}\,\alpha_m\omega^2\int_{0}^1\hspace{-0.2cm}\nabla_y \Gamma^\omega(z_j,z_m+t(y-z_m))\cdot (y-z_m)dt\cdot U^t(y)dy}_{(L^2(D_j))^3}\right)^2\norm{W_k^j}_{(L^2(D_j))^3}^2
\nonumber\\
\underset{CSI}{\leq} &\sum_{k=1}^3\hspace{-0.06cm}\left(\sum_{\substack{m=1 \\ m\neq j}}^M\hspace{-0.1cm}\left(\int_{D_j}\hspace{-0.1cm}\sum_{i=1}^3 \norm{\,\alpha_m\omega^2U^t}_{(L^2(D_m))^3}^2\norm{\hspace{-0.08cm}\int_{0}^1\hspace{-0.3cm}\nabla_y \Gamma^\omega_i(z_j,z_m+t(\cdot-z_m))\cdot (\cdot-z_m)dt}_{(L^2(D_m))^3}^2\hspace{-0.1cm}\right)^{\hspace{-0.1cm}\frac{1}{2}} \hspace{-0.08cm}\right)^{\hspace{-0.1cm}2}\hspace{-0.18cm}\norm{W_k^j}_{(L^2(D_j))^3}^2
\nonumber\\
&\hspace{-1.5cm}\underset{\eqref{grad-Gamma-x-Di-y-Dj-homogeneous-background-mp},\eqref{def-alpham},\eqref{norm-W-multiple-particle}}{=}O\left(a^{{7}-2h_j}\left[H_6 (M-1)^{1/2}+H_5\left(\sum_{\substack{m=1 \\ m\neq j}}^M\dfrac{1}{d_{mj}^4}\right)^{1/2}\right]^2 \left(\sum_{\substack{m=1 \\ m\neq j}}^M\norm{U^t}_{(L^2(D_m))^3}^2\right) \right).
\nonumber
\\
&\hspace{-0.95cm}S_1{\underset{\eqref{sum-i-neq-j-reciprocal-dij-k>3-mp}}{=}}O\left( a^{\frac{7}{2}-h_j}[(M-1)^{1/2}+{a^{-2t}}]\left(\sum_{\substack{m=1 \\ m\neq j}}^M\norm{U^t}_{(L^2(D_m))^3}^2\right)^{1/2}\right)
\nonumber\\
\end{align}
Here, the last estimate is due to H\"older inequality and above mentioned $H_5$ and $ H_6$ are constants satisfying $|\nabla_y\Gamma_{kl}^\omega(z_j,z_m+t(y-z_m))|\leq \left(H_6+\dfrac{H_5}{d_{mj}^2}\right)^2$ for $y\in D_m$, see \eqref{grad-Gamma-x-Di-y-Dj-homogeneous-background-mp} of Lemma \ref{lemma-gamma-in-Dp-Dq-p-ne-q} in Appendix. Hence, we obtain the behaviour of $S_1$ as
\begin{align} 
S_1\underset{\eqref{norm-u-app-piecewise-density-multiple-particle-lemma}}{=} &O\left(a^{5-2h-\frac{s}{2}-{\max\{2t,s/2\}}}\right).\label{S-1-app-multiple-particle-for-I-3}
\end{align}
We can observe the behaviour of $S_2$ in a manner similar to how we observed the behaviour of $S_1$, and it is given by
{
\begin{eqnarray}
S_2
&= &O\left(a^{5-2h-\frac{s}{2}-{2t}}\right)\label{S-2-app-multiple-particle-for-I-3}.
\end{eqnarray}}
Substituting the  approximations \eqref{S-1-app-multiple-particle-for-I-3} and \eqref{S-2-app-multiple-particle-for-I-3} of $S_1$ and $S_2$ in \eqref{I3-terms-gamma-zi-zj-multiple-particle}, we obtain
\begin{eqnarray}
I_3=\sum_{\substack{m=1 \\ m\neq j}}^M\int_{D_j}W^jdx\cdot \left(\Gamma^\omega(z_j,z_m)\cdot\int_{D_m}\,\alpha_m\omega^2U^t(y)dy\right)
+{O\left(a^{5-2h-\frac{s}{2}-{2t}} \right)}
.\label{I3-app-multiple-particle}
\end{eqnarray}
\end{enumerate}
Substituting \eqref{I_1-app-multiple-particle}, \eqref{I-2-app-multiple-particle} and \eqref{I3-app-multiple-particle} in \eqref{int-ut-app-multiple-particle} we get the desired result \eqref{int-ut-app-Dj-multiple-particle-lemma-simple}.
\begin{align}
\int_{D_j}{\hspace{-0.3cm}U^t(x)dx}=\;\;& 
\int_{D_j}\hspace{-0.2cm}W^jdx\cdot \left[ U^i(z_j)+\sum_{\substack{m=1 \\ m\neq j}}^M\left(\Gamma^\omega(z_j,z_m)\cdot\int_{D_m}\hspace{-0.3cm}\alpha_m\omega^2U^t(y)dy\right)\right]+O(a^{\min\{4-h,4-2h-\frac{s}{2},5-2h-\frac{s}{2}-2t\}})
\nonumber\\
\underset{\eqref{int-W-k-j-app-1}}{=}&\dfrac{1}{(1-\alpha_j\,\omega^2 \lambda_{n_0}^j)}\sum_{l=1}^{l_{\lambda_{n_0}^j}}\overline{\int_{D_j}\hspace{-0.33cm}{e_{n_0,l}^j(x)dx}}\otimes\hspace{-0.1cm}\int_{D_j}\hspace{-0.33cm}{e_{n_0,l}^j(x)dx}\cdot \left[ U^i(z_j)+\sum_{\substack{m=1 \\ m\neq j}}^M \left(\Gamma^\omega(z_j,z_m)\cdot\hspace{-0.15cm}\int_{D_m}\hspace{-0.4cm}\alpha_m\omega^2U^t(y)dy\right)\right]
\nonumber\\
&
+O(a^3)+O(a^{4-h-s})+{O\left(a^{4-2h-\frac{s}{2}+\min\{0,1-{2t}\}} \right)}
\nonumber\\
=\;\;&\dfrac{1}{(1-\alpha_j\,\omega^2 \lambda_{n_0}^j)}
\sum_{l=1}^{l_{\lambda_{n_0}^j}}\overline{\int_{D_j}\hspace{-0.33cm}{e_{n_0,l}^j(x)dx}}\otimes\hspace{-0.1cm}\int_{D_j}\hspace{-0.33cm}{e_{n_0,l}^j(x)dx}\cdot\hspace{-0.15cm}\left[ U^i(z_j)+\hspace{-0.1cm}\sum_{\substack{m=1 \\ m\neq j}}^M  \hspace{-0.1cm}\left(\hspace{-0.1cm}\Gamma^\omega(z_j,z_m)\cdot\hspace{-0.1cm}\int_{D_m}\hspace{-0.4cm}\alpha_m\omega^2U^t(y)dy\hspace{-0.1cm}\right)\right]
\nonumber\\
&
+{O\left(a^{3+\min\{0,1-2h-\frac{s}{2}\}}\right),\mbox{ with $h\geq \frac{s}{2}$ and $t\leq \frac{1}{2}$}} .
\end{align}
%{This is because $O\left(a^{3+\min\{0,1-2h-\frac{s}{2}+\min\{0,h-\frac{s}{2},1-2t\}\}}\right)
%=O\left(a^{3+\min\{0,1-2h-\frac{s}{2}\}}\right)$, with $h\geq \frac{s}{2}$ and $t\leq \frac{1}{2}$. THIS SENTENCE CAN BE REMOVED}
\end{proof}

\subsection{The Algebraic system and its solvability }\hfill

By introducing  $Q_{j}:=\int_{D_j}U^t(y)dy$ and
\begin{eqnarray}\label{C-j-matrix-definition}
C^{(j)}:={\dfrac{1}{(1-\alpha_j\,\omega^2 \lambda_{n_0}^j)}\sum_{l=1}^{l_{\lambda_{n_0}^j}}\overline{\int_{D_j}\hspace{-0.2cm}{e_{n_0,l}^j(x)dx}}\otimes\int_{D_j}\hspace{-0.2cm}{e_{n_0,l}^j(x)dx}}
\end{eqnarray} 
  for $j=1,2,\cdots,M$, the system of equations \eqref{int-ut-app-Dj-multiple-particle-lemma-simple} can be compactly presented as follows:\begin{eqnarray}
Q_{j}=C^{(j)}\cdot U^i(z_j)+\sum_{\substack{m=1 \\ m\neq j}}^M\,\alpha_m\omega^2C^{(j)}\cdot \Gamma^\omega(z_j,z_m)\cdot  Q_{m}+{O\left(a^{3+\min\{0,1-2h-\frac{s}{2}\}}
\right)},\; j=1,2,\cdots,M.\nonumber
\end{eqnarray}
The corresponding system in matrix representation is as follows,
\begin{eqnarray}
(\mathtt{I}-B)Q=U^{I}+\mathtt{Err}\label{matrix-represention-for-int-ut-app-C-need-not-be-invertible}
\end{eqnarray}
where,
  $\mathtt{Err}:=({O\left(a^{3+\min\{0,1-2h-\frac{s}{2}\}}
\right)})_{j=1}^{3M}$, 
  $Q:=(Q_{j})_{j=1}^M$ and    $U^{\mathtt{i}}:=(C^{(j)}\cdot U^i(z_j))_{j=1}^M$  are  vectors of size $3M $,  $\mathtt{I}$ and $B$  are matrices of size ${3M\times 3M}$ with $\mathtt{I}$ denoting the identity matrix,
   and $B$ defined as 
   
  \begin{align}\label{def-alg-sys-Matr}
B\hspace{-0.1cm}:=\hspace{-0.08cm}\omega^2
\hspace{-0.1cm}\begin{bmatrix}\mathtt{0} & \alpha_2C^{(1)}\, \Gamma^\omega(z_1,z_2) 
& \alpha_3C^{(1)}\,\Gamma^\omega(z_1,z_3)
& \cdots \hspace{-0.2cm}&\hspace{-1cm}\alpha_MC^{(1)}\,\Gamma^\omega(z_1,z_M)
 \\
 &&&&\\
 \alpha_1\,C^{(2)}\,\Gamma^\omega(z_2,z_1) & \mathtt{0} &\alpha_3C^{(2)}\,\Gamma^\omega(z_2,z_3)
 & \cdots 
 &\hspace{-1cm}\alpha_MC^{(2)}\,\Gamma^\omega(z_2,z_M)\\
 &&&&\\
%\alpha_1C^{(3)}\,\Gamma^\omega(z_3,z_1) & \alpha_2C^{(3)}\,\Gamma^\omega(z_3,z_2)& \mathtt{0} &\cdots &\hspace{-1cm}\alpha_MC^{(3)}\,\Gamma^\omega(z_3,z_M)\\
\cdots &\cdots&\cdots&\cdots&\cdots\\
\alpha_1C^{( {M})}\,\Gamma^\omega(z_M,z_1) & \alpha_2C^{( {M})}\,\Gamma^\omega(z_M,z_2) &\cdots 
&\alpha_{M-1}C^{( {M})}\,\Gamma^\omega(z_M,z_{M-1})
& \mathtt{0}
\end{bmatrix}\qquad
\end{align}
\begin{align}\label{def-alg-sys-vect-Q}
Q:=({Q}_{1}^\top,\,{Q}_{2}^\top,\cdots, Q_{m}^\top), \quad \mbox{nothing but}\;\, Q(l)=\int_{D_j}U^t_k(y)dy,\;\, l=3(j-1)+k, j=1,2,\cdots, M \, \mbox{and}\; k=1,2,3
\end{align}
and
\begin{align}\label{def-alg-sys-vect-UI}
U^{I}:=\begin{bmatrix}
\left(C^{(1)}\cdot U^i(z_1)\right)^\top &  
\left(C^{(2)}\cdot U^i(z_2)\right)^\top &
\cdots &
\left(C^{( {M})}\cdot U^i(z_M)\right)^\top
\end{bmatrix}^\top=\mathtt{C}\cdot \bar{U},
\end{align}
where $\mathtt{C}_{3M\times 3M}$ is a tridiagonal matrix with $C^{(j)}$, $j=1,2,\cdots, M$, as diagonal block matrices, and $\bar{U}:=(U^i(z_j))_{j=1}^M$.
Here $\mathtt{0}$ denotes zero matrix of size $3\times 3$.
{
Furthermore, we observe the following behaviour
\begin{align}\label{C-j-matrix-behaviour}
C^{(j)}\underset{\eqref{eigen-function-scaling-mp},\eqref{def-omega-choosen-mp}}{=}O(a^{3-h_j}).
\end{align}
}
The aforementioned linear algebraic system is solvable for the constant vectors $Q_{j}$ , $1 \leq j \leq M$, when the matrix $\mathtt{I}-B$ is non-singular. One such possible scenario is $\norm{B}<1$ and it is true for the case $1-h-{s}> 0$. We will discuss its non-singularity in the following proposition. 

Let $\bar{Q}$ be a solution vector of the unperturbed linear system corresponding to \eqref{matrix-represention-for-int-ut-app-C-need-not-be-invertible}, i.e., $\bar{Q}$ satisfies the following linear algebraic system:
\begin{eqnarray}
(\mathtt{I}-B)\cdot \bar{Q}=U^{I}.\label{perturbed-linear-system-for-int-ut-app}
\end{eqnarray}
%%%%%%%%%%%%%%%%%%%%%%%%%%%%%%%%%%%%%%%%%%%%%%%%%%%%%%%%%%%%%%%%%%%%%%%%%%%%%%%%%%%%%%%%%%%%%%%%%

\begin{proposition}\label{lemma-algebraic-system-solvability}
The matrix $(\mathtt{I}-B)$ is invertible whenever $1-h-{s}\geq 0$.
\end{proposition}
\begin{proof}
Let us discuss the invertibility of $(\mathtt{I}-B)$ by deriving a  sufficient conditions on the denseness of the inclusions such that the matrix $B$ satisfies $\norm{B}<1$. To derive this, let us arbitrarily choose a row $p$, $1\leq p\leq 3M$, from the matrix $B$ and consider the sum $\sum_{q=1}^{3M}|B_{pq}|$. Observe that, for any fixed positive integers $p$ and $q$ not exceeding $3M$, there exists integers $j,m,k$ and $l$ such that $1\leq j,m\leq M$ and $1\leq k,l\leq 3$ satisfying $p=3(m-1)+l$ and $q=3(j-1)+k$. Thus,
\begin{align}
\sum_{q=1}^{3M}|B_{pq}|=\;\;&\sum_{j=1}^{M}\left[\sum_{k=1}^3|B_{(3(m-1)+l)(3(j-1)+k)}|\right]
\nonumber\\
=\;\;&\sum_{\substack{j=1\\j\neq m}}^{M}\left[\sum_{k=1}^3|\alpha_j\omega^2C^{(m)}_l\cdot \Gamma_k^\omega(z_m,z_j)|\right]
\nonumber\\
\leq\;\;& \sum_{\substack{j=1\\j\neq m}}^{M}\left[|\alpha_j\omega^2| |C_{l}^{(m)}|\;\sum_{k=1}^3\left( \sum_{i=1}^3|\Gamma^\omega_{ki}(z_m,z_j)|^2\right)^{1/2}\right]
\nonumber\\
\underset{\eqref{Gamma-is-bounded-x-in-Di-y-in-Dj-mp}}{\leq}\,&\sum_{\substack{j=1\\j\neq m}}^{M}\left[|\alpha_j\omega^2|  |C_{l}^{(m)}|\;\sum_{k=1}^3\left( \sum_{i=1}^3\left(\frac{H_3}{d_{mj}}+H_4\right)^2\right)^{1/2}\right]
\nonumber\\
&\hspace{-1cm}{\underset{(\ref{def-alpham},\ref{int-W1-approximation-multiple-particle*})}{=}}O\bigg(a^{-2+3-h}\sum_{\substack{j=1\\j\neq m}}^{M}\left(\frac{H_3}{d_{mj}}+H_4\right)\bigg)
\nonumber\\
 \underset{\eqref{sum-i-neq-j-reciprocal-dij-k<3-mp}}{=}&O(a^{1-h-s})\nonumber\\
\therefore\;\;\norm{B}_\infty=\;\;&\max_{i=1,\cdots,3M}\{\sum_{j=1}^{3M}{|B_{ij}|}\}
\,=\,O(a^{1-h-s}).\label{norm-B-estimation-mp}
\end{align}
Thus $(I-B)$ is invertible whenever $1-h-s>0$. In addition, we can also prove that $(I-B)$ is invertible when $1-h-s=0$ and {$\max_j\{|\tilde{\alpha}_j\omega^2|\,\norm{\tilde{C}^{{(j)}}}_{\infty}\}3\sqrt{3}(d_1+1)\max\{H_3,H_4\}<1$}, where $\tilde{\alpha}_j$ and $\tilde{C}^{(j)}$
satisfies 
\begin{align}\label{tilde-gamma-tilde-C-definition}
    \alpha_j=(c_j a^{-2}-\rho_0)=\tilde{\alpha}_ja^{-2},\;\; C^{(j)}\underset{\eqref{eigen-function-scaling-mp}}{=}\dfrac{a^3}{(1-\alpha_j\,\omega^2 \lambda_{n_0}^j)}\sum_{l=1}^{l_{\lambda_{n_0}^j}}\overline{\int_{B_j}\hspace{-0.2cm}{e_{n_0,l}^j(\xi)d\xi}}\otimes\hspace{-0.1cm}\int_{B_j}\hspace{-0.33cm}{e_{n_0,l}^j(\xi)d\xi}\sim \tilde{C}^{(j)}a^{3-h_j}
\end{align}
and from \eqref{sum-i-neq-j-reciprocal-dij-k<3-mp}
$\underset{j=1,j\neq m}{\sum^M}|z_j-z_m|^{-1}=d_0a^{-t}+d_1a^{-s}$.
\end{proof}
%%%%%%%%%%%
%%%%%

%%%%%%%%%%%%%%%%%%%%%%%%%%%%%%%%%%%%%%%%%%%%%%%%%%%%%%%%%%%%%%%%%%%%%%%%%%%%%%%%%%%%%%%%%%%%%%%%%%%%%
%%%%%%%%%%%%%%%%%%%%%%%%%%%%%%%%%%%%%%%%%%%%%%%%%%%%%%%%%%%%%%%%%
 \subsection{ Estimation of Scattered field (End of the proof of Theorem \ref{theorem-mp})}\label{section-ending-proof} \hfill
  
Substitute apriori estimate, \eqref{norm-u-app-piecewise-density-multiple-particle} in \eqref{us-app-before-substituting-norm-u-integral-u-multiple-particle} to get scattered field approximation, for ${\, x\in \mathbb{R}^3\setminus\bar{D}}$, as
\begin{align}
U^s(x)=\;&\sum_{j=1}^M \alpha_j\omega^2\,\Gamma^{\omega}(x,z_j)\cdot Q_{j}+O\left( a^{2-s-h}\right)
\nonumber\\
  &\hspace{-1cm}\underset{(\ref{def-alpham},\ref{T1-estimation-mp})}{=}\sum_{j=1}^M \alpha_j\omega^2\,\Gamma^{\omega}(x,z_j)\cdot \bar{Q}_{j}+O\bigg(a^{{1-s+\min\{0,1-2h-\frac{s}{2}\}}}\bigg)+O(a^{2-s-h})
  \nonumber\\
  =\;&\sum_{j=1}^M \alpha_j\omega^2\,\Gamma^{\omega}(x,z_j)\cdot \bar{Q}_{j}+O\bigg(a^{{1-s+\min\{0,1-2h-\frac{s}{2}\}}}\bigg).\label{scattered-field-in-terms-of-unperturbed-solution}
\end{align}

Here, the last step is due to \eqref{def-alpham} and the following two observations,
\begin{eqnarray}
\setcounter{mysubequations}{0}
   \hspace{-0.3cm}   \mysubnumber \qquad \norm{Q-\bar{Q}}_{\infty}&\underset{\eqref{matrix-represention-for-int-ut-app-C-need-not-be-invertible}, \eqref{perturbed-linear-system-for-int-ut-app},\eqref{norm-B-estimation-mp}}{\leq}&\dfrac{1}{1-\norm{B}} \norm{\mathtt{Err}}_{\infty}
{\underset{\eqref{matrix-represention-for-int-ut-app-C-need-not-be-invertible}}{=}}{O(a^{3+\min\{0,1-2h-\frac{s}{2}\}})}\qquad \qquad \qquad\label{norm-Q-tildwQ-estimation-mp}
\end{eqnarray}
\begin{align}
\mysubnumber
\quad
 \vert\sum_{j=1}^M \alpha_j\omega^2\,\Gamma^{\omega}(x,z_j)\cdot\bigg( Q_{j}-\bar{Q}_{j}\bigg)\vert^2
=\;\,&\sum_{k=1}^3|\left(\sum_{j=1}^M \alpha_j\omega^2\,\Gamma^{\omega}_k(x,z_j)\cdot (Q_{j}-\bar{Q}_{j})\right)|^2\qquad\qquad\qquad
\nonumber\\
\underset{{CSI}}{\lesssim}& \max\limits_{j=1}^M\{|\alpha_j\omega^2|^2\}\sum_{k=1}^3\left(\sum_{j=1}^M |\Gamma^{\omega}_k(x,z_j)| |Q_{j}-\bar{Q}_{j}|\right)^2
\nonumber\\
&\hspace{-1cm}\underset{(\ref{def-alpham},\ref{norm-Q-tildwQ-estimation-mp})}{=} O\bigg(a^{{1-s+\min\{0,1-2h-\frac{s}{2}\}}}\bigg).\label{T1-estimation-mp}
\end{align}
By following the same process as done in the previous sections to approximate the scattered field $U^s$, we can deduce the estimation of the farfield $U^\infty$ corresponding to the scattered field $U^s$,  and is given by
\begin{eqnarray}
U^\infty(\hat{x},\theta)=\sum_{j=1}^M\alpha_j\omega^2\Gamma^\infty(\hat{x},z_j)\cdot \bar{Q}_{j}++O\bigg(a^{{1-s+\min\{0,1-2h-\frac{s}{2}\}}}\bigg),\label{U-infty-estimation-with-Gamma-infty-term-mp}
\end{eqnarray}
which holds uniformly for all observation and incidence directions $\hat{x}$ and $\theta$ in $\mathbb{S}^2
$ respectively.
By making use of the asymptotic  expansion \eqref{assymptotic-expansion-Gamma-omega} of fundamental matrix $\Gamma^\omega$  and the definition of its farfield $\Gamma^\infty$ from \eqref{p,s-parts-of farfield-gamma} in \eqref{U-infty-estimation-with-Gamma-infty-term-mp}, we get
\begin{eqnarray}
U^\infty(\hat{x},\theta)\cdot 4\pi(\lambda+2\mu)\beta_1\hat{x}=\sum_{j=1}^M\alpha_j\omega^2U^i_{\mathtt{p}}(z_j,-\hat{x})\cdot \bar{Q}_{j}+O\bigg(a^{{1-s+\min\{0,1-2h-\frac{s}{2}\}}}\bigg)\label{sct-field-Main-Thm-proof-pin}\\ 
\mbox{and}\;\;\;\;U^\infty(\hat{x},\theta)\cdot 4\pi\mu\beta_2\hat{x}^\perp=\sum_{j=1}^M\alpha_j\omega^2U^i_{\mathtt{s}}(z_j,-\hat{x})\cdot \bar{Q}_{j}+O\bigg(a^{{1-s+\min\{0,1-2h-\frac{s}{2}\}}}\bigg).\label{sct-field-Main-Thm-proof-sin}
\end{eqnarray}
  Observe that \eqref{scattered-field-in-terms-of-unperturbed-solution} is nothing but the asymptotic expansion  \eqref{sct-field-Main-Thm} of the scattered field which we tried to derive, and the sum of \eqref{sct-field-Main-Thm-proof-pin} and \eqref{sct-field-Main-Thm-proof-sin} gives the required asymptotic expansion \eqref{sct-field-Main-Thm} of the farfield. In a similar manner, utilizing \eqref{assymptotic-expansion-Gamma-omega} and \eqref{p,s-parts-of farfield-gamma} in \eqref{U-infty-estimation-with-Gamma-infty-term-mp}, we obtain the desired results \eqref{p-sct-field-Main-Thm-up-infty} and \eqref{s-sct-field-Main-Thm-us-infty}. Hence the proof of Theorem \eqref{theorem-mp} is completed. \qed

%%%%%%%%%%%%%%%%%%%%%%%%%%%%%%%%%%%%%%%%%%%%%%%%%%%%%%%%%%%%%%%%%%%%%%%%%%%%%%%%%%%%%%%%%%%%%%%%%%%%%%%%%%%%%%%%

\section{Proof of Corollary \ref{corollary-N-interactions}}

\begin{proof}

Given $\bar{Q}^N$, $N\in\mathbb{N}$ denoting the truncated solution of order $N$ of the unperturbed problem \eqref{perturbed-linear-system-for-int-ut-app}, i.e., 
\begin{eqnarray}
\bar{Q}^N:=\sum_{n=0}^N B^n\cdot U^I,
\end{eqnarray}
and the scattered field due to the $N^{th}$ level scattering by 
\begin{eqnarray}
U^{s,N}(x):=\sum_{j=1}^M \alpha_j\omega^2\Gamma^\omega(x,z_j)\cdot \bar{Q}_{j}^N\label{scattered-field-for-truncated-solution}
\end{eqnarray}
where, $\bar{Q}_{j}^N:=\left[ {\bar{Q}^N({3j-2}),\bar{Q}^N({3j-1}),\bar{Q}^N({3j})}\right]^\top\;$, $j=1,2,\cdots,M$.

Taking the difference between \eqref{sct-field-Main-Thm} and \eqref{scattered-field-for-truncated-solution}, we obtain
\begin{align}
U^s(x)-U^{s,N}(x)=\;\,&\sum_{j=1}^M \alpha_j\omega^2\Gamma^\omega(x,z_j)\cdot (\bar{Q}_{j}-\bar{Q}_{j}^N)+O\bigg(a^{{1-s+\min\{0,1-2h-\frac{s}{2}\}}}\bigg)
\nonumber\\
\underset{\eqref{Q-Q_N-bar-estimate}}{=}&
O\bigg( a^{1-h-{s}}\,a^{(N+1)(1-h-s)}\bigg)+O\bigg(a^{{1-s+\min\{0,1-2h-\frac{s}{2}\}}}\bigg)
\nonumber\\
=\;\,&O\bigg(a^{1-h-s+\min\{(N+1)(1-h-s),\min\{h,1-h-\frac{s}{2}\}\}}\bigg)
=O\bigg( a^{1-h-{s}}\,a^{(N+1)(1-h-s)}\bigg).\label{us-Us-N-estimation}
\end{align}
since we choose $h$ such that
\begin{eqnarray}
0< 1-h-s\leq \min\left\{\frac{h}{N+1},\frac{\frac{s}{2}}{N}\right\}\label{h-condn-for-nth-level-scattering}
\end{eqnarray}
For,
\begin{align}
\setcounter{mysubequations}{0}
      \mysubnumber\; \;\norm{B^n\cdot U^I}_\infty\underset{(\ref{def-alg-sys-Matr},\ref{def-alg-sys-vect-UI})}{=}\lVert B^n\cdot(C\cdot\;&\; \bar{U})\rVert_\infty\leq \norm{B}_\infty^n\norm{C\cdot \bar{U}}_\infty\underset{(\ref{int-W1-approximation-multiple-particle*},\ref{norm-B-estimation-mp})}{=}O(a^{3-h}a^{n(1-h-s)})
\label{norm-Bn-Ui-estimation-mp}
     \\ \mysubnumber \; \;|\sum_{j=1}^M \alpha_j\omega^2\Gamma^\omega(x,z_j)\cdot (\bar{Q_{j}}-\bar{Q}_{j}^N)|\;&=\left(\sum_{k=1}^3|\sum_{j=1}^M \alpha_j\omega^2\Gamma^\omega_k(x,z_j)\cdot (\bar{Q_{j}}-\bar{Q}_{j}^N)|^2\right)^{1/2}\nonumber\\
&\leq \left(\sum_{k=1}^3\left(\sum_{j=1}^M |\alpha_j\omega^2|\, |\Gamma^\omega_k(x,z_j)| \,\,
|(\bar{Q_{j}}-\bar{Q}_{j}^N)|\right)^2\right)^{1/2}
\nonumber\\
&= O\bigg(M^{1/2}\,a^{-2}\norm{\bar{Q}-\bar{Q}^N}_2\bigg)
= O\bigg(M^{1/2}\,a^{-2}{M^{1/2}}\norm{\bar{Q}-\bar{Q}^N}_\infty\bigg)
\nonumber\\
&\hspace{-0.2cm}\underset{\eqref{norm-Bn-Ui-estimation-mp}}{=} 
 O\bigg( a^{1-h-{s}+(N+1)(1-h-s)}\bigg)\label{Q-Q_N-bar-estimate}.
\end{align}
      
In addition, observe that the difference between the scattered fields at the $N^{th}$ and $(N-1)^{th}$ level scattering satisfies 
\begin{align}
 U^{s,N}(x)-U^{s,N-1}(x)=\sum_{j=1}^M \alpha_j\omega^2\Gamma^\omega(x,z_j)\cdot (\bar{Q}_{j}^N-\bar{Q}_{j}^{N-1}) =\sum_{j=1}^M \alpha_j\omega^2\Gamma^\omega(x,z_j)\cdot (B^{N+1}\cdot U^{I})
 \nonumber\\
\sim a^{1-h-s+N(1-h-s)}\gg  a^{1-h-{s}+(N+1)(1-h-s)}.
\end{align}
\end{proof}
\begin{remark}
 {When $1-h-s=0$ and $\max_j\{|\tilde{\alpha}_j\omega^2|\}\norm{\tilde{C}}_{\infty}3\sqrt{3}\left[d_1^{\circ}+1\right]\max\{H_3,H_4\} {<1}$, the scattered field expansion \eqref{sct-field-Main-Thm} can be expressed as
\begin{eqnarray}
{U^s(x)=\sum_{j=1}^M\alpha_j\omega^2\Gamma^\omega(x,z_j)\cdot \bar{Q}_{j}+O(a^{\min\{1, \min\{h,\frac{s}{2}\}\}})}=\sum_{j=1}^M\alpha_j\omega^2\Gamma^\omega(x,z_j)\cdot \bar{Q}_{j}+O(a^{\frac{s}{2}}).
\end{eqnarray}
Moreover, employing \eqref{scattered-field-for-truncated-solution*}, we observe that the difference between the scattered fields 
$ U^s(x)$ and $U^{s,\infty}(x)$ is of the order $O(a^{\frac{s}{2}+\min\{0,\,1-{2t}\}})$. Here, $U^{s,\infty}(x,\theta)$ represents the field generated after all interactions between the particles, defined as $U^{s,\infty}(x,\theta)=\sum_{j=1}^M \alpha_j \omega^2 \Gamma^\omega(x,z_j)\cdot Q_{j}^\infty$, where $Q^\infty$ is given by $Q^\infty=\sum_{n=0}^\infty B_k^n\cdot U$. 
}
\end{remark}
{From Corollary \ref{corollary-N-interactions}, it becomes evident that employing an incident frequency close to resonance of the order $a^h$,
 where $h$ satisfies \eqref{h-condn-for-nth-level-scattering}, enables us to capture the interactions between inclusions up to order $N$, with $N\in \mathbb{N}$. Moreover, under the condition $s=1-h$, when the incident frequency is close to resonance and the inclusions are nearby, all the interactions between inclusions become visible. Consequently, the Foldy field can be entirely reconstructed, representing the field scattered after all interactions between inclusions.}

%%%%%%%%%%%%%%%%%%%%%%%%%%%%%%%%%%%%%%%%%%%%%%%%%%%%%%%%%%%%%%%%%%%%%%%%%

	 \section{Proof of Theorem \ref{theorem-eff-med}}\label{effective-medium}

\subsection{The equivalent media}

Using the invertibility condition $1-h-s\geq 0$ of the algebraic system \eqref{Algebraic-system-thrm}, as given in Proposition \ref{lemma-algebraic-system-solvability}, and considering the behaviour of the $C^{(j)}$  matrix, which behaves as $O(a^{3-h})$ as shown in \eqref{C-j-matrix-behaviour}, we can observe behaviour of the dominating term as follows:
\begin{eqnarray}
\sum_{m=1}^M \alpha_m\omega^2e^{-\mathbf{\mathtt{i}}\kappa_{\mathtt{p}(\mathtt{s})}\hat{x}\cdot z_m} Q_{m}=O(a^{1-h-s}).
\end{eqnarray}
We consider the following cases:
\begin{itemize}
\item [(i)] Case when $1-h-s>0$:\\
If the number of obstacles is $M:=M(a):=a^{-s}$, $s\leq 2h$ and $t$ satisfies $0< t\leq \frac{1}{2}$ as $a\to 0$, then from \eqref{p-sct-field-Main-Thm-up-infty} and \eqref{s-sct-field-Main-Thm-us-infty}, we deduce that 
\begin{eqnarray}
U^{p,\infty}(\hat{x},\theta)\cdot \hat{x}\to 0,\quad U^{s,\infty}(\hat{x},\theta)\cdot \hat{x}^{\top}\to 0,\quad \mbox{as}\; a\to 0
\end{eqnarray}
This implies that this collection of obstacles has no effect on the homogeneous medium as $a\to 0$.
\item [(ii)] Case when $s=1-h$:\\
We divide the bounded domain $\Omega$ as explained in section \ref{distribution of inclusions}, and the results of this case are discussed in the following sections.  First, we discuss the equivalent behaviour between a collection of many inclusions and an extended penetrable inclusion modelled by an additive potential. To do this, we discuss the equivalent scattering problem and its Lippmann Schwinger equation in section \ref{section-L-S-E-eff}.  Furthermore, in section \ref{section-thrm-proof-eff-med}, we provide the proof of Theorem \ref{theorem-eff-med}, which addresses the error rate between the fields generated by the inclusions and those generated by the equivalent medium. 
\end{itemize}

	 \subsection{ The Lippmann Schwinger equation corresponding to the equivalent scattering problem}\label{section-L-S-E-eff}
	\hfill\\
{
For simplicity, we assume that all inclusions $D_m$ for $m=1,2,\cdots, M$ are identical; this implies they have the same mass density $\rho_j=\rho_1$ and the same elastic scattering coefficient matrix $\tilde{C}=\tilde{C}^{(j)}$ for $j=1,2,\cdots, M$, satisfying \eqref{tilde-gamma-tilde-C-definition}.
\\
Consider the Lippmann Schwinger equation: 
\begin{eqnarray}
{Y^t(x)=U^i(x)+\omega^2\int_{\Omega}\Gamma^\omega(x,y)C Y^t(y)dy}\; \, x\in\mathbb{R}^3,\label{lippman-schwinger-equation-eff-med}
\end{eqnarray}
where,
$C$ is a $3\times 3 $ matrix defined by $C:=\tilde{\alpha}\tilde{C}$ inside $\Omega$ and $0$ outside of $\Omega$.  The matrix $\tilde{C}$ is known as the elastic scattering coefficient matrix and it is independent of $a$. Additionally, $\tilde{\alpha}$ is a constant that satisfies $\alpha_1=\tilde{\alpha}a^{-2}$. 
}

 We define the Lam\'e potential  $T:(L^2(\Omega))^3\to (L^2(\Omega))^3$ as follows:
 \begin{eqnarray}
 T(Y)(x):=\int_{\Omega}\Gamma^\omega(x,y)\,\tilde{\alpha}\,\tilde{C}\,Y(y)\,dy.
 \end{eqnarray}
For any bounded domain $\Omega\subset \mathbb{R}^3$, the Lam\'e potential $T:(L^2(\Omega))^3\to (H^2(\Omega))^3$ is well-defined and bounded. Thus, there exist a positive constant $\beta_0$ such that
\begin{eqnarray}
\norm{T(Y)}_{(H^2(\Omega))^3}\leq \beta_0\norm{Y}_{L^2(\Omega)^3}.\label{lame-potential-norm-inequality}
\end{eqnarray}

\begin{lemma}
There exists a unique solution $Y^t$ to the Lippmann Schwinger equation \eqref{lippman-schwinger-equation-eff-med}, and it satisfies
\begin{eqnarray}\label{infinity-norm-of-Y-grad-Y-inequality}
\norm{Y^t}_{(L^\infty(\Omega))}\leq \beta \norm{U^i}_{(H^2(\Omega))^3} \quad \mbox{and}\quad \norm{\nabla Y^t}_{(L^\infty(\Omega))}\leq 
\tilde{\beta}\norm{U^i}_{(H^2(\tilde{\Omega})},
\end{eqnarray}
where $\beta$ and $\tilde{\beta}$ are positive constants and $\tilde{\Omega}$ is a bounded domain containing $\bar{\Omega}$.
\end{lemma}
\begin{proof}

{We observe that equation \eqref{lippman-schwinger-equation-eff-med} can be written in a simplified form as the non-homogeneous system  $(I-\omega^2 \,T)g=f$, where $g, f\in (L^2(D))^3$.} Since the operator $T$ is compact on $(L^2(D))^3$ (see \cite{COLTON-IE}, for instance), by applying the Fredholm alternative, we can prove the existence of a solution to non-homogeneous system.

Equation \eqref{lippmann-schwinger-equation-algebraic-system-eff-med} is the corresponding Lippmann Schwinger equation for the following scattering problem: 
\begin{equation}
(\Delta^e+\rho_0\omega^2)Y{+}\omega^2\,C\,\chi_{\Omega}
 Y=0, \quad \mbox{in}\quad \mathbb{R}^3\label{scattering-problem-eff-med}
\end{equation} 
In this equation, $Y:=Y^i+Y^s$, where $Y^i$ represents the incident field satisfying $(\Delta^e+\omega^2\rho_0)Y^i=0$ in $\mathbb{R}^3$, and $Y^s$ is the scattered field satisfying the \textit{K.R.C} condition. The uniqueness of the solution to \eqref{lippman-schwinger-equation-eff-med} is a consequence of the uniqueness of the solution to the scattering problem \eqref{scattering-problem-eff-med}.
Since the forward scattering problem is well-posed, we can conclude that a solution exists and is unique for the scattering problem \eqref{scattering-problem-eff-med} problem. As a result, the solution to the Lippmann Schwinger equation \eqref{lippman-schwinger-equation-eff-med} is also unique.\medskip
\\
Therefore, the compact form of \eqref{lippman-schwinger-equation-eff-med}, i.e., $(I-\omega^2 T)Y=U^i$, is invertible from $(L^2(\Omega))^3$ to $(L^2(\Omega))^3$, and there exists a constant $\beta_1$ such that
\begin{equation}
\norm{Y}_{(L^2(\Omega))^3}\leq \beta_1 \norm{U^i}_{(L^2(\Omega))^3}\label{norm-Y-on-L-2-inequality}.
\end{equation} 
From the Lippmann Schwinger equation \eqref{lippman-schwinger-equation-eff-med}, we get
\begin{eqnarray}
\norm{Y}_{(H^2(\Omega))^3}&\leq& \norm{T(Y)}_{(H^2(\Omega))^3}+\norm{U^i}_{(H^2(\Omega))^3}\nonumber\\
&\underset{\eqref{lame-potential-norm-inequality}}{ \leq} &\beta_0\norm{Y}_{(L^2(\Omega))^3}+\norm{U^i}_{(H^2(\Omega))^3}
\nonumber\\
&\underset{\eqref{norm-Y-on-L-2-inequality}}{\leq} &\beta_0(\beta_1\norm{U^i}_{(L^2(\Omega))^3})+\norm{U^i}_{(H^2(\Omega))^3}
\nonumber\\
&\leq &\beta_2 \norm{U^i}_{(H^2(\Omega))^3}.\label{norm-Y-inequality-on-H2-space}
\end{eqnarray}
By the Sobolev embedding theorem\cite{KESAVAN}, for  an integer $m\geq 1$ and $1\leq p <\infty$, we have $W^{m,p}(\Omega)\hookrightarrow L^\infty(\Omega)$ whenever $m> \frac{3}{p}$. Consequently, $H^2(\Omega)\hookrightarrow L^\infty(\Omega)$. Furthermore, there exist a constant $\beta_3$ such that
\begin{eqnarray}
\norm{Y}_{L^\infty(\Omega)}&\leq &\beta_3 \norm{Y}_{(H^2(\Omega))^3}
\underset{\eqref{norm-Y-inequality-on-H2-space}}{\leq} {\beta_3\beta_2\norm{U^i}_{(H^2(\Omega))^3}}.\label{norm-Y-L-infty-inequality}
\end{eqnarray}

Since $L^\infty(\tilde{\Omega})\subset L^p(\tilde{\Omega})$ for $1\leq p\leq \infty$, we get
\begin{eqnarray}
\norm{Y}_{(L^p(\tilde{\Omega}))^3}&\leq& \beta_5\norm{Y}_{L^\infty(\tilde{\Omega})}
\underset{\eqref{norm-Y-L-infty-inequality}}{ \leq}  \beta_5 \beta_3\beta_2 \norm{U^i}_{(H^2(\tilde{\Omega}))^3}.\label{norm-Y-on-L-p-inequality}
\end{eqnarray}
Using  interior estimates of scattering problem, we have $Y\in (W^{2,p}(\Omega))^3$ for $p\geq 1$. Then there exist $\beta_4$ such that
{
\begin{eqnarray}
\norm{Y}_{(W^{2,p}(\Omega))^3}\leq \beta_4 \norm{Y}_{(L^p(\tilde{\Omega}))^3}
\end{eqnarray}
}
where $\tilde{\Omega}$ is large bounded domain which contain $\bar{\Omega}$. Hence
\begin{eqnarray}
{\norm{Y}_{(W^{2,p}(\Omega))^3}\leq \beta_4\norm{Y}_{(L^p(\tilde{\Omega}))^3}}\underset{\eqref{norm-Y-on-L-p-inequality}}{\leq} \beta_4\beta_5\beta_3\beta_2\norm{U^i}_{(H^2(\tilde{\Omega}))^3}.\label{norm-W-2,p-norm-Ui-inequality}
\end{eqnarray}
Since $Y\in (W^{2,p}(\Omega))^3$ implies $Y$ and $\nabla Y \in (W^{1,p}(\Omega))^3$, hence using the Sobolev embedding that $W^{1,p}(\Omega)\hookrightarrow L^\infty(\Omega)$ for $p>3$, we get  
\begin{eqnarray}
\norm{\nabla Y}_{L^\infty(\Omega)}\leq \beta_7 \norm{\nabla Y}_{(W^{1,p}(\Omega))^3}\leq \beta_7 \norm{Y}_{(W^{2,p}(\Omega))^3}\underset{\eqref{norm-W-2,p-norm-Ui-inequality}}{\leq} \tilde{\beta} \norm{U^i}_{(H^2(\tilde{\Omega}))^3}
\end{eqnarray}
where $\beta_7$ and $\beta:=\beta_7\beta_4\beta_5\beta_3\beta_2$ are positive constants. Hence \eqref{infinity-norm-of-Y-grad-Y-inequality} is proved.

\end{proof}

	 \subsection{Proof of Theorem \ref{theorem-eff-med}}\label{section-thrm-proof-eff-med}
	 \begin{proof}
For each  $z_j$ with $j=1,2,\cdots, M$, the Lippmann Schwinger equation  \eqref{lippman-schwinger-equation-eff-med}  can be represented as
\begin{align}
Y^t(z_j)
=& U^i(z_j)+P_1+P_2+P_3 \label{lippmann-schwinger-equation-algebraic-system-eff-med}
\end{align}
where 
\begin{eqnarray}\label{def-P123}
P_1:=\sum_{\substack{m=1 \\ m\neq j}}^{[a^{h-1}]}\int_{\Omega_m}\hspace{-0.4cm}\Gamma^\omega(z_j,y)\tilde{\alpha}\tilde{C}Y^t(y)dy, \;\;
P_2:=\int_{\Omega_j}\hspace{-0.34cm}\Gamma^\omega(z_j,y)\tilde{\alpha}\tilde{C}Y^t(y)dy, \;\;
P_3:=\int_{\Omega\setminus\cup_{m=1}^{[a^{h-1}]}\Omega_m}\hspace{-0.4cm}\Gamma^\Omega(z_j,y)CY^t(y)dy.
\end{eqnarray}
To proceed,  we will estimate each of the terms $P_i$, $i=1,2,3$:
\begin{enumerate}
\item  Estimation of $P_1$: Consider $P_1$ and apply the Taylor series for $\Gamma^\omega(z_j,y)\tilde{\alpha}\tilde{C}Y^t(y)$ about $y$ near $z_m$. This yields
\begin{eqnarray}
P_1=\sum_{\substack{m=1 \\ m\neq j}}^{[a^{h-1}]}\int_{\Omega_m}\Gamma^\omega(z_j, z_m)\tilde{\alpha}\tilde{C}Y^t(z_m)dy+\sum_{\substack{m=1 \\ m\neq j}}^{[a^{h-1}]}\int_{\Omega_m}R_m(z_m,y)\,dy\label{taylor-gamm-omega-y-eff-med}
\end{eqnarray}
where
\begin{align}
R_m(z_m,y)=&\int_{0}^1\nabla_y [\Gamma^\omega(z_j, z_m+t(y-z_m))\,C\,Y^t(z_m+t(y-z_m))]\cdot (y-z_m)dtdy
\nonumber\\
=&\int_{0}^1 [\nabla_y(\Gamma^\omega(z_j, z_m+t(y-z_m))\,C)\cdot Y^t(z_m+t(y-z_m))]\cdot (y-z_m)dt
\nonumber\\
&+\int_{0}^1[\Gamma^\omega(z_j, z_m+t(y-z_m))\,{C}\,\nabla_y Y^t(z_m+t(y-z_m))]\cdot (y-z_m)dt.
\end{align}
For $m\neq j$, we have $|\Gamma^\omega(z_j, z_m+t(y-z_m))|=O\bigg(\frac{1}{|z_j-(z_m+t(y-z_m))|}\bigg)$ and $|\nabla_y\Gamma^\omega(z_j,z_m+t(y-z_m))|=O\bigg(\frac{1}{|z_j-(z_m+t(y-z_m))|^2}\bigg)$. Thus
\begin{align}
\sum_{\substack{m=1\\m\neq j}}^{[a^{h-1}]}\hspace{-0.1cm}\int_{\Omega_m}\hspace{-0.3cm}|R_m(z_m,y)|&\underset{\eqref{infinity-norm-of-Y-grad-Y-inequality},\eqref{Gamma-is-bounded-x-in-Di-y-in-Dj-mp},\eqref{grad-Gamma-x-Di-y-Dj-homogeneous-background-mp}}{\leq}\hspace{-0.35cm} \sum_{n=1}^{[\frac{a^{\frac{h-1}{3}}-1}{2}]}\hspace{-0.3cm}(24n^2\hspace{-0.05cm}+\hspace{-0.1cm}2 )a^{1-h}\hspace{-0.1cm}\left(\dfrac{\mathtt{H_5}|C|\beta\norm{U^i}_{(H^2(\Omega))^3}}{n^2\, (a^{\frac{1-h}{3}}-\frac{a}{2})^2}+\hspace{-0.01cm}\dfrac{\mathtt{H_3}|C|\tilde{\beta}\norm{U^i}_{(H^2(\Omega))^3}}{n\, (a^{\frac{1-h}{3}}-\frac{a}{2})}\hspace{-0.01cm}\right)\hspace{-0.1cm} a^{\frac{1-h}{3}}
\nonumber\\
&\qquad\hspace{0.1cm}\lesssim \;\;\sum_{n=1}^{[\frac{a^{\frac{h-1}{3}}-1]}{2}]}{\hspace{-0.3cm}(24n^2+2)\beta^\circ}a^{1-h}\bigg(\dfrac{\mathtt{H_5}\,\beta}{n^2\, (\frac{a^{\frac{1-h}{3}}}{2})^2}+\dfrac{\mathtt{H_3}\,\tilde{\beta}}{n\, (\frac{a^{\frac{1-h}{3}}}{2})}\bigg)a^{\frac{1-h}{3}}
=O(a^{{\frac{(1-h)}{3}}})\label{remainder-term-eff-med-second-term}
\end{align}
 where $\norm{U^i}_{(H^2(\Omega))^3}=O(1)$ bounded by $\beta^\circ$, a constant independent of $a$. \\
 
 By substituting \eqref{remainder-term-eff-med-second-term} into \eqref{taylor-gamm-omega-y-eff-med}, we obtain
 \begin{eqnarray}
 P_1=\sum_{\substack{m=1 \\ m\neq j}}^{[a^{h-1}]}\int_{\Omega_m}\Gamma^\omega(z_j, z_m)\tilde{\alpha}\tilde{C}Y^t(z_m)dy+O(a^{{\frac{(1-h)}{3}}}).\label{second-term-estimation-eff-med}
 \end{eqnarray}
 \medskip
\item Estimation of $P_2$: Making use the definition of $P_2$ from \eqref{def-P123}, we obtain
\begin{align}
|P_2|
\leq |\tilde{\alpha}|\,|\int_{\Omega_j}\Gamma^\omega(z_j,y)\tilde{C}Y^t(y)dy|
\underset{\eqref{infinity-norm-of-Y-grad-Y-inequality}}{\lesssim}& |\tilde{\alpha}||{\tilde{C}}|_{\infty}\,\beta\,\beta^\circ \int_{\Omega_j}| \dfrac{1}{|z_j-y|}|dy.\label{integral-value-third-term-eff-med}
\end{align}
  To evaluate the integral in \eqref{integral-value-third-term-eff-med}, we consider:
\begin{eqnarray}
\int_{\Omega_j}| \dfrac{1}{|z_j-y|}|dy&=&\int_{B(z_j,r)}\dfrac{1}{|z_j-y|}dy+\int_{\Omega_j\setminus B(z_j,r)}\dfrac{1}{|z_j-y|}dy
\nonumber\\
&=&4\pi\int_{0}^r\dfrac{1}{r}r^2 dr+\frac{1}{r}\, \mbox{Vol}(\Omega_j\setminus B(z_j,r))=2\pi\,r^2+\frac{1}{r}[a^{1-h}-\frac{4}{3}\pi r^3]\nonumber
\end{eqnarray}
Noting that $2\pi\,r^2+\frac{1}{r}[a^{1-h}-\frac{4}{3}\pi r^3]$ has a critical point at $r=(\frac{3}{4\pi})^{\frac{1}{3}}a^{\frac{1-h}{3}}$, we can deduce that
\begin{eqnarray}
\int_{\Omega_j}| \dfrac{1}{|z_j-y|}|dy\leq 2\pi(3/4\pi)^{\frac{2}{3}}a^{\frac{2(1-h)}{3}}+(4\pi/ 3)^{\frac{1}{3}}a^{\frac{2(1-h)}{3}}-\frac{4}{3}\pi (3/4\pi)^{\frac{2}{3}}a^{\frac{2(1-h)}{3}}=O(a^{\frac{2(1-h)}{3}}).\label{mod-gamma-omega-x-zj-integral-value-eff-med}
\end{eqnarray}
Hence, combining \eqref{mod-gamma-omega-x-zj-integral-value-eff-med} and \eqref{integral-value-third-term-eff-med}, we conclude that the behavior of $P_2$ as:
\begin{eqnarray}
P_2=O(a^{\frac{2(1-h)}{3}}).\label{third-term-estimation-eff-med}
\end{eqnarray}
\medskip
\item Estimation of $P_3$: 
{ To estimate the term $P_3$ defined in \eqref{def-P123}, we need to address two scenarios based on the distance of $z_j$ to the boundary $\partial \Omega$, specifically focusing on the proximity of $z_j$ to the cubes $\Omega^{'}_m$,  $m=1,2,\cdots, [a^{\frac{h-1}{3}}]$.  In the first case, when $z_j$ is away from the boundary $\partial\Omega$, $\Gamma^\omega(z_j,y)$ is bounded for $y$ on $\partial \Omega$. In the second case, when the point $z_j$ is close to  any cube $\Omega^{'}_m$,  $\Gamma^\omega(z_j,y)$ becomes singular for $y$ near $z_j$. To address these two situations while evaluating the integral involving $\Gamma^\omega(z_j,y)$, we partition the integral into two parts. One part involves the cubes $\Omega_m^{'}$ near to $z_j$, denoted as ${E}_{(j)}$, while the other part is denoted by ${F}_{(j)}$. Hence, }

\begin{align}\label{last-term-estimation-eff-med}
|P_3|=\;&|\int_{E_{(j)}}\Gamma^\omega(z_j,y)\tilde{\alpha} \tilde{C}Y^t(y)dy+\int_{F_{(j)}}\Gamma^\omega(z_j,y)\tilde{\alpha} \tilde{C}Y^t(y)dy|
\nonumber\\
\leq\;& |\tilde{\alpha}||\tilde{C}|_{\infty}\norm{Y^t}_{L^\infty}\int_{E_{(j)}}|\Gamma^\omega(z_j,y)|dy+|\tilde{\alpha}||\tilde{C}|_{\infty}\norm{Y^t}_{L^\infty}|\Gamma^\omega(z_j,y)|_{\infty}\,\mbox{Vol}(F_{(j)})
\nonumber\\
\underset{\eqref{infinity-norm-of-Y-grad-Y-inequality}}{\leq} &|\tilde{\alpha}||\tilde{C}|_{\infty}|\beta \norm{U^i}_{(H^2(\Omega))^3}\bigg(\sum_{m=1}^{[a^{2(h-1)/3}]}\frac{1}{d_{jm}}\, \mbox{Vol}(\Omega^{'}_m)+|\Gamma^\omega(z_j,y)|_{\infty}\,\mbox{Vol}(F_{(j)})\bigg)
\nonumber\\
{\leq} \;&|\tilde{\alpha}||\tilde{C}|_{\infty}|\beta \norm{U^i}_{(H^2(\Omega))^3}\bigg(\sum_{n=1}^{[a^{(h-1)/3}]}(24n^2+2)\frac{1}{n(a^{\frac{1-h}{3}}-\frac{a}{2})}\,\mbox{Vol}(\Omega^{'}_m)+|\Gamma^\omega(z_j,y)|_{\infty}\,\mbox{Vol}(F_{(j)})\bigg)
\nonumber\\
\leq \;&|\tilde{\alpha}||\tilde{C}|_{\infty}|\beta \beta^\circ|\bigg({a^{1-h}a^{-2(1-h)/3}}+|\,|\Gamma^\omega(z_j,y)|_{\infty}\,\,a^{(1-h)/3}\bigg)=O(a^{\frac{1-h}{3}}).
\end{align}
\end{enumerate}

Substituting the estimates \eqref{second-term-estimation-eff-med}, \eqref{third-term-estimation-eff-med} and \eqref{last-term-estimation-eff-med} of $P_1,P_2$ and $P_3$ in \eqref{lippmann-schwinger-equation-algebraic-system-eff-med}, we get
{
\begin{eqnarray}
Y^t(z_j)=U^i(z_j,\theta)+\sum_{\substack{m=1 \\ m\neq j}}^{[a^{h-1}]}\int_{\Omega_m}\omega^2 \Gamma^\omega(z_j, z_m)\tilde{\alpha}\tilde{C}Y^t(z_m)dy+O(a^{\frac{1-h}{3}}).\label{algebraic-system-eff-med-LSE}
\end{eqnarray}
}

Introducing $U_j:={C^{(j)}}^{-1}Q_{j}$, $j=1,2,\cdots, M$ in \eqref{Algebraic-system-thrm} and making use the periodic arrangement of inclusions as discussed in section \ref{distribution of inclusions}, we obtain:
\begin{eqnarray}
U_j&=&U^i(z_j,\theta)
+\sum_{\substack{m=1 \\ m\neq j}}^M\,\alpha_m\omega^2\Gamma^\omega(z_j,z_m)\cdot  C^{(m)}\,U_m.\label{algebraic-system-with-inverse-C-eff-med}\\
&\underset{\eqref{tilde-gamma-tilde-C-definition}}{=}&U^i(z_j,\theta)+\sum_{\substack{m=1\\m\neq j}}^{[a^{h-1}]}a^{1-h}\omega^2\Gamma^\omega(z_j,z_m)\cdot \tilde{\alpha}\,\,\tilde{C}\,U_m, \quad j=1,2,3,\cdots, M=[a^{h-1}].\label{algebraic-system-written-periodic-arrang}
\end{eqnarray}
 Comparing  \eqref{algebraic-system-written-periodic-arrang} with \eqref{algebraic-system-eff-med-LSE}, will lead to the difference estimate
\begin{align}
U_j-Y^t(z_j)=&\sum_{\substack{m=1 \\ m\neq j}}^{[a^{h-1}]}a^{1-h}\omega^2\Gamma^\omega(z_j,z_m)\cdot \tilde{\alpha}\tilde{ C}\,U_m -  \sum_{\substack{m=1 \\ m\neq j}}^{[a^{h-1}]}\int_{\Omega_m}\omega^2\Gamma^\omega(z_j, z_m)\tilde{\alpha}\tilde{C}Y^t(z_m)dy+O(a^{\frac{(1-h)}{3}})
\nonumber\\
=&\sum_{\substack{m=1 \\ m\neq j}}^M\,\tilde{\alpha}\omega^2\Gamma^\omega(z_j,z_m)\cdot a^{1-h}\tilde{ C}\,U_m - \sum_{\substack{m=1 \\ m\neq j}}^M\omega^2\Gamma^\omega(z_j, z_m)\tilde{\alpha}\tilde{C}Y^t(z_m) \, a^{1-h}+O(a^{\frac{(1-h)}{3}})
\nonumber\\
=&{\sum_{\substack{m=1 \\ m\neq j}}^M\,a^{1-h}\tilde{\alpha}\,\omega^2\Gamma^\omega(z_j,z_m)\cdot \tilde{ C}\,(U_m -Y^t(z_m)) +O(a^{\frac{(1-h)}{3}})}.\label{scattered-field-difference-with-eff-med}
\end{align}

Observing that the coefficient matrix of the systems \eqref{lemma-algebraic-system-solvability} and \eqref{scattered-field-difference-with-eff-med}/\eqref{Algebraic-system-thrm} is identical, and applying the invertibility condition established in proposition \ref{lemma-algebraic-system-solvability}, we deduce that
 \begin{eqnarray}
 \sum_{j=1}^M(U_j-Y^t(z_j))=O(Ma^{\frac{(1-h)}{3}}).\label{scattered-field-diff-eff-med-over-sum-j=1-to-M}
 \end{eqnarray}

Similarly, introducing $U_j:={C^{(j)}}^{-1}Q_{j}$, $j=1,2,\cdots, M$ in the farfields expansions \eqref{p-sct-field-Main-Thm-up-infty} and \eqref{s-sct-field-Main-Thm-us-infty} associated to $\mathtt{p}$ and $\mathtt{s}$ parts, we obtain:
\begin{align}
U^\infty_{\mathtt{p}}(\hat{x},\theta)=&(\hat{x}\otimes\hat{x})\sum_{j=1}^M \frac{\omega^2}{4\pi (\lambda+2\mu)}\alpha_j\,e^{-\mathbf{\mathtt{i}}\kappa_{\mathtt{p}}\hat{x}\cdot z_j}\cdot a^{3-h} \tilde{C}U_j+{O\left(   a^{1-s+\min\{0,1-2h-\frac{s}{2}\}}\right)}\label{p-part-farfield-for-system-eff-med}
\\
U_{\mathtt{s}}^\infty(\hat{x},\theta)=&(I-\hat{x}\otimes \hat{x})\sum_{j=1}^M \frac{\omega^2}{4\pi \mu}\alpha_j\,e^{-\mathbf{\mathtt{i}}\kappa_{\mathtt{s}}\hat{x}\cdot z_j}\cdot a^{3-h} \tilde{C}U_j+{O\left(   a^{1-s+\min\{0,1-2h-\frac{s}{2}\}}\right)}.\label{s-part-farfield-for-system-eff-med}
\end{align}
Observe that the $\mathtt{p}$ and $\mathtt{s}$ parts of the farfields corresponding to the Lippmann Schwinger equation \eqref{lippman-schwinger-equation-eff-med} of the equivalent scattering problem are given by:
\begin{eqnarray}
Y_{\mathtt{p}}^\infty(\hat{x},\theta)&=&(\hat{x}\otimes \hat{x})\,\omega^2\int_{\Omega}\frac{1}{4\pi (\lambda+2\mu)}\,e^{-\mathbf{\mathtt{i}}\kappa_{\mathtt{p}}\hat{x}\cdot y}\cdot\tilde{\alpha}\tilde{C} Y^t(y)dy\label{p-part-farfield-eff-med}
\\
Y_{\mathtt{s}}^\infty(\hat{x},\theta)&=&(I-\hat{x}\otimes \hat{x})\,\omega^2\int_{\Omega}\frac{1}{4\pi \mu}\,e^{-\mathbf{\mathtt{i}}\kappa_{\mathtt{s}}\hat{x}\cdot y}\cdot\tilde{\alpha}\tilde{C} Y^t(y)dy\label{s-part-farfield-eff-med}
\end{eqnarray}
Substracting \eqref{p-part-farfield-eff-med} from \eqref{p-part-farfield-for-system-eff-med} yields:
{
\begin{align}
(U^\infty_p(\hat{x},\theta)-Y_{\mathtt{p}}^\infty(\hat{x},\theta))\cdot 4\pi(\lambda+2\mu)\hat{x}=&\sum_{j=1}^M \omega^2{\alpha_1}\,e^{-\mathbf{\mathtt{i}}\kappa_{\mathtt{p}}\hat{x}\cdot z_j} a^{3-h} \tilde{C}U_j\cdot \hat{x}-\,\omega^2\int_{\Omega}e^{-\mathbf{\mathtt{i}}\kappa_{\mathtt{p}}\hat{x}\cdot y}\,\tilde{\alpha}\tilde{C} Y^t(y)dy\cdot \hat{x}
\nonumber\\
&+{O\left(   a^{1-s+\min\{0,1-2h-\frac{s}{2}\}}\right)}
\nonumber\\
=&\sum_{j=1}^M \omega^2\tilde{\alpha}\,e^{-\mathbf{\mathtt{i}}\kappa_{\mathtt{p}}\hat{x}\cdot z_j} a^{1-h} \tilde{C}U_j\cdot \hat{x}-J_1-J_2+O\left(   a^{1-s+\min\{0,1-2h-\frac{s}{2}\}}\right)\label{farfield-difference-p-part-eff-med}
\end{align}
}
where 
\begin{eqnarray}
J_1:=\omega^2\sum_{j=1}^{[a^{h-1}]}\int_{\Omega_j}e^{-\mathbf{\mathtt{i}}\kappa_{\mathtt{p}}\hat{x}\cdot y}\,\tilde{\alpha}\tilde{C} Y^t(y)dy\cdot \hat{x},\quad J_2:=\omega^2\int_{\Omega\setminus \cup_{j=1}^{[a^{h-1}]}\Omega_j}e^{-\mathbf{\mathtt{i}}\kappa_{\mathtt{p}}\hat{x}\cdot y}\,\tilde{\alpha}\tilde{C} Y^t(y)dy\cdot \hat{x}.
\end{eqnarray}
Now we estimate $J_1$ and $J_2$ by making use of arrangement of inclusions as explained in section \eqref{distribution of inclusions}:
\begin{enumerate}
\item Estimation of $J_1$:
\\
Consider $J_1$ and apply Taylor's series expansion of $e^{-\mathbf{\mathtt{i}}\kappa_{\mathtt{p}}\hat{x}\cdot y}Y^t(y)$ about $y$ near $z_j$. This yields
\begin{align}
J_1&=\omega^2\sum_{j=1}^{[a^{h-1}]}\left[\int_{\Omega_j}\hspace{-0.3cm}\tilde{\alpha}\,\tilde{C}\cdot \bigg(e^{-\mathbf{\mathtt{i}}\kappa_{\mathtt{p}}\hat{x}\cdot z_j}Y^t(z_j)+\hspace{-0.2cm} \int_{0}^1\hspace{-0.23cm}\nabla_y [e^{-\mathbf{\mathtt{i}}\kappa_{\mathtt{p}}\hat{x}\cdot (z_j+t(y-z_j))} Y^t(z_j+t(y-z_j))]\cdot (y-z_j) dt\bigg) dy\right]\cdot \hat{x}
\nonumber\\
&=\omega^2\sum_{j=1}^{[a^{h-1}]}a^{1-h}\tilde{\alpha}\,
e^{-\mathbf{\mathtt{i}}\kappa_{\mathtt{p}} \hat{x}\cdot z_j}\tilde{C}\cdot Y^t(z_j)\cdot \hat{x}+\omega^2\sum_{j=1}^{[a^{h-1}]}\int_{\Omega_j}\tilde{\alpha}\tilde{C}\cdot \mathtt{Rem}_{j}(z_j,y)\,dy\cdot \hat{x}.\label{second-term-frafield-difference-taylor-eff-med}
\end{align}
where
\begin{eqnarray}
\mathtt{Rem}_{j}(z_j,y)&:=&\int_{0}^1\hspace{-0.23cm}\nabla_y [e^{-\mathbf{\mathtt{i}}\kappa_{\mathtt{p}}\hat{x}\cdot (z_j+t(y-z_j)} Y^t(z_j+t(y-z_j))]\cdot (y-z_j) dt.
\end{eqnarray}
Observing that, $\nabla_y e^{-\mathbf{\mathtt{i}}\kappa_{\mathtt{p}}\hat{x}\cdot (z_j+t(y-z_j)}=-\mathbf{\mathtt{i}}\kappa_{\mathtt{p}}\hat{x}e^{-\mathbf{\mathtt{i}}\kappa_{\mathtt{p}}\hat{x}\cdot (z_j+t(y-z_j)}$ and
\begin{align}
|\mathtt{Rem}_{j}(z_j,y)|
=&|\hspace{-0.2cm}\int_{0}^{1}\hspace{-0.3cm}\nabla_y e^{-\mathbf{\mathtt{i}}\kappa_{\mathtt{p}}\hat{x}\cdot (z_j+t(y-z_j)}\otimes Y^t(z_j+t(y-z_j))\hspace{-0.05cm}+\hspace{-0.05cm}e^{-\mathbf{\mathtt{i}}\kappa_{\mathtt{p}}\hat{x}\cdot (z_j+t(y-z_j)}\nabla_y Y^t(z_j+t(y-z_j))]\cdot (y-z_j)dt
| 
\nonumber \\
 &\hspace{-2.1cm}\leq\int_{0}^1\hspace{-0.23cm}|\nabla_y e^{-\mathbf{\mathtt{i}}\kappa_{\mathtt{p}}\hat{x}\cdot (z_j+t(y-z_j)}||Y^t(z_j+t(y-z_j))||y-z_i|dt+\hspace{-0.2cm}\int_{0}^{1}\hspace{-0.23cm}|e^{-\mathbf{\mathtt{i}}\kappa_{\mathtt{p}}\hat{x}\cdot (z_j+t(y-z_j)}||\nabla_y Y^t(z_j+t(y-z_j))||y-z_j|dt,
\nonumber
\end{align}
 we deduce
\begin{eqnarray}
\sum_{j=1}^{[a^{h-1}]}\int_{\Omega_j}|\tilde{\alpha}\,\tilde{C}\, \mathtt{Rem}_{j}(z_j,y)|\,dy\underset{\eqref{infinity-norm-of-Y-grad-Y-inequality}}{\lesssim} |\tilde{\alpha}||\tilde{C}| \beta^\circ\sum_{n=1}^{[a^{\frac{\frac{h-1}{3}-1}{2}}]}\hspace{-0.4cm}(24n^2+2)[\kappa_{\mathtt{p}}\beta
+\tilde{\beta}|a^{\frac{1-h}{3}}a^{1-h}=O(a^{\frac{1-h}{3}})+O(a^{1-h})
\end{eqnarray}
and hence we get the below estimate of $J_1$ from \eqref{second-term-frafield-difference-taylor-eff-med}:
\begin{eqnarray}
J_1=\omega^2 \sum_{j=1}^{[a^{h-1}]}a^{1-h}\tilde{\alpha}\,
e^{-\mathbf{\mathtt{i}}\kappa_{\mathtt{p}} \hat{x}\cdot z_j}\tilde{C}\cdot Y^t(z_j)\cdot \hat{x}+O(a^{\frac{1-h}{3}})\label{second-term-esti-farfield-p-eff-medium}
\end{eqnarray}
\medskip
\item Estimation of $J_2$:
Making use the definition of $J_2$ from \eqref{farfield-difference-p-part-eff-med}, we obtain
\begin{equation}
|J_2|\,\leq\, |\omega^2 e^{-\mathbf{\mathtt{i}}\kappa_{\mathtt{p}}\hat{x}\cdot y}|_\infty |\tilde{\alpha}||\tilde{C}|_{\infty}|Y^t|_{L^\infty} {\mbox{Vol}(\Omega\setminus \cup_{j=1}^{[a^{h-1}]}\Omega_j)}
\,\underset{\eqref{infinity-norm-of-Y-grad-Y-inequality}}{\lesssim} \,|\omega^2|\tilde{\alpha}||\tilde{C}|_{\infty}\beta\beta^\circ a^{\frac{1-h}{3}}.\label{last-term-farfield-difference-p-eff-med}
\end{equation}
\end{enumerate}

Substituting the estimates \eqref{second-term-esti-farfield-p-eff-medium} and \eqref{last-term-farfield-difference-p-eff-med} of $J_1$ and $J_2$ respectively in \eqref{farfield-difference-p-part-eff-med}, we derive that
\begin{align}
(U^\infty_p(\hat{x},\theta)\hspace{-0.1cm}-\hspace{-0.1cm}Y_{\mathtt{p}}^\infty(\hat{x},\theta))\hspace{-0.07cm}\cdot 4\pi(\lambda+2\mu)\hat{x}=&\hspace{-0.1cm}\sum_{j=1}^M\hspace{-0.1cm} \omega^2\,\gamma\,e^{-\mathbf{\mathtt{i}}\kappa_{\mathtt{p}}\hat{x}\cdot z_j} a^{3-h} \tilde{C}U_j\cdot \hat{x}-\hspace{-0.06cm}\omega^2 \hspace{-0.13cm}\sum_{j=1}^{[a^{h-1}]}\hspace{-0.1cm}a^{1-h}\tilde{\alpha}\,
e^{-\mathbf{\mathtt{i}}\kappa_{\mathtt{p}} \hat{x}\cdot z_j}\tilde{C}\cdot Y^t(z_j)\cdot \hat{x}
\nonumber\\
&+O(a^{\frac{1-h}{3}})+{O\left(   a^{1-s+\min\{0,1-2h-\frac{s}{2}\}}\right)}
\nonumber\\
=\;\;&\omega^2\hspace{-0.1cm}\sum_{j=1}^M \hspace{-0.06cm}\tilde{\alpha} a^{1-h}e^{-\mathbf{\mathtt{i}}\kappa_{\mathtt{p}}\hat{x}\cdot z_j} \tilde{C}\bigg(\hspace{-0.13cm}U_j\hspace{-0.1cm}-\hspace{-0.1cm}Y^t(z_j)\hspace{-0.11cm}\bigg)\hspace{-0.08cm}\cdot\hspace{-0.06cm} \hat{x}
\hspace{-0.08cm}+\hspace{-0.08cm}O(a^{\frac{1-h}{3}})\hspace{-0.08cm}+\hspace{-0.08cm}{O\left(   a^{1-s+\min\{0,1-2h-\frac{s}{2}\}}\right)}
\nonumber\\
\underset{\eqref{scattered-field-diff-eff-med-over-sum-j=1-to-M}}{=}& O(a^{1-h}Ma^{\frac{1-h}{3}})+O(a^{\frac{1-h}{3}})+{O\left(   a^{1-s+\min\{0,1-2h-\frac{s}{2}\}}\right)}\nonumber\\
=\;\;& O(a^{\frac{1-h}{3}})+O\bigg(a^{h+\min\{0,\, \frac{3s}{2}-1\}}\bigg)\qquad
\nonumber\\
=\;\;& O (a^{\min\{\frac{s}{3},\min\{h,\frac{s}{2}\}\}})=O(a^{\min\{h,\frac{s}{3}\}}).
\end{align}
Similarly, by following the same procedure, we obtain the difference in the $s$-part of the farfields 
\begin{eqnarray}
U_{\mathtt{s}}^\infty(\hat{x},\theta)-Y_{\mathtt{s}}^\infty(\hat{x},\theta))\cdot 4\pi \mu \hat{x}^\top=O(a^{\min\{h,\frac{s}{3}\}})
\end{eqnarray}
uniformly for all $\hat{x},\theta$ in $\mathbb{S}^2$.
\end{proof}

\section{Appendix}\label{section-appendix}
 \begin{lemma}\label{lemma-Gamma-propertie-Dl-mp} The fundamental matrix $\Gamma^{\omega}$ satisfies the following properties: 
 \begin{enumerate}
\item \label{lemma-Gamma-x,y-bounded-x,y-in-Dl-mp}For $x,y$ in $D_m$, $(\Gamma^\omega-\Gamma^0)_{ij}(x,y)$,  $i,j=1,2,3$ is bounded.%, for any domain $D_m$. 
 \item \label{lemma-grad-Gamma-x,y-bounded-y-in-D-x-away-D-mp}
{For $y\in D_m$ and $x$ away from $\Omega$, $ \nabla_y\Gamma^{\omega}_{ij}(x,y)$, $i,j=1,2,3$ is bounded.}
 \end{enumerate}
\end{lemma}
%%%%%%%%%%%%%%%%%%%%%%%%%%%%%%%%%%%%%%%%%%%%%%%%%%%%%%%%%%%%%%%%%%%
\begin{proof}
The proof of these results are obvious due to the singularity type of the fundamental tensor $\Gamma^\omega$. However, we provide a brief sketch of the proof as we need to fix some quantities that were used in the main parts of the text, as the constant $H_l, l=1, ..., 6,$ below.
\begin{enumerate}
\item Making use the representations \eqref{entrywise-FM-mp} and \eqref{entrywise-FM-zerof-mp} of the fundamental matrix $\Gamma^{\omega}(x,y)$ and the Kelvin matrix $\Gamma^{0}(x,y)$ respectively, by performing the similar computations as it was done in \cite[Lemma 4.2]{challa2023extraction}), one can prove that the $ij^{\mbox{th}}$ entry of $(\Gamma^\omega-\Gamma^0)$ is bounded by $H_1$ on each $D_m$, where
\begin{eqnarray}\label{Gamma-Gammm0-bounded-x,y-in-D-H1-value}
 H_1:=\dfrac{1}{4\pi{\rho_0}}\left[\dfrac{2\kappa_{\mathtt{s}}}{c_{\mathtt{s}}^2}\dfrac{1}{1-\frac{1}{2}\kappa_{\mathtt{s}}\,{\max\limits_{1\leq m\leq M}}diam(D_m)}+\dfrac{\kappa_{\mathtt{p}}}{c_{\mathtt{s}}^2}\dfrac{1}{1-\frac{1}{2}\kappa_{\mathtt{p}}\,{\max\limits_{1\leq m\leq M}}diam(D_m)}\right].
\end{eqnarray}

\medskip
%%%%%%%%%%%%%%%%%%%%%%%%%%%%%%%%%%%%%%%%%%%%%%%%%%%%%%%%%%%%%%%%%%%%%%%%%%%%%%%
\item To find the bound for $ \nabla_y\Gamma^{\omega}_{ij}(x,y)$,  $i,j=1,2,3$ for $y\in D_m$ and $x$ away from $\Omega$, 
\begin{itemize}
\item First, observe that 
$|\nabla_y\Gamma^{\omega}_{ij}(x,y)|^2=\sum_{l=1}^3|\dfrac{\partial}{\partial y_{l}}\Gamma^{\omega}_{ij}(x,y)|^2$. 

\item Second, consider the expression \eqref{entrywise-FM-mp-Appendix-use} of $\Gamma^{\omega}_{ij}(x,y)$ which can be rewritten further as
\begin{eqnarray}
\hspace{-0.1cm}\Gamma^{\omega}_{ij}(x,y)\hspace{-0.3cm}&=&\hspace{-0.3cm}\frac{\delta_{ij}\,\, e^{\mathbf{\mathtt{i}}\kappa_{\mathtt{s}}|x-y|}}{4\pi \mu |x-y|}+\frac{1}{4\pi\omega^2\rho_0}\hspace{-0.1cm}\left[-\hspace{-0.05cm}\left(\dfrac{\kappa_{\mathtt{s}}^2e^{\mathbf{\mathtt{i}}\kappa_{\mathtt{s}}|x-y|}-\kappa_{\mathtt{p}}^2e^{\mathbf{\mathtt{i}}\kappa_{\mathtt{p}}|x-y|}}{|x-y|}\right)\hspace{-0.05cm}\dfrac{(x_i-y_i)(x_j-y_j)}{|x-y|^2}-\left(\dfrac{\mathbf{\mathtt{i}}\kappa_{\mathtt{s}}e^{\mathbf{\mathtt{i}}\kappa_{\mathtt{s}}|x-y|}-\mathbf{\mathtt{i}}\kappa_{\mathtt{p}}e^{\mathbf{\mathtt{i}}\kappa_{\mathtt{p}}|x-y|}}{|x-y|}\right)\right.
\nonumber\\
&&\left.\left(3\dfrac{(x_i-y_i)(x_j-y_j)}{|x-y|^3}-\dfrac{\delta_{ij}}{|x-y|}\right)+\left(\dfrac{e^{\mathbf{\mathtt{i}}\kappa_{\mathtt{s}}|x-y|}-e^{\mathbf{\mathtt{i}}\kappa_{\mathtt{p}}|x-y|}}{|x-y|}\right)\left(3\dfrac{(x_i-y_i)(x_j-y_j)}{|x-y|^4}-\dfrac{\delta_{ij}}{|x-y|^{{2}}}\right)\right],\nonumber
\end{eqnarray} 
and hence for $l=1,2,3$ we deduce
\begin{align}
\dfrac{\partial}{\partial y_{l}}\Gamma^{\omega}_{ij}(x,y)=& \dfrac{\partial}{\partial y_{l}}\left(\frac{\delta_{ij}}{4\pi \mu |x-y|}e^{\mathbf{\mathtt{i}}\kappa_{\mathtt{s}}|x-y|}\right)+\frac{1}{4\pi\omega^2\rho_0}\left[-\dfrac{\partial}{\partial y_{l}}\left(\left(\dfrac{\kappa_{\mathtt{s}}^2e^{\mathbf{\mathtt{i}}\kappa_{\mathtt{s}}|x-y|}-\kappa_{\mathtt{p}}^2e^{\mathbf{\mathtt{i}}\kappa_{\mathtt{p}}|x-y|}}{|x-y|}\right)\dfrac{(x_i-y_i)(x_j-y_j)}{|x-y|^2}\right)\right.\nonumber\\
&-\dfrac{\partial}{\partial y_{l}}\left(\left(\dfrac{\mathbf{\mathtt{i}}\kappa_{\mathtt{s}}e^{\mathbf{\mathtt{i}}\kappa_{\mathtt{s}}|x-y|}-\mathbf{\mathtt{i}}\kappa_{\mathtt{p}}e^{\mathbf{\mathtt{i}}\kappa_{\mathtt{p}}|x-y|}}{|x-y|}\right)\left(3\dfrac{(x_i-y_i)(x_j-y_j)}{|x-y|^3}-\dfrac{\delta_{ij}}{|x-y|}\right)\right)
\nonumber\\
&\left.+\dfrac{\partial}{\partial y_{l}}\left(\left(\dfrac{e^{\mathbf{\mathtt{i}}\kappa_{\mathtt{s}}|x-y|}-e^{\mathbf{\mathtt{i}}\kappa_{\mathtt{p}}|x-y|}}{|x-y|}\right)\left(3\dfrac{(x_i-y_i)(x_j-y_j)}{|x-y|^4}-\dfrac{\delta_{ij}}{|x-y|^{2}}\right)\right)\right]\label{partial-y-m-gamma-omega-first-step-grad-gamma-mp}
\end{align}
\begin{align}
&\hspace{-0.1cm}=\hspace{-0.1cm}\dfrac{\delta_{ij}}{4\pi\mu}\hspace{-0.1cm}\left(\hspace{-0.1cm}\mathbf{\mathtt{i}}\kappa_{\mathtt{s}}-\dfrac{1}{|x-y|}\hspace{-0.05cm}\right)\hspace{-0.1cm}\dfrac{e^{\mathbf{\mathtt{i}}\kappa_{\mathtt{s}}|x-y|}}{|x-y|^2}(y_l-x_l)\hspace{-0.05cm}+\hspace{-0.1cm}\frac{1}{4\pi\omega^2\rho_0}\hspace{-0.1cm}\left[\hspace{-0.05cm}\dfrac{\mathbf{\mathtt{i}}\kappa_{\mathtt{s}}^3e^{\mathbf{\mathtt{i}}\kappa_{\mathtt{s}}|x-y|}-\mathbf{\mathtt{i}}\kappa_{\mathtt{p}}^3e^{\mathbf{\mathtt{i}}\kappa_{\mathtt{p}}|x-y|}}{|x-y|}\dfrac{(x_l-y_l)}{|x-y|}\dfrac{(x_i-y_i)(x_j-y_j)}{|x-y|^2}\right.
\nonumber\\
&-3\left(\hspace{-0.05cm}\dfrac{\kappa_{\mathtt{s}}^2e^{\mathbf{\mathtt{i}}\kappa_{\mathtt{s}}|x-y|}-\kappa_{\mathtt{p}}^2e^{\mathbf{\mathtt{i}}\kappa_{\mathtt{p}}|x-y|}}{|x-y|}\right)\dfrac{(x_l-y_l)(x_j-y_j)(x_i-y_i)}{|x-y|^4}+\left(\hspace{-0.05cm}\dfrac{\kappa_{\mathtt{s}}^2e^{\mathbf{\mathtt{i}}\kappa_{\mathtt{s}}|x-y|}-\kappa_{\mathtt{p}}^2e^{\mathbf{\mathtt{i}}\kappa_{\mathtt{p}}|x-y|}}{|x-y|}\right)\left(\dfrac{\delta_{lj}(x_i-y_i)}{|x-y|^2}+\dfrac{\delta_{li}(x_j-y_j)}{|x-y|^2}\right)
\nonumber\\
&-\left(\dfrac{\kappa_{\mathtt{s}}^2e^{\mathbf{\mathtt{i}}\kappa_{\mathtt{s}}|x-y|}-\kappa_{\mathtt{p}}^2e^{\mathbf{\mathtt{i}}\kappa_{\mathtt{p}}|x-y|}}{|x-y|}\right)\dfrac{(x_l-y_l)}{|x-y|^2}\left(3\dfrac{(x_i-y_i)(x_j-y_j)}{|x-y|^2}-\delta_{ij}\right)
-\left(\dfrac{\mathbf{\mathtt{i}}\kappa_{\mathtt{s}}e^{\mathbf{\mathtt{i}}\kappa_{\mathtt{s}}|x-y|}-\mathbf{\mathtt{i}}\kappa_{\mathtt{p}}e^{\mathbf{\mathtt{i}}\kappa_{\mathtt{p}}|x-y|}}{|x-y|}\right)\dfrac{(x_l-y_l)}{|x-y|^3}\quad
\nonumber\\
&\left(12\dfrac{(x_i-y_i)(x_j-y_j)}{|x-y|^2}-2\delta_{ij}\right)+\left(\dfrac{\mathbf{\mathtt{i}}\kappa_{\mathtt{s}}e^{\mathbf{\mathtt{i}}\kappa_{\mathtt{s}}|x-y|}-\mathbf{\mathtt{i}}\kappa_{\mathtt{p}}e^{\mathbf{\mathtt{i}}\kappa_{\mathtt{p}}|x-y|}}{|x-y|}\right)\left(\dfrac{3\delta_{lj}(x_i-y_i)+3\delta_{li}(x_j-y_j)}{|x-y|^3}\right)\qquad
\nonumber\\
&-\left(\dfrac{\mathbf{\mathtt{i}}\kappa_{\mathtt{s}}e^{\mathbf{\mathtt{i}}\kappa_{\mathtt{s}}|x-y|}-\mathbf{\mathtt{i}}\kappa_{\mathtt{p}}e^{\mathbf{\mathtt{i}}\kappa_{\mathtt{p}}|x-y|}}{|x-y|}\right)\dfrac{(x_l-y_l)}{|x-y|^3}\left(3\dfrac{(x_j-y_j)(x_i-y_i)}{|x-y|^2}-\delta_{ij}\right)+\left(\dfrac{e^{\mathbf{\mathtt{i}}\kappa_{\mathtt{s}}|x-y|}-e^{\mathbf{\mathtt{i}}\kappa_{\mathtt{p}}|x-y|}}{|x-y|}\right)\dfrac{(x_l-y_l)}{|x-y|^4}\qquad
\nonumber\\
&\left.
\left(15\dfrac{(x_i-y_i)(x_j-y_j)}{|x-y|^2}-3\delta_{ij}\right)-\left(\dfrac{e^{\mathbf{\mathtt{i}}\kappa_{\mathtt{s}}|x-y|}-e^{\mathbf{\mathtt{i}}\kappa_{\mathtt{p}}|x-y|}}{|x-y|}\right)\left(\dfrac{3\delta_{lj}(x_i-y_i)+3\delta_{li}(x_j-y_j)}{|x-y|^4}\right)
\label{derivative-partial-gamma-element-xaway}\right].
\end{align}
\item  Thus from \eqref{derivative-partial-gamma-element-xaway}, for $l=1,2,3$, we get 
\begin{align}
|\dfrac{\partial}{\partial y_{l}}\Gamma^{\omega}_{ij}(x,y)|\leq&\left(
\frac{\delta_{ij}}{4\pi\mu}\left(\kappa_{\mathtt{s}}+\frac{1}{|x-y|}\right)\frac{1}{|x-y|}+\frac{1}{4\pi\omega^2\rho_0}\left[\frac{(\kappa_{\mathtt{s}}^3+\kappa_{\mathtt{p}}^3)}{|x-y|}+3\frac{(\kappa_{\mathtt{s}}^2+\kappa_{\mathtt{p}}^2)}{|x-y|^2}+\frac{(\kappa_{\mathtt{s}}^2+\kappa_{\mathtt{p}}^2)}{|x-y|^2}(\delta_{li}+\delta_{lj})\right.\right.
\nonumber\\
&+\frac{(\kappa_{\mathtt{s}}^2+\kappa_{\mathtt{p}}^2)}{|x-y|^2}(3+\delta_{ij})+\frac{(\kappa_{\mathtt{s}}+\kappa_{\mathtt{p}})}{|x-y|^3}\left(12+2\delta_{ij}\right)+3\frac{(\kappa_{\mathtt{s}}+\kappa_{\mathtt{p}})}{|x-y|^3}(\delta_{lj}+\delta_{li})+\frac{(\kappa_{\mathtt{s}}+\kappa_{\mathtt{p}})}{|x-y|^3}\left(3+\delta_{ij}\right)\nonumber\\
&\left.\left.+\frac{2}{|x-y|^4}\left(15+3\delta_{ij}\right)+\frac{6}{|x-y|^4}(\delta_{lj}+\delta_{li})\right]
\right). \label{derivative-partial-gamma-element-xaway-size-magnitude}
\end{align}
\end{itemize}
Since $x$ is away from $\Omega$, and $y\in D_m$, we have $|x-y|> dist(\partial\Omega,\partial D_m)$ and so  $\dfrac{1}{|x-y|}<\dfrac{1}{\mathfrak{d}}$, with $\mathfrak{d}:=dist(\partial \Omega,{\cup_{m=1}^M} \partial D_m)$.
Finally, making use of \eqref{derivative-partial-gamma-element-xaway-size-magnitude}, we get the bound for $ \nabla_y\Gamma^{\omega}_{ij}(x,y)$,  $i,j=1,2,3$ for $y\in D_m$ and $x$ away from $\Omega$, as below;
\begin{eqnarray}
|\nabla_y \Gamma^{\omega}_{ij}(x,y)|\,\leq H_2,\label{grad-Gamma-bounded-x-away-D-y-D-mp}
\end{eqnarray}
where
\begin{align}
H_2:=\sqrt{3}\left(
\frac{1}{4\pi\mu}\left(\kappa_{\mathtt{s}}+\frac{1}{\mathfrak{d}}\right)\frac{1}{\mathfrak{d}}+\frac{1}{4\pi\omega^2\rho_0}\left[\frac{(\kappa_{\mathtt{s}}^3+\kappa_{\mathtt{p}}^3)}{\mathfrak{d}^2}+\dfrac{(\kappa_{\mathtt{s}}^2+\kappa_{\mathtt{p}}^2)}{\mathfrak{d}}\left(3+\frac{6}{\mathfrak{d}}\right)+24\,\frac{(\kappa_{\mathtt{s}}+\kappa_{\mathtt{p}})}{\mathfrak{d}^3}+\frac{48}{\mathfrak{d}^4}\right]\right).\label{def-of-H2}
\end{align}
\end{enumerate}
\end{proof}

%%%%%%%%%%%%%%%%%%%%%%%%%%%%%%%%%%%%%%%%%%%%%%%%%%%%%%%%%%%%%%%%
\begin{lemma} For $x\in D_m$ and $y\in D_q$, $m\neq q$, entries of the fundamental matrix $\Gamma^\omega$ satisfies the following properties; \label{lemma-gamma-in-Dp-Dq-p-ne-q}
\begin{enumerate}
\item \label{Gamma-x,y-bounded-x-in-Dp-y-in-Dq-mp} Each entry of $\Gamma^\omega(x,y)$ is bounded and behaves as $O(\frac{1}{d_{mq}})$. 
 \item Each entry of $\nabla_x\Gamma^{\omega}(x,y) $ is bounded and behaves as $O(\frac{1}{d_{mq}^2})$.
\end{enumerate}
\end{lemma}
\begin{proof}
Let  $x\in D_m$ and $y\in D_q$, $m,q=1,\cdots,M$ with $m\neq q$.
\begin{enumerate}
\item 
 Making use the series representation \eqref{entrywise-FM-mp} of the fundamental matrix $\Gamma^{\omega}$, we can get the 
$ij^{th} $ entry of $\Gamma^{\omega}(x,y)$, for $i,j=1,2,3$, as 
\begin{align}
\Gamma^{\omega}_{ij}(x,y)
=&\dfrac{1}{4\pi \rho_0}\left[\left(\dfrac{1}{c_{\mathtt{s}}^2}+\dfrac{1}{c_{\mathtt{p}}^2}\right)\dfrac{1}{2}|x-y|^{-1}\delta_{ij}-\left(\dfrac{1}{c_{\mathtt{s}}^2}-\dfrac{1}{c_{\mathtt{p}}^2}\right)\dfrac{-1}{2}|x-y|^{-3}(x-y)_i(x-y)_j\right.
\nonumber\\
&\left.+\sum_{n=1}^{\infty}\frac{(\mathbf{\mathtt{i}}\omega)^n\delta_{ij}}{(n+2)n!}\left( \frac{n+1}{c_{\mathtt{s}}^{n+2}}+\frac{1}{c_{\mathtt{p}}^{n+2}}\right)|x-y|^{n-1}
- \sum_{n=1}^{\infty}\frac{\mathbf{(\mathtt{i}\omega)}^n(n-1)}{(n+2)n!}\left(\frac{1}{c_{s}^{n+2}}-\frac{1}{c_{\mathtt{p}}^{n+2}}\right)|x-y|^{n-3}(x-y)_i(x-y)_j\right],\nonumber
\end{align}
 and hence
\begin{align}
|\Gamma^{\omega}_{ij}(x,y)|
\leq & \;\, \dfrac{1}{4\pi\rho_0}\left(\dfrac{1}{c_{\mathtt{s}}^2}+\dfrac{1}{c_{\mathtt{p}}^2}\right)\dfrac{1}{d_{mq}}+\dfrac{1}{4\pi\rho_0}\sum_{n=1}^{\infty}\frac{1}{(n+2)n!}\left[\left( \frac{2n}{c_{\mathtt{s}}^{n+2}}+\frac{n}{c_{p}^{n+2}}\right)|x-y|^{n-1}\right]
\nonumber\\
\leq &\dfrac{1}{4\pi\rho_0}\left(\dfrac{1}{c_{\mathtt{s}}^2}+\dfrac{1}{c_{\mathtt{p}}^2}\right)\dfrac{1}{d_{mq}}+\dfrac{1}{4\pi\rho_0}\left[\dfrac{\kappa_{\mathtt{s}}^3}{\omega^2}\sum_{n=1}^{\infty}\frac{2}{(n+2)(n-1)!} (\kappa_{\mathtt{s}}|x-y|)^{n-1}+\dfrac{\kappa_{\mathtt{p}}^3}{\omega^2}\sum_{n=1}^{\infty}\frac{1}{(n+2)(n-1)!}(\kappa_{\mathtt{p}}|x-y|)^{n-1}\right]
\nonumber\\
\leq &\;\frac{H_3}{d_{mq}}+H_4, \mbox{ for }\frac{1}{2}\max\{\kappa_{\mathtt{s}},\kappa_{\mathtt{p}}\}\,diam(\Omega)<1, \label{Gamma-is-bounded-x-in-Di-y-in-Dj-mp}  
\end{align}
\begin{eqnarray}
\mbox{where, }\;\;H_3:=\dfrac{1}{4\pi\rho_0}\left(\dfrac{1}{c_{\mathtt{s}}^2}+\dfrac{1}{c_{\mathtt{p}}^2}\right)
\;\mbox{ and }\;
H_4:=\dfrac{1}{4\pi\rho_0}\left[\dfrac{2\kappa_{\mathtt{s}}}{c_{\mathtt{s}}^2}\dfrac{1}{1-\frac{1}{2}\kappa_{\mathtt{s}}\,diam(D)}+\dfrac{\kappa_{\mathtt{p}}}{c_{\mathtt{s}}^2} \dfrac{1}{1-\frac{1}{2}\kappa_{\mathtt{p}}\,diam(D)}\right].\label{def-of-H3-H4}
\end{eqnarray}

%%%%%%%%%%%%%%%%%%%%%%%%%%%%%%%%%%%%%%%%%%%%%%%%%%%%%%%%%%%%%%%%%%%%%%%%%%%%%%%%
\item Again by making use the entry wise series representation \eqref{entrywise-FM-mp} of the fundamental matrix $\Gamma^{\omega}$, for $i,j=1,2,3$, we can compute its respective gradient as
\begin{align}
\nabla_x\Gamma_{ij}^{\omega}(x,y)\,=&\frac{1}{4\pi\rho_0}\sum_{n=0}^{\infty}\frac{\mathbf{\mathtt{i}}^n(n-1)}{(n+2)n!}\left( \frac{n+1}{c_{\mathtt{s}}^{n+2}}+\frac{1}{c_{\mathtt{p}}^{n+2}}\right)\omega^n\delta_{ij}|x-y|^{n-3}(x-y)- \frac{1}{4\pi\rho_0}\sum_{n=0}^{\infty}\frac{\mathbf{\mathtt{i}}^n(n-1)}{(n+2)n!}\left(\frac{1}{c_{s}^{n+2}}-\frac{1}{c_{\mathtt{p}}^{n+2}}\right)
\nonumber\\
&\omega^n|x-y|^{n-3}\left[(n-3)|x-y|^{-2}(x-y)(x-y)_i(x-y)_j+\mathtt{e_i}(x-y)_j+(x-y)_i\mathtt{e_j}\right]\nonumber
%%%%%%%%%%%%%%%%%%%%%%%%
%%%%%%%%%%%%%%%%%%%%%%%%%
,\end{align}
and hence
\begin{align}
|\nabla_x\Gamma_{ij}^\omega(x,y)|\,\leq& 
%%%%%%%%%%%%%%%%
%%%%%%%%%%%%%%%%%%%%%%%%%%%
 \dfrac{1}{4\pi\rho_0}\dfrac{1}{\omega^2}\left[3\dfrac{(\kappa_{\mathtt{s}}^2+\kappa_{\mathtt{p}}^2)}{|x-y|^{2}}+\dfrac{1}{8}(6\kappa_{\mathtt{s}}^4+4\kappa_{\mathtt{p}}^4)\right]+\dfrac{1}{4\pi\rho_0}\sum_{n=3}^{\infty}\dfrac{1}{(n+2)(n-2)!}\dfrac{1}{\omega^2} \left[2\,\kappa_{\mathtt{s}}^{n+2}|x-y|^{n-2}+\kappa_{\mathtt{p}}^{n+2}|x-y|^{n-2} \right]
 \nonumber\\
 \leq &\dfrac{1}{4\pi\rho_0\omega^2}\left[3\dfrac{(\kappa_{\mathtt{s}}^2+\kappa_{\mathtt{p}}^2)}{|x-y|^{2}}+\dfrac{1}{4}(3\kappa_{\mathtt{s}}^4+2\kappa_{\mathtt{p}}^4)\right]+\dfrac{1}{4\pi\rho_0\omega^2}\sum_{n=1}^{\infty}\dfrac{1}{(n+4)n!} \left[2\,\kappa_{\mathtt{s}}^{n+4}|x-y|^{n}+\kappa_{\mathtt{p}}^{n+4}|x-y|^{n} \right]
 \nonumber\\
 \leq &\dfrac{3}{4\pi\rho_0\omega^2}\dfrac{(\kappa_{\mathtt{s}}^2+\kappa_{\mathtt{p}}^2)}{|x-y|^{2}}+\dfrac{1}{4\pi\rho_0\omega^2}\left[\dfrac{1}{4}(3\kappa_{\mathtt{s}}^4+2\kappa_{\mathtt{p}}^4)+
 2\kappa_{\mathtt{s}}^4\dfrac{\frac{1}{2}\kappa_{\mathtt{s}}\,diam(\Omega)}{1-\frac{1}{2}\kappa_{\mathtt{s}}\,diam(\Omega)}+\kappa_{\mathtt{p}}^4 \dfrac{\frac{1}{2}\kappa_{\mathtt{p}}\,diam(\Omega)}{1-\frac{1}{2}\kappa_{\mathtt{p}}\,diam(\Omega)}\right]
 \nonumber\\
 \leq & \dfrac{H_5}{d_{mq}^2}+H_6,\label{grad-Gamma-x-Di-y-Dj-homogeneous-background-mp}
\end{align}
where
\begin{align}
H_5:=\dfrac{3}{4\pi\rho_0 } \left(\frac{1}{c_{\mathtt{s}}^2}+\frac{1}{c_{\mathtt{p}}^2}\right)\quad \mbox{and} \quad H_6:= \dfrac{1}{4\pi\rho_0\omega^2}\left[\dfrac{1}{4}(3\kappa_{\mathtt{s}}^4+2\kappa_{\mathtt{p}}^4)+
 2\kappa_{\mathtt{s}}^4\dfrac{\frac{1}{2}\kappa_{\mathtt{s}}\,diam(\Omega)}{1-\frac{1}{2}\kappa_{\mathtt{s}}\,diam(\Omega)}+\kappa_{\mathtt{p}}^4 \dfrac{\frac{1}{2}\kappa_{\mathtt{p}}\,diam(\Omega)}{1-\frac{1}{2}\kappa_{\mathtt{p}}\,diam(\Omega)}\right]\label{def-of-H5-H6}
\end{align}
the last step in \eqref{grad-Gamma-x-Di-y-Dj-homogeneous-background-mp} is because $|x-y|\leq diam(\Omega)$ for $x\in D_m$, $y\in D_q$, $m\neq q$, and $n!\geq 2^{n-1}$, $\forall n$, and with the condition $\frac{1}{2}\max\{\kappa_{\mathtt{p}},\kappa_{\mathtt{s}}\}\,diam(\Omega)<1$.
\end{enumerate}
\end{proof}
%%%%%%%%%%%%%%%%%%%%%%%%%%%%%%%%%%%%%%%%%%%%%%%%%%%%%%%%%%%%%%%%%%%%%%%%%%%%%%%%%%%%
%%%%%%%%%%%%%%%%%%%%%%%%%%%%%%%%%%%%%
%%%%%%%%%%%%%%%%%%%%%%%%%%%%%%

%%%%%%%%%%%%%%%%%%%%%%%%%%%%%%%%%%%%%%%%%%%%%%%%%%%%%%%%%%%%%%%%%%%%%%%%%%%%%%%%%%%%%%

%%%%%%%%%%%%%%%%%%%%%%%%%%%%%%%%%%%%%%%%%%%%%%%%%%%%
\bibliographystyle{abbrv}
%\bibliography{all-references-elastic}

\begin{thebibliography}{10}

\bibitem{A-A-C-K-S-2016}
A.~Alsaedi, F.~Alzahrani, D.~P. Challa, M.~Kirane, and M.~Sini.
\newblock Extraction of the index of refraction by embedding multiple small
  inclusions.
\newblock {\em Inverse Problems}, 32(4):045004, 18, 2016.

\bibitem{alves2002far}
C.~J.~S. Alves and R.~Kress.
\newblock On the  operator in elastic obstacle scattering.
\newblock {\em IMA J. Appl. Math.}, 67(1):1--21, 2002.

\bibitem{AMMARI-BIOMEDICAL-IMAGING}
H.~Ammari.
\newblock {\em An introduction to mathematics of emerging biomedical imaging},
  volume~62 of {\em Math\'{e}matiques \& Applications (Berlin) [Mathematics \&
  Applications]}.
\newblock Springer, Berlin, 2008.

\bibitem{AMMARI-MATHEMATICAL-METHODS-IN-ELASTIC-IMAGING}
H.~Ammari, E.~Bretin, J.~Garnier, H.~Kang, H.~Lee, and A.~Wahab.
\newblock {\em Mathematical methods in elasticity imaging}.
\newblock Princeton Series in Applied Mathematics. Princeton University Press,
  Princeton, NJ, 2015.

 \bibitem{Ammari}
 H.~Ammari, P.~Calmon, and E.~Iakovleva.
 \newblock Direct elastic imaging of a small inclusion.
 \newblock {\em SIAM J. Imaging Sci.}, 1(2):169--187, 2008.

\bibitem{Ammari-Challa-Choudhury-Sini-1}
H.~Ammari, D.~P. Challa, A.~P. Choudhury, and M.~Sini.
\newblock The point-interaction approximation for the fields generated by
  contrasted bubbles at arbitrary fixed frequencies.
\newblock {\em J. Differential Equations}, 267(4):2104--2191, 2019.

\bibitem{Ammari-Challa-Choudhury-Sini-2}
H.~Ammari, D.~P. Challa, A.~P. Choudhury, and M.~Sini.
\newblock The equivalent media generated by bubbles of high contrasts:
  volumetric metamaterials and metasurfaces.
\newblock {\em Multiscale Model. Simul.}, 18(1):240--293, 2020.

\bibitem{Ammari-F-G-L-Z}
H.~Ammari, B.~Fitzpatrick, D.~Gontier, H.~Lee, and H.~Zhang.
\newblock Minnaert resonances for acoustic waves in bubbly media.
\newblock {\em Ann. Inst. H. Poincar\'{e} C Anal. Non Lin\'{e}aire},
  35(7):1975--1998, 2018.

\bibitem{LAYER-POTENTIAL-TECHNIQUES}
H.~Ammari, H.~Kang, and H.~Lee.
\newblock {\em Layer potential techniques in spectral analysis}, volume 153 of
  {\em Mathematical Surveys and Monographs}.
\newblock American Mathematical Society, Providence, RI, 2009.

\bibitem{MR4510105}
H.~Ammari, B.~Li, and J.~Zou.
\newblock Mathematical analysis of electromagnetic scattering by dielectric
  nanoparticles with high refractive indices.
\newblock {\em Trans. Amer. Math. Soc.}, 376(1):39--90, 2023.

\bibitem{B-S-2021}
A.~Bouzekri and M.~Sini.
\newblock Mesoscale approximation of the electromagnetic fields.
\newblock {\em Ann. Henri Poincar\'{e}}, 22(6):1979--2028, 2021.

\bibitem{Caloz}
C.~Caloz and Z.-L. Deck-L{\'e}ger.
\newblock Spacetime metamaterials - \uppercase{P}art \uppercase{I}:
  \uppercase{G}eneral \uppercase{C}oncepts.
\newblock {\em IEEE Transactions on Antennas and Propagation},
  68(3):1569--1582, 2019.

\bibitem{Cao-Ghandriche-Sini-2023}
X.~Cao, A.~Ghandriche, and M.~Sini.
\newblock The electromagnetic waves generated by a cluster of nanoparticles
  with high refractive indices.
\newblock {\em J. Lond. Math. Soc. (2)}, 108(4):1531--1616, 2023.

\bibitem{C-C-S-18}
D.~P. Challa, A.~P. Choudhury, and M.~Sini.
\newblock Mathematical imaging using electric or magnetic nanoparticles as
  contrast agents.
\newblock {\em Inverse Probl. Imaging}, 12(3):573--605, 2018.

\bibitem{challa2023extraction}
D.~P. Challa, D.~Gangadaraiah, and M.~Sini.
\newblock Extraction of the mass density using elastic fields generated by
  injected highly dense small scaled inclusions.
\newblock {\em arXiv:2305.04317}, 2023.

\bibitem{C-S-2014}
D.~P. Challa and M.~Sini.
\newblock On the justification of the {F}oldy-{L}ax approximation for the
  acoustic scattering by small rigid bodies of arbitrary shapes.
\newblock {\em Multiscale Model. Simul.}, 12(1):55--108, 2014.

\bibitem{multiscalechallasini2016}
D.~P. Challa and M.~Sini.
\newblock Multiscale analysis of the acoustic scattering by many scatterers of
  impedance type.
\newblock {\em Z. Angew. Math. Phys.}, 67(3):Art. 58, 31, 2016.

\bibitem{COLTON-IE}
D.~Colton and R.~Kress.
\newblock {\em Integral equation methods in scattering theory}, volume~72 of
  {\em Classics in Applied Mathematics}.
\newblock Society for Industrial and Applied Mathematics (SIAM), Philadelphia,
  PA, 2013.

\bibitem{COLTON-KRESS}
D.~Colton and R.~Kress.
\newblock {\em Inverse acoustic and electromagnetic scattering theory},
  volume~93 of {\em Applied Mathematical Sciences}.
\newblock Springer, New York, third edition, 2013.

\bibitem{D-G-S-2021}
A.~Dabrowski, A.~Ghandriche, and M.~Sini.
\newblock Mathematical analysis of the acoustic imaging modality using bubbles
  as contrast agents at nearly resonating frequencies.
\newblock {\em Inverse Probl. Imaging}, 15(3):555--597, 2021.

\bibitem{LOW-FREQUENCY}
G.~Dassios and R.~Kleinman.
\newblock {\em Low frequency scattering}.
\newblock Oxford Mathematical Monographs. The Clarendon Press, Oxford
  University Press, New York, 2000.

\bibitem{MR4003541}
Y.~Deng, H.~Li, and H.~Liu.
\newblock On spectral properties of {N}euman-{P}oincar\'{e} operator and
  plasmonic resonances in 3{D} elastostatics.
\newblock {\em J. Spectr. Theory}, 9(3):767--789, 2019.

  \bibitem{MR4083345}
  Y.~Deng, H.~Li, and H.~Liu.
  \newblock Analysis of surface polariton resonance for nanoparticles in elastic   system.
 \newblock {\em SIAM J. Math. Anal.}, 52(2):1786--1805, 2020.

  \bibitem{MR4126865}
  Y.~Deng, H.~Li, and H.~Liu.
  \newblock Spectral properties of {N}eumann-{P}oincar\'{e} operator and
    anomalous localized resonance in elasticity beyond quasi-static limit.
  \newblock {\em J. Elasticity}, 140(2):213--242, 2020.

\bibitem{G-S-2022}
A.~Ghandriche and M.~Sini.
\newblock An introduction to the mathematics of the imaging modalities using
 small scaled contrast agents.
\newblock {\em ICCM Not.}, 10(1):28--43, 2022.

\bibitem{Ghandriche-Sini}
A.~Ghandriche and M.~Sini.
\newblock Photo-acoustic inversion using plasmonic contrast agents: the full
  {M}axwell model.
\newblock {\em J. Differential Equations}, 341:1--78, 2022.

\bibitem{ghandriche2023calderon}
A.~Ghandriche and M.~Sini.
\newblock The calderon problem revisited: Reconstruction with resonant
  perturbations.
\newblock {\em arXiv:2307.12055}, 2023.

\bibitem{G-S-2023}
A.~Ghandriche and M.~Sini.
\newblock Simultaneous reconstruction of optical and acoustical properties in
  photoacoustic imaging using plasmonics.
\newblock {\em SIAM J. Appl. Math.}, 83(4):1738--1765, 2023.

\bibitem{Jikov-book}
V.~V. Jikov, S.~M. Kozlov, and O.~A. Oleinik.
\newblock {\em Homogenization of differential operators and integral
  functionals}.
\newblock Springer Science \& Business Media, 2012.

\bibitem{KESAVAN}
S.~Kesavan.
\newblock {\em Topics in functional analysis and applications}.
\newblock John Wiley \& Sons, Inc., New York, 1989.

\bibitem{Acoustics}
L.~E. Kinsler, A.~R. Frey, A.~B. Coppens, and J.~V. Sanders.
\newblock {\em Fundamentals of acoustics}.
\newblock John wiley \& sons, 2000.

\bibitem{kupradze1967potential}
V.~D. Kupradze.
\newblock Potential methods in the theory of elasticity.
\newblock Technical report, 1965.

\bibitem{kupradze1980three}
V.~D. Kupradze, T.~G. Gegelia, M.~O. Bashele\u{\i}shvili, and T.~V.
  Burchuladze.
\newblock {\em Three-dimensional problems of the mathematical theory of
  elasticity and thermoelasticity}, volume~25 of {\em North-Holland Series in
  Applied Mathematics and Mechanics}.
\newblock North-Holland Publishing Co., Amsterdam-New York, \uppercase{R}ussian
  edition, 1979.


\bibitem{MR4362220}
H.~Li, H.~Liu, and J.~Zou.
\newblock Minnaert resonances for bubbles in soft elastic materials.
\newblock {\em SIAM J. Appl. Math.}, 82(1):119--141, 2022.

\bibitem{MR4573428}
H.~Li, H.~Liu, and J.~Zou.
\newblock Elastodynamical resonances and cloaking of negative material
  structures beyond quasistatic approximation.
\newblock {\em Stud. Appl. Math.}, 150(3):716--754, 2023.

\bibitem{Maugin}
G.~A. Maugin.
\newblock {\em Material Inhomogeneities in Elasticity}.
\newblock Chapman and Hall, 1993.

\bibitem{Mazya-book}
V.~Maz'ya, A.~Movchan, and M.~Nieves.
\newblock {\em Green's kernels and meso-scale approximations in perforated
  domains}, volume 2077 of {\em Lecture Notes in Mathematics}.
\newblock Springer, Heidelberg, 2013.

\bibitem{Ntziachristos}
V.~Ntziachristos.
\newblock Going deeper than microscopy: the optical imaging frontier in
  biology.
\newblock {\em Nature methods}, 7(8):603--614, 2010.

\bibitem{Osipov-Tretyakov}
A.~V. Osipov and S.~A. Tretyakov.
\newblock {\em Modern electromagnetic scattering theory with applications}.
\newblock Chichester, UK: John Wiley and Sons, 2017.

\bibitem{Heating}
J.~Pearce, A.~Giustini, R.~Stigliano, and P.~Jack~Hoopes.
\newblock Magnetic heating of nanoparticles: the importance of particle
  clustering to achieve therapeutic temperatures.
\newblock {\em Journal of nanotechnology in engineering and medicine},
  4(1):011005, 2013.

\bibitem{MRAMAG-book1}
A.~G. Ramm.
\newblock {\em Wave scattering by small bodies of arbitrary shapes}.
\newblock World Scientific Publishing Co. Pte. Ltd., Hackensack, NJ, 2005.

\bibitem{Serdyukov-et-al}
A.~Serdyukov, I.~Semchenko, S.~Tretyakov, and A.~Sihvola.
\newblock {\em Electromagnetics of bi-anisotropic materials: Theory and
  applications}.
\newblock Gordon and Breach Science Publishers, Amsterdam, 2001.

\end{thebibliography}

%\bibliographystyle{IEEEtr}
%%%%%%%%%%%%%%%%%%%%%%%%%%%%%%%%%%%%%%%%
                                       
\end{document}